\documentclass[reqno]{amsart}
\usepackage[left=1in,right=1in,top=1in,bottom=1in]{geometry}
\usepackage{microtype}
\usepackage{hyperref}

\usepackage{amsmath}             
\usepackage{amsfonts}             
\usepackage{amsthm}               
\usepackage{amsbsy}
\usepackage{amssymb}
\usepackage{amsthm}
\usepackage[nobysame]{amsrefs}
\usepackage{mathrsfs}

\usepackage{enumerate}

\usepackage{calligra}
\usepackage{stmaryrd}
\usepackage{relsize}

\usepackage{tikz}
\usetikzlibrary{shapes,snakes,calc,arrows}
\usetikzlibrary{decorations.pathreplacing,shapes.geometric}
\usetikzlibrary{calc,positioning}
\tikzset{Box/.style={very thick, rounded corners}}
\tikzset{marked/.style={star, star point height = .75mm, star points =5, fill=black,minimum size=2mm, inner sep=0mm} }
\tikzset{verythickline/.style = {line width=7pt}}
\tikzset{thickline/.style = {line width=5pt}}
\tikzset{medthick/.style = {line width=3pt}}
\tikzset{med/.style = {line width=2pt}}
\tikzset{count/.style = {fill=white,circle,draw,thin, inner sep=2pt}}
\tikzset{rcount/.style = {fill=white,rectangle,draw,thin,inner sep=2pt, rounded corners}}
\tikzset{cpr/.style = {draw,fill=white,rectangle,thin, rounded corners}}

\theoremstyle{plain}
	\newtheorem{theorem*}{Theorem}
	\newtheorem{theorem}{Theorem}[section]

	\newtheorem{proposition}[theorem]{Proposition}
	\newtheorem{lemma}[theorem]{Lemma}
	
	\newtheorem{corollary}[theorem]{Corollary}

\theoremstyle{definition}	\newtheorem{definition}[theorem]{Definition}
	
	\newtheorem{example}[theorem]{Example}

\theoremstyle{remark}
	\newtheorem{remark}[theorem]{Remark}

\DeclareMathOperator{\NC}{NC}
\DeclareMathOperator{\BNC}{BNC}
\DeclareMathOperator{\tr}{tr}
\DeclareMathOperator{\alg}{alg}

\newcommand{\op}{{\text{op}}}

\newcommand{\qqand}{\qquad \text{and}\qquad}
\newcommand{\qand}{\quad \text{and}\quad}

\newcommand{\C}{\mathbb{C}}
\newcommand{\bC}{\mathbb{C}}
\newcommand{\bR}{\mathbb{R}}

\newcommand{\bN}{\mathbb{N}}

\renewcommand{\H}{\mathcal{H}}
\newcommand{\F}{\mathcal{F}}

\newcommand{\fM}{\mathfrak{M}}

\renewcommand{\P}{\mathcal{P}}

\newcommand{\vX}{\mathbf{X}}
\newcommand{\vY}{\mathbf{Y}}
\newcommand{\vZ}{\mathbf{Z}}
\newcommand{\vS}{\mathbf{S}}

\newcommand{\vD}{\mathbf{D}}
\newcommand{\vC}{\mathbf{C}}

\newcommand{\A}{\mathcal{A}}

\newcommand{\sx}{s_{\chi}}

\newcommand{\sxv}{s_{{\chi}|_V}}

\newcommand{\xv}{\chi|_V}
\newcommand{\pleq}{\preceq_{\chi}}
\newcommand{\ple}{\prec_{\chi}}
\newcommand{\xh}{\widehat{\chi}}

\newcommand{\oxh}{\widehat{0_{\chi}}}

\renewcommand{\phi}{\varphi} %%%%% I bother to write \varphi

\newcounter{my_enumerate_counter}

\title{Bi-Free Entropy with Respect to Completely Positive Maps}

\author{Georgios Katsimpas}
\address{Department of Mathematics and Statistics, York University, 4700 Keele Street, Toronto, Ontario, M3J 1P3, Canada}
\email{gkats@mathstat.yorku.ca}

\author{Paul Skoufranis}
\address{Department of Mathematics and Statistics, York University, 4700 Keele Street, Toronto, Ontario, M3J 1P3, Canada}
\email{pskoufra@yorku.ca}

\date{\today}
\thanks{The funding for the first author was partially supported and the research of the second author was supported in part by NSERC (Canada) grant RGPIN-2017-05711.}

\begin{document}

\begin{abstract}
In this paper, a notion of non-microstate bi-free entropy with respect to completely positive maps is constructed thereby extending the notions of non-microstate bi-free entropy and free entropy with respect to a completely positive map.  By extending the operator-valued bi-free structures to allow for more analytical arguments, a notion of conjugate variables is constructed using both moment and cumulant expressions.  The notions of free Fisher information and entropy are then extended to this setting and used to show minima of the Fisher information and maxima of the non-microstate bi-free entropy at bi-R-diagonal elements.
\end{abstract}

\maketitle

\section{Introduction}
\label{sctn:introduction}

Free entropy was introduced in a series of papers by Voiculescu including \cites{V1998-2, V1999} that cemented the foundations of free probability and its applications to operator algebras.  Of note is the non-microstate approach of \cite{V1998-2} that generalized the notions of Fisher information and entropy to the non-commutative random variables studied in free probability by using a conjugate variable system and free Brownian motions.  These ideas were further extended to the operator-valued setting by Shlyakhtenko in \cite{S1998} by modifying the conjugate variable formulae to involve a completely positive map on the algebra of amalgamation.  One immediate application was \cite{S1998}*{Proposition 7.14} that obtained a formula for the Jones index of a subfactor.  Furthermore, free entropy with respect to a completely positive map was  essential to the work in \cite{NSS1999} which demonstrated that minimal values for the free Fisher information and maximal values for the non-microstate free entropy existed and were obtained at R-diagonal elements.

More recently in \cite{V2014}, Voiculescu extended the notion of free independence to simultaneously study the left and right actions of algebras on reduced free product spaces.  In particular, this permits a notion of independence, called bi-free independence, that contains both free and classical independence (see \cite{S2016-1}) and a free probability construction to simultaneously study both a von Neumann algebra and its commutant (see Example \ref{exam:factors}).  Significant effort has gone into enhancing results from free probability to the bi-free setting and examining potential applications.  In terms of entropy, Charlesworth and the second author recently developed \cites{CS2020, CS2019}, thereby extending both the microstate and non-microstate free entropies to the bi-free setting.

The purpose of this paper is to extend the notion of non-microstate bi-free entropy to incorporate the existence of a completely positive map and examine applications of said theory.  In particular, our main applications are Theorem \ref{thm:minimizing-fisher-info} which examines the minimal value of the bi-free Fisher information for collections of pairs of operators with similarities in their distributions, and the following theorem which cumulates the first author's work on bi-R-diagonal elements from \cite{K2019}.

\begin{theorem*}
Let $(\A, \varphi)$ be a C$^*$-non-commutative probability space and let $x,y \in \A$ be such that $x^*x$ and $xx^*$ have the same distribution with respect to $\varphi$ and $y^*y$ and $yy^*$ have the same distribution with respect to $\varphi$.  With $\tau_2 : \A \otimes M_2(\bC) \otimes M_2(\bC)^{\op} \to \bC$ defined by
\[
\tau_2(T \otimes b_1 \otimes b_2) = \varphi(T) \tr_2(b_1b_2)
\]
and
\[
X = x \otimes E_{1,2} \otimes I_2 + x^* \otimes E_{2,1} \otimes I_2 \qand Y = y \otimes I_2 \otimes E_{1,2} + y^* \otimes I_2 \otimes E_{2,1},
\]
we have that
\[
\chi^*(\{x,x^*\} \sqcup \{y,y^*\}) \leq 2 \chi^*(X \sqcup Y)
\]
and equality holds whenever the pair $(x, y)$ is bi-R-diagonal and alternating adjoint flipping.
\end{theorem*}

This paper is structured as follows.  After reviewing some preliminaries and notation pertaining to bi-free probability in Section \ref{sec:prelim}, Section \ref{sec:Analytical-B-B-NCPS} will extend the structures used in operator-valued bi-free probability.  This is necessary as expectations in operator-valued bi-free probability need not be positive and thus to perform analytical computations additional structures are required.  These structures occur and are modelled based on the left and right actions of a II$_1$ factor on its $L_2$-space (see Example \ref{exam:vonNeumannAnalytical}).  By adding a tracial state on the algebra of amalgamation that satisfies certain compatibility conditions, the appropriate $L_2$-spaces can be constructed and used to study operator-valued bi-free probability.  It is these structures that the authors believe will be vital to future applications of operator-valued bi-free probability.

In Section \ref{sec:Bi-Multi} the operator-valued bi-free cumulant functions are extended to allow for the last entry to be an element of the corresponding $L_2$-space.  This is thematic in bi-free probability where the last entry corresponds to the location where left and right operators mix and therefore it is natural to extend the bi-free moment and cumulant functions to have a mixture of left and right operators in the last entry.  These analytical extensions of the operator-valued bi-free cumulants are shown to have the appropriate bi-multiplicative properties by methods similar to \cite{CNS2015-1}.

In Section \ref{sec:Conj} the bi-free conjugate variables with respect to completely positive maps are defined via moment relations.  It is then shown that these moment relations transfer to cumulant relations using the results of Section \ref{sec:Bi-Multi} and thus the natural properties of free conjugate variables extend to this setting.  One technical detail in developing the bi-free entropy with respect to a completely positive map is to show that if one perturbs operators by  operator-valued bi-free central limit distributions, then the resulting operators have bi-free conjugate variables.  Thus, it is necessary to show that one can always add in the appropriate operator-valued bi-free central limit distributions into the structures from Section \ref{sec:Analytical-B-B-NCPS} and remain within that context.  Results along these line sufficient for this paper are developed in Section \ref{sec:semis}.

In Sections \ref{sec:Fisher} and \ref{sec:Entropy} the bi-free Fisher information and entropy with respect to a completely positive map are defined and shown to have the desired properties from \cites{V1998-2, S1998, CS2020} with very similar proofs.  Sections \ref{sec:minimize-Fisher} and \ref{sec:Max-Entropy} extend the techniques from \cite{NSS1999} taking into account the differences in bi-free probability to obtain the minimal value of the bi-free Fisher information of a collection of pairs of operators with similarities in their distributions and the above theorem.  Finally, Section \ref{sec:Other} outlines some other results that immediately extend from \cite{NSS1999} to the bi-free setting using the results and techniques from this paper.

\section{Preliminaries}
\label{sec:prelim}

In this section we will remind the reader on the basic combinatorial and operator-valued structures that have been used in previous papers on bi-free independence.  For a more in-depth reminder of these concepts, we refer the readers to the original papers \cites{CNS2015-1, CNS2015-2}.

\subsection*{Combinatorics on the Lattice of Bi-Non-Crossing Partitions}

For $n\in\bN$, the collection of all partitions on $\{1,\ldots,n\}$ is denoted by $\mathcal{P}(n)$.  The elements of any $\pi\in \mathcal{P}(n)$ are called the \emph{blocks} of $\pi$. A partial ordered on $\P(n)$ is defined via refinement where for $\pi,\sigma\in\mathcal{P}(n)$ we write $\pi\leq\sigma$ if every block of $\pi$ is contained in a block of $\sigma$.  The maximal element of $\mathcal{P}(n)$ with respect to this partial order is the partition consisting of one block and is denoted by $1_n$, while the minimal element is the partition consisting of $n$-blocks and denoted by $0_n$. Note that $\mathcal{P}(n)$ becomes a lattice under this partial ordering.  For $\pi,\sigma\in\mathcal{P}(n)$ the \emph{join} of $\pi$ and $\sigma$, denoted $\pi\vee \sigma$, is the minimum element of the non-empty set $\{\upsilon\in\mathcal{P}(n)\, \mid \, \upsilon\geq\pi,\sigma\}$. A partition $\pi$ on $\{1,\ldots, n\}$ is said to be \emph{non-crossing} if whenever $V,W$ are blocks of $\pi$ and $v_1,v_2\in V , w_1,w_2\in W$ are such that
\[
v_1 < w_1 < v_2 < w_2,
\]
then $V=W$.   The lattice of non-crossing partitions on $\{1,\ldots, n\}$ is denoted $\NC(n)$.

In the bi-free setting, all operators are implicitly imbued with a direction; either left or right.  Given a sequence of $n$ operators, we will use a map $\chi \in {\{\ell,r\}}^n$ to distinguish whether the $k^{\text{th}}$ operator is a left or a right operator.  Such a map automatically gives rise to a permutation $\sx$ on $\{1,\ldots,n\}$ defined as follows: if ${\chi}^{-1}(\{\ell\}) = \{i_1<\ldots <i_p\}$ and ${\chi}^{-1}(\{r\}) = \{j_1<\ldots <j_{n-p}\}$, then 
\[   
\sx(k) = 
     \begin{cases}
       i_k &\text{if }  k\leq p\\
       j_{n+1-k} &\text{if }  k>p \\
     \end{cases}.
\]

From a combinatorial point of view, the main difference between free and bi-free probability arise from dealing with $\sx$.   The permutation $\sx$ naturally induces a total order $\prec_\chi$ on $\{1,\ldots,n\}$ henceforth referred to as the \textit{$\chi $-order} as follows:
\[
i \prec_\chi j \iff \sx^{-1}(i)<  \sx^{-1}(j).
\]
Instead of reading $\{1,\ldots,n\}$ in the traditional order, this corresponds to first reading the elements of $\{1,\ldots,n\}$ labelled ``$\ell$'' in increasing order followed by reading the elements labelled ``$r$'' in decreasing order. Note that if $V$ is any non-empty subset of $\{1,\ldots,n\}$, the map $\xv$ naturally gives rise to a map $\sxv$, which should be thought of as a permutation on $\{1,\ldots,|V|\}$.

\begin{definition}\label{BNC}
Let $n\in\bN$ and $ \chi\in {\{\ell,r\}}^n$. A partition $\pi \in\mathcal{P}(n)$ is called \textit{bi-non-crossing with respect to $\chi$} if the partition $\sx^{-1}\circ \pi$ (i.e. the partition obtained by applying the permutation $\sx^{-1}$ to each entry of every block of $\pi$) is non-crossing. Equivalently, $\pi$ is bi-non-crossing with respect to $\chi$ if whenever $V,W$ are blocks of $\pi$ and $v_1,v_2\in V , w_1,w_2\in W$ are such that
\[
v_1 \prec_\chi w_1 \prec_\chi v_2 \prec_\chi w_2,
\]
then $V=W$.  The collection of bi-non-crossing partitions with respect to $\chi$ is denoted by $\BNC(\chi)$.  It is clear that
\[
\BNC(\chi) = \left\{\pi\in\mathcal{P}(n)\, \mid \, \sx^{-1}\circ\pi\in \NC(n)\right\} = \left\{\sx\circ \sigma\, \mid \, \sigma \in \NC(n) \right\}.
\]
\end{definition}

In the context when the map $\chi$ is clear, we will refer to an element of $\BNC(\chi)$ simply as bi-non-crossing.  To each partition $\pi \in BNC(\chi)$ we can associate a ``bi-non-crossing diagram'' (with respect to $\chi$) by placing nodes along two vertical lines, labelled $1$ to $n$ from top to bottom, such that the nodes on the left line correspond to those values for which $\chi(k) = \ell$ (similarly for the right), and connecting those nodes which are in the same block of $\pi$ in a non-crossing manner.  In particular, a partition $\pi \in \P(n)$ is in $\BNC(\chi)$ if and only if it can be drawn in this non-crossing way.

\begin{example}
If $\chi\in {\{\ell,r\}}^6$ is such that $\chi^{-1}(\{\ell\}) = \{1,2,3,6\}$ and $\chi^{-1}(\{r\}) = \{4,5\}$, then 
\[
(\sx(1),\ldots,\sx(6)) = (1,2,3,6,5,4)
\]
and the partition given by 
\[
\pi = \{\{1,4\}, \{2,5\}, \{3,6\}\}
\]
is bi-non-crossing with respect to $\chi$ even though $\pi\notin \NC(6)$. This may also be seen via the following diagrams:
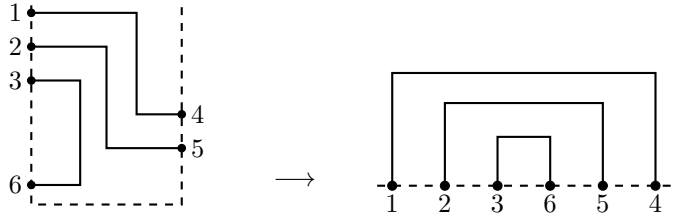
\begin{align*}
\begin{tikzpicture}[baseline]
			\draw[thick,dashed] (-.5,2.4) -- (-.5,-.25) -- (1.5,-.25) -- (1.5,2.4);
\node[left] at (-.5,2.3) {1};
			\draw[black,fill=black] (-.5,2.3) circle (0.05);
\node[left] at (-.5,1.85) {2};
			\draw[black,fill=black] (-.5,1.85) circle (0.05);
\node[left] at (-.5,1.4) {3};
			\draw[black,fill=black] (-.5,1.4) circle (0.05);
\node[right] at (1.5,0.95) {4};
			\draw[black,fill=black] (1.5,0.95) circle (0.05);
\node[right] at (1.5,0.5) {5};
			\draw[black,fill=black] (1.5,0.5) circle (0.05);
\node[left] at (-.5,0.01) {6};
			\draw[black,fill=black] (-.5,0.01) circle (0.05);
			\draw[thick, black] (-.5,2.3) -- (0.9,2.3) -- (0.9,0.95) -- (1.5,0.95);
			\draw[thick, black] (-.5,1.85) -- (0.5,1.85) -- (0.5,0.5) -- (1.5,0.5);
			\draw[thick, black] (-.5,1.4) -- (0.15,1.4) -- (0.15, 0.01) -- (-.5,0.01);
	\end{tikzpicture}
	\qquad\longrightarrow	\qquad
\begin{tikzpicture}[baseline]
			\draw[thick,dashed] (5,0) -- (9,0);
\filldraw (5.2,0) circle (0.06) node[anchor=north] {1};
\filldraw (5.9,0) circle (0.06) node[anchor=north] {2};
\filldraw (6.6,0) circle (0.06) node[anchor=north] {3};
\filldraw (7.3,0) circle (0.06) node[anchor=north] {6};
\filldraw (8,0) circle (0.06) node[anchor=north] {5};
\filldraw (8.7,0) circle (0.06) node[anchor=north] {4};	
\draw[thick, black] (5.2,0) -- (5.2,1.5) -- (8.7,1.5) -- (8.7,0);	
\draw[thick, black] (5.9,0) -- (5.9,1.1) -- (8,1.1) -- (8,0);	
\draw[thick, black] (6.6,0) -- (6.6,0.65) -- (7.3,0.65) -- (7.3,0);	
			\end{tikzpicture}
			\qquad
	\end{align*}
\end{example}

The set of bi-non-crossing partitions with respect to a map $\chi\in {\{\ell,r\}}^n$ automatically inherits the lattice structure from $\mathcal{P}(n)$ via the partial order of refinement. The minimal and maximal elements of $\BNC(\chi)$ will be denoted  by $0_{\chi}$ and $1_{\chi}$ respectively, and note that $0_{\chi}=\sx(0_n)=0_n$ and $1_{\chi}=\sx(1_n)=1_n$.  For $\emptyset\neq V\subseteq\{1,\ldots,n\}$, we denote by $\min_\leq V$ and $\min_{\pleq}V$ the minimum  element of $V$ with respect to the natural order and the ${\chi}$-order of $\{1,\ldots,n\}$ respectively, with similar notation used for the maximum elements.  For $\pi,\sigma\in\BNC(\chi)$ with $\sigma\leq\pi$ we will denote by $[\sigma,\pi]$ the interval with respect to the partial order of refinement.

\begin{definition}\label{defbncmobius}
The \emph{bi-non-crossing M{\"o}bius function} is the map
\[
\mu_{\BNC}: \bigcup_{n\in\bN} \bigcup_{\chi\in {\{\ell,r\}}^n} \BNC(\chi)\times\BNC(\chi)\to\mathbb{Z}
\]
defined recursively by
\[
\sum_{\substack{\nu\in\BNC(n)\\
\pi\leq\tau\leq\sigma\\}}\mu_{\BNC}(\pi,\nu) = \sum_{\substack{\tau\in\BNC(n)\\
\pi\leq\tau\leq \sigma\\}}\mu_{\BNC}(\nu,\sigma) =
     \begin{cases}
       1 &\text{if }  \pi=\sigma\\
        0 &\text{if } \pi <\sigma\\
     \end{cases}
\]
whenever $\pi\leq\sigma$, while taking the zero value otherwise.
\end{definition}

The connection between the bi-non-crossing M{\"o}bius function and the M{\"o}bius function on the lattice of non-crossing partitions $\mu_{\NC}$ is given by the formula
\[
\mu_{\BNC}(\pi,\sigma) = \mu_{\NC}(\sx^{-1}\circ\pi, \sx^{-1}\circ\sigma)
\]
for all $\pi\leq\sigma\in\BNC(\chi)$.  Hence $\mu_{\BNC}$ inherits permuted analogues of the multiplicative properties of $\mu_{\NC}$ (see \cite{CNS2015-2}*{Section 3}).  In particular, if $n\in\bN$, if $\chi\in\{\ell,r\}^n$, if $\pi,\sigma\in\BNC(\chi)$ such that $\pi\leq\sigma$, and if $V_1,\ldots,V_m$ are unions of blocks of $\pi$ that partition $\{1,\ldots,n\}$, then the natural map
\begin{align*}
[\sigma,\pi] & \longrightarrow \prod_{k=1}^m [\sigma|_{V_k}, \pi|_{V_k}]
\\ \tau & \mapsto\qquad {({\tau |}_{V_k})}_{k=1}^m
\end{align*}
is a bijection and 
\[
\mu_{\BNC}(\sigma,\pi) = \mu_{\BNC}({\sigma |}_{V_1}, \pi|_{V_1}) \cdots  \mu_{\BNC}({\sigma |}_{V_m},\pi|_{V_m}).
\]

\subsection*{B-B-Non-Commutative Probability Spaces and Bi-Freeness}

To study operator-valued bi-free independence, certain structures are required.  Thus we shall remind the reader of the general structures as developed in \cite{CNS2015-1} and refer the reader there for more details.

\begin{definition}\label{BBncpsdef}
Let $B$ be a unital $*$-algebra and let $B^{\op}$ denote the unital $*$-algebra with the same elements as $B$ with the opposite multiplication.   A \textit{B-B-non-commutative probability space} consists of a triple $(A,E,\varepsilon)$ where $A$ is a unital $*$-algebra, $\varepsilon: B\otimes B^{\op}\to A$ is a unital $*$-homomorphism such that the restrictions ${\varepsilon |}_{B\otimes 1_B}$ and ${\varepsilon |}_{1_B\otimes B^{\op}}$ are both injective, and $E:A\to B$ is a unital linear map that such that
\[
E(\varepsilon(b_1\otimes b_2) a) = b_1 E(a) b_2
\qqand E(a \varepsilon(b\otimes 1_B)) = E(a \varepsilon(1_B\otimes b)),
\]
for all $b,b_1,b_2\in B$ and $a\in A$. In addition, consider the unital $*$-subalgebras $A_\ell$ and $A_r$ of $A$ given by
\[
A_\ell = \{a\in A\, \mid \, a \varepsilon (1_B\otimes b) = \varepsilon(1_B\otimes b) a \text{ for all } b\in B \}
\]
and
\[
A_r = \{a\in A\, \mid \, a \varepsilon (b\otimes 1_B) = \varepsilon(b\otimes 1_B) a\text{ for all }  b\in B  \}.
\]
We call $A_\ell$ and $A_r$ the \textit{left and right algebras of $A$} respectively.
\end{definition}
 
Note one can always assume that a $B$-$B$-non-commutative probability space is generated as a $*$-algebra by $A_\ell$ and $A_r$.

\begin{example}
\label{exam:tensor-up}
Let $\A$ and $B$ be unital $*$-algebras and let $\varphi : \A \to \bC$ be a unital, linear map.  If $A = \A \otimes B \otimes B^{\op}$, if $\varepsilon  : B \otimes B^{\op} \to A$ is defined by $\varepsilon(b_1 \otimes b_2) = 1_\A \otimes b_1 \otimes b_2$ for all $b_1, b_2 \in B$, and $E : A \to B$ is defined by
\[
E(a \otimes b_1 \otimes b_2) = \varphi(a) b_1b_2
\] 
for all $a \in A$ and $b_1, b_2 \in B$, then $(A, E, \varepsilon)$ is a $B$-$B$-non-commutative probability space.  Indeed, clearly $\varepsilon$ is a unital injective $*$-homomorphism.  Furthermore, note for all $Z\in\A$ and $b,b_1,b_2,b_3,b_4\in B$ that
\[
E((1_\A\otimes b_1\otimes b_2)(Z\otimes b_3\otimes b_4)) = \phi(Z)b_1b_3b_4b_2 = b_1 E(Z\otimes b_3\otimes b_4)b_2
\] 
and  
\[
E((Z\otimes b_1\otimes b_2)(1_\A\otimes b\otimes 1_B)) = \phi(Z)b_1bb_2 = E((Z\otimes b_1\otimes b_2)(1_\A \otimes 1_B\otimes b)).
\]
Hence $E$ satisfies the required properties.

For future use, notice that
\[
\A \otimes B \otimes 1_B \subseteq \A_\ell \qqand \A \otimes 1_B \otimes B^{\op} \subseteq \A_r.
\]
Moreover, in the case $B = \bC$, $(A, E, \varepsilon)$ efficiently reduces down to $(\A, \varphi)$; the usual notion of a non-commutative probability space.
\end{example}

\begin{example}
\label{exam:factors}
Let $\fM$ be a finite von Neumann algebra with a tracial state $\tau : \fM \to \bC$ and let $L_2(\fM, \tau)$ be the GNS Hilbert space generated by $(\fM, \tau)$.  For $T \in \fM$, let $L_T$ denote the left action of $T$ on $L_2(\fM, \tau)$, and let $R_T$ denote the right action of $T$ on $L_2(\fM, \tau)$.  Furthermore, let $A$ be the algebra generated by $\{L_T, R_T \, \mid \, T \in \fM\}$.

Let $B$ be a unital von Neumann subalgebra of $\fM$ and let $E_B : \fM \to B$ be the conditional expectation of $\fM$ onto $B$.  Recall that if $P : L_2(\fM, \tau) \to L_2(B, \tau)$ is the orthogonal projection of $L_2(\fM, \tau)$ onto $L_2(B, \tau)$, then $E_B(Z) = P(Z1_\fM)$ for all $Z \in \fM$.  

Define $\varepsilon : B \otimes B^{\op} \to A$ by $\varepsilon(b_1 \otimes b_2) = L_{b_1} R_{b_2}$ and define $E : A \to B$ by
\[
E(Z) = P(Z 1_{\fM})
\]
for all $Z \in A$.  Elementary von Neumann algebra theory implies that the range of $E$ is indeed contained in $B$.  To see that $(A, E, \varepsilon)$ is a $B$-$B$-non-commutative probability space, first note that $\varepsilon$ is clearly a unital $*$-homomorphism that is injective when restricted to $B \otimes 1_B$ and when restricted to $1_B \otimes B^\op$.  Moreover, note  for all $Z\in A$ and $b,b_1,b_2\in B$ that
\[
E(L_{b_1} R_{b_2}Z ) = P(b_1(Z1_{\fM})b_2) = b_1P(Z1_{\fM})b_2 = b_1 E(Z)b_2
\] 
and  
\[
E(TL_b) = P(TL_b 1_\fM)= P(Tb) = P(TR_b 1_\fM) = E(TR_b).
\]
Hence $E$ satisfies the required properties.
\end{example}

The map $\varepsilon:B\otimes B^{\op}\to A$ encodes the left and right elements of $B$ in $A$.  For notational purposes, for each $b \in B$ we will denote $\varepsilon(b\otimes 1_B)$ and $\varepsilon(1_B\otimes b)$ by $L_b$ and $R_b$ respectively and we denote 
\[
B_\ell = \varepsilon(B\otimes 1_B)=\{L_b \, \mid \, b\in B\}
\qand
B_r = \varepsilon(1_B\otimes B^{\op})=\{R_b  \, \mid \, b\in B\}.
\]
To examine bi-free independence with amalgamation over $B$, it is necessary that left operators are contained in $A_\ell$ (i.e. commute with the right copy of $B$) and right operators are contained in $A_r$ (i.e. commute with the left copy of $B$).

\begin{definition}[\cite{CNS2015-1}]  \label{bifreedef}
Let $(A,E,\varepsilon)$ be a $B$-$B$-non-commutative probability space.
\begin{enumerate}[(i)]
\item A \textit{pair of $B$-algebras} is a pair $(C,D)$ consisting of unital subalgebras of $A$ such that
\[
B_\ell \subseteq C\subseteq A_\ell \qand B_r \subseteq D\subseteq A_r.
\]
\item A family ${\{(C_k,D_k)\}}_{k\in K}$ of pairs of $B$-algebras in $A$ is called \textit{bi-free with amalgamation over $B$} if there exist $B$-$B$-bimodules with specified $B$-vector states ${\{(\mathcal{X}_k,\overset{\circ}{\mathcal{X}}_k,p_k)\}}_{k\in K}$ and unital homomorphisms
\[
l_k:C_k\to \mathcal{L}_\ell(\mathcal{X}_k) \qand r_k:D_k\to \mathcal{L}_r(\mathcal{X}_k),
\]
such that the joint distribution of the family ${\{(C_k,D_k)\}}_{k\in K}$ with respect to $E$ coincides with the joint distribution of the images
\[
{\{((\lambda_k\circ l_k)(C_k),(\rho_k\circ r_k)(D_k))\}}_{k\in K}
\]
in the space $\mathcal{L}(*_{k\in K}\mathcal{X}_k)$, with respect to $E_{\mathcal{L}(*_{k\in K}\mathcal{X}_k)}$, where $*_{k\in K}\mathcal{X}_k$ is the reduced free product of ${\{(\mathcal{X}_k,\overset{\circ}{\mathcal{X}}_k,p_k)\}}_{k\in K}$ with amalgamation over $B$.
\end{enumerate}
\end{definition}

\begin{remark}\label{rem:inflating-preservers-bi-freeness}
Let $\A$ and $B$ be unital $*$-algebras and let $\varphi : \A \to \bC$ be a unital linear map.  Let $(A, E, \varepsilon)$ be as in Example \ref{exam:tensor-up}.  By \cites{S2016-1, CS2019}, if $\{(C_k, D_k)\}$ are $*$-subalgebras of $\A$ that are bi-free with respect to $\varphi$, then $\{(C_k \otimes B \otimes 1_B, D_k \otimes 1_B \otimes B^\op)\}_{k \in K}$ are bi-free with amalgamation over $B$ with respect to $E$.  Thus Example \ref{exam:tensor-up} is the correct notion of ``inflating $(\A, \varphi)$ by $B$'' in the bi-free setting.
\end{remark}

\begin{example}
Let $\fM_1$ and $\fM_2$ be finite von Neumann algebras with a common von Neumann subalgebra $B$ and tracial states $\tau_1$ and $\tau_2$ respectively such that $\tau_1|_B = \tau_2|_B$.  Let $\fM = \fM_1\ast_B \fM_2$ be the reduced free product von Neumann algebra with amalgamation over $B$, let $E_B : \fM \to B$ be the conditional expectation of $\fM$ onto $B$, and let $\tau = \tau_1 \ast \tau_2 = \tau_1|_B \circ E_B$ be the tracial state on $\fM$.  If $E$ and $\varepsilon$ are as in Example \ref{exam:factors} for $(\fM, \tau)$, then
\[
\left\{\left(\{L_X \, \mid \, X \in \fM_1\}, \{R_Y \, \mid \, Y \in \fM_1\}\right)\right\}\qqand \left\{\left(\{L_X \, \mid \, X \in \fM_2\}, \{R_Y \, \mid \, Y \in \fM_2\}\right)\right\}
\]
are bi-free with amalgamation over $B$.
\end{example}

In order to study bi-free independence with amalgamation, the operator-valued bi-free moment and cumulant functions are key.  These functions have specific properties that are described via the following concept.  In what follows and for the remainder of the paper, given an $n$-tuple $(Z_1, \ldots, Z_n)$ and  $V \subseteq \{1,\ldots, n\}$, we will use $(Z_1, \ldots, Z_n)|_V$ to denote the $|V|$-tuple where only the entries $Z_k$ where $k \in V$ remain.

\begin{definition}[\cite{CNS2015-1}*{Definition 4.2.1}]
\label{defn:bi-multiplicative-function}
Let $(A,E,\varepsilon)$ be a $B$-$B$-non-commutative probability space. A map
\[
\Phi : \bigcup_{n\in\bN}\bigcup_{\chi\in\{\ell,r\}^n}\BNC(\chi)\times A_{\chi(1)}\times A_{\chi(2)}\times\ldots\times A_{\chi(n)}\to B
\]
is called \textit{bi-multiplicative} if it is $\C$-linear in each of the $A_{\chi(k)}$ entries and for all $n\in\bN$, $\chi\in \{\ell,r\}^n$, $\pi\in\BNC(\chi)$, $b\in B$, and $Z_k\in A_{\chi(k)}$ the following four conditions hold:
\begin{enumerate}[(i)]
\item Let 
\[
q = \max_\leq\{k\in\{1,\ldots,n\} \, \mid \, \chi(k)\neq \chi(n)\}.
\]

If $\chi(n)=\ell,$ then
\[
\Phi_{1_{\chi}}(Z_1,\ldots,Z_{n-1},Z_n L_b) = \begin{cases}
       \Phi_{1_\chi}(Z_1,\ldots,Z_{q-1},Z_q R_b,Z_{q+1},\ldots,Z_n)  &\text{if }q\neq -\infty,\\
       \Phi_{1_\chi}(Z_1,\ldots,Z_{n-1},Z_n) b &   \text{if }  q= - \infty. \\
     \end{cases}
\]
If $\chi(n)=r,$ then
\[
\Phi_{1_{\chi}}(Z_1,\ldots,Z_{n-1},Z_n R_b) = \begin{cases}
       \Phi_{1_\chi}(Z_1,\ldots,Z_{q-1},Z_q L_b,Z_{q+1},\ldots,Z_n)  &\text{if }q\neq -\infty,\\
       b \Phi_{1_\chi}(Z_1,\ldots,Z_{n-1},Z_n) &   \text{if } q= - \infty. \\
     \end{cases}
\]
\item Let $p\in\{1,\ldots,n\}$ and let
\[
q = \max_\leq\left\{k\in\{1,\ldots,n\}\, \mid \, \chi(k)=\chi(p), k<p\right\}.
\]
If $\chi(p)=\ell,$ then
\[
\Phi_{1_\chi}(Z_1,\ldots,Z_{p-1},L_b Z_p,Z_{p+1}\ldots,Z_n)= \begin{cases}
       \Phi_{1_\chi}(Z_1,\ldots,Z_{q-1},Z_q L_b,Z_{q+1},\ldots,Z_n) &\text{if }q\neq -\infty,\\
      b \Phi_{1_\chi}(Z_1,\ldots,Z_{n-1},Z_n) &  \text{if }q= - \infty. \\
     \end{cases}
\]
If $\chi(p)=r,$ then
\[
\Phi_{1_\chi}(Z_1,\ldots,Z_{p-1},R_b Z_p,Z_{p+1}\ldots,Z_n)= \begin{cases}
       \Phi_{1_\chi}(Z_1,\ldots,Z_{q-1},Z_q R_b,Z_{q+1},\ldots,Z_n) & \text{if }q\neq -\infty,\\
      \Phi_{1_\chi}(Z_1,\ldots,Z_{n-1},Z_n)  b &  \text{if }q= - \infty. \\
     \end{cases}
\]
\item Suppose $V_1,\ldots, V_m$ are unions of blocks of $\pi$ that partition $\{1,\ldots,n\}$ with each being a $\chi$-interval (i.e. an interval in the $\chi$-ordering) and the sets $V_1,\ldots,V_m$ are ordered by $\pleq$ (i.e. $(\min_{\pleq}V_k)\ple (\min_{\pleq}V_{k+1})$ for all $k$). Then
\[
\Phi_\pi(Z_1,\ldots,Z_n) = \Phi_{\pi|_{V_1}}({(Z_1,\ldots,Z_n)|}_{V_1})\cdots  \Phi_{\pi|_{V_m}}({(Z_1,\ldots,Z_n)|}_{V_m}).
\]
\item Suppose that $V$ and $W$ are unions of blocks of $\pi$ that partition $\{1,\ldots,n\}$, $V$ is a $\chi$-interval, and $\sx(1),\sx(n)\in W$. Let 
\[
p = \max_{\pleq}\left\{k\in W \, \left| \,  k \ple \min_{\pleq}V \right. \right\} \qqand q = \min_{\pleq}\left\{k\in W \, \left| \, \max_{\pleq}V \ple  k\right. \right\}.
\]
Then, we have that
\begin{align*} 
\Phi_\pi(Z_1,\ldots,Z_n) &=  \begin{cases}
       \Phi_{\pi|_W}\left({(Z_1,\ldots,Z_{p-1},Z_p L_{\Phi_{\pi|_V}({(Z_1,\ldots,Z_n)|}_V)},Z_{p+1},\ldots,Z_n) |}_W\right) &\text{if }\chi(p)=\ell,\\
       \Phi_{\pi|_W}\left({(Z_1,\ldots,Z_{p-1},R_{\Phi_{\pi|_V}({(Z_1,\ldots,Z_n)|}_V)} Z_p ,Z_{p+1},\ldots,Z_n) |}_W\right) &  \text{if }\chi(p)=r, \\
     \end{cases} \\ 
 &=  \begin{cases}
       \Phi_{\pi|_W}\left({(Z_1,\ldots,Z_{q-1},L_{\Phi_{\pi|_V}({(Z_1,\ldots,Z_n)|}_V)} Z_q, Z_{q+1},\ldots,Z_n) |}_W\right)  &\text{if }\chi(q)=\ell,\\
       \Phi_{\pi|_W}\left({(Z_1,\ldots,Z_{q-1}, Z_q R_{\Phi_{\pi|_V}({(Z_1,\ldots,Z_n)|}_V)} ,Z_{q+1},\ldots,Z_n) |}_W\right) &  \text{if }  \chi(q)=r. \\
     \end{cases}
\end{align*}
\end{enumerate}
\end{definition}

Given a $B$-$B$-non-commutative probability space $(A,E,\varepsilon)$, the moment and cumulant functions are well-defined bi-multiplicative functions.

\begin{definition}\label{bidef}
Let $(A,E,\varepsilon)$ be a $B$-$B$-non-commutative probability space.
\begin{enumerate}[(i)]
\item The \textit{operator-valued bi-free moment function} 
\[
E : \bigcup_{n\in\bN}\bigcup_{\chi\in\{\ell,r\}^n}\BNC(\chi)\times A_{\chi(1)}\times\ldots\times A_{\chi(n)}\to B
\]
is the bi-multiplicative function (see \cite{CNS2015-1}*{Theorem 5.1.4}) that satisfies
\[
E_{1_\chi}(Z_1,Z_2,\ldots,Z_n) = E(Z_1 Z_2\cdots  Z_n),
\]
for all $n\in\bN$, $\chi\in \{\ell,r\}^n$, and $Z_k\in A_{\chi(k)}$.
\item The \textit{operator-valued bi-free cumulant function}
\[
\kappa^B : \bigcup_{n\in\bN}\bigcup_{\chi\in\{\ell,r\}^n}\BNC(\chi)\times A_{\chi(1)}\times\ldots\times A_{\chi(n)}\to B
\]
is the bi-multiplicative function (see \cite{CNS2015-1}*{Corollary 6.2.2}) defined by
\[
\kappa^B_\pi(Z_1,\ldots,Z_n) = 
\sum_{\substack{\sigma\in\BNC(\chi)\\
\sigma\leq\pi\\}}E_\sigma(Z_1,\ldots,Z_n)\mu_{\BNC}(\sigma,\pi),
\]
for each $n\in\bN$, $\chi\in \{\ell,r\}^n$, $\pi\in\BNC(\chi)$, and $Z_k\in A_{\chi(k)}$. In the special case when $\pi=1_\chi,$ the map $\kappa^B_{1_\chi}$ is simply  denoted by $\kappa^B_\chi$. An instance of M{\"o}bius inversion yields that the equality
\[
E_\sigma(Z_1,\ldots,Z_n) = 
\sum_{\substack{\pi\in\BNC(\chi)\\
\pi\leq\sigma\\}}\kappa^B_\pi(Z_1,\ldots,Z_n)
\]
holds for all $n\in\bN$, $\chi\in \{\ell,r\}^n$, $\sigma\in\BNC(\chi)$, and $Z_k\in A_{\chi(k)}$.
\end{enumerate}
\end{definition}

The condition of bi-freeness with amalgamation over $B$ for a family of pairs of $B$-faces is equivalent to the vanishing of their mixed operator-valued bi-free cumulants, as the following result indicates.

\begin{theorem}[\cite{CNS2015-1}*{Theorem 8.1.1}]
Let $(A,E,\varepsilon)$ be a $B$-$B$-non-commutative probability space and let ${\{(C_k,D_k)\}}_{k\in K}$ be a family of pairs of $B$-algebras in $A$. The following are equivalent:
\begin{enumerate}[(i)]
\item the family ${\{(C_k,D_k)\}}_{k\in K}$ is bi-free with amalgamation over $B$,
\item for all $n\in\bN$, $\chi\in\{\ell,r\}^n$, $Z_1,\ldots,Z_n\in A$, and non-constant maps $\gamma:\{1,\ldots,n\}\to K$ such that
\[
Z_k\in \begin{cases}
       C_{\gamma(k)} &\text{if } \chi (k)=\ell \\
        D_{\gamma(k)}  &\text{if }\chi (k) = r\\   
     \end{cases} 
\]
we have that
\[
\kappa^B_\chi(Z_1,\ldots,Z_n)=0.
\]
\end{enumerate}
\end{theorem}

\section{Analytical B-B-Non-Commutative Probability Spaces} 
\label{sec:Analytical-B-B-NCPS}

The notion of a $B$-$B$-non-commutative probability space is purely an algebraic construction.  In order perform the more analytical computations necessary in this paper, additional structure is needed.  These structures are analogous to those observed in Example \ref{exam:factors} and will be seen to be the correct enhancement of a $B$-$B$-non-commutative probability space to perform functional analysis.

Given a unital $*$-algebra $A$, by a state $\tau:A\to \C$ we will always mean a unital, linear functional with the property that $\tau(a^*a) \geq 0$ for all $a \in A$.  If $N_\tau = \{a\in A : \tau(a^*a)=0\}$, then $L_2(A,\tau)$ will denote the Hilbert space completion of the quotient space $A/N_\tau$ with respect to the inner product induced by $\tau$  given by
\[
\left\langle a_1+N_\tau,a_2+N_\tau\right\rangle = \tau(a_2^*a_1),
\]
for all $a_1,a_2\in A$ and $\left\| \,\, \cdot \, \, \right\|_\tau$ will denote the Hilbert space norm on $L_2(A, \tau)$.

\begin{definition}\label{enhancedbbdef}
Given a unital $*$-algebra $B$, an \textit{analytical $B$-$B$-non-commutative probability space} consists of a tuple $(A,E,\varepsilon,\tau)$ such that
\begin{enumerate}[(i)]
\item $(A,E,\varepsilon)$ is a $B$-$B$-non-commutative probability space,
\item $\tau:A\to\C$ is a state that is compatible with $E$; that is, 
\[
\tau(a) = \tau\left(L_{E(a)}\right) = \tau\left(R_{E(a)}\right)
\]
for all $a\in A$,
\item the canonical state $\tau_B:B\to\C$ defined by $\tau_B(b) = \tau(L_b)$ for all $b\in B$ is tracial,
\item left multiplication of $A$ on $A/N_\tau$ are bounded linear operators and thus extend to bounded linear operators on $L_2(A,\tau)$, and
\item $E$ is completely positive when restricted to $A_\ell$ and when restricted to $A_r$.
\end{enumerate}
\end{definition}

\begin{remark}\label{remarkenhancedbbdef}
Given an analytical $B$-$B$-non-commutative probability space $(A, E, \varepsilon, \tau)$, note the following.  
\begin{enumerate}[(i)]
\item The that fact that $\tau_B$ is a state immediately follows from the fact that $\tau$ is a state and $\varepsilon$ is a $*$-homomorphism.  Specifically, for positivity, notice for all $b\in B$ that
\[
\tau_B(b^*b) = \tau(L_{b^*b}) = \tau((L_b)^* L_b)\geq 0.
\]
\item Note for all $b \in B$ that 
\[
{\left\|b+N_{\tau_B}\right\|}_{\tau_B}^2 = \tau_B(b^*b) = \tau(L_{b^*b})= {\left\|L_b+N_\tau\right\|}_\tau^2 .
\]
Hence the map from $B/N_{\tau_B}$ to $L_2(A,\tau)$ defined by
\[
b+N_{\tau_B}\mapsto L_b+N_\tau
\]
for all $b \in B$ is a well-defined, linear isometry. Therefore, a standard density argument yields that
\[
L_2(B,\tau_B)\cong {\overline{\{L_b+N_\tau\, \mid \, b\in B\}}}^{{\left\|\cdot\right\|}_\tau}\subseteq L_2(A,\tau).
\]
Henceforth, we shall only be making reference to the space $L_2(B,\tau_B)$ via this identification.
\item The state $\tau$ naturally extends to a linear functional on $L_2(A,\tau)$ by defining
\[
\tau(\xi) = {\left\langle \xi,1_A+N_\tau\right\rangle }_{L_2(A,\tau)}
\]
for all $\xi\in L_2(A,\tau)$. Similarly, the scalar $\tau_B(\zeta) = \tau(\zeta)$ is well-defined for any $\zeta\in L_2(B,\tau_B)$.
\item As left multiplication by $A$ on $A/N_\tau$ is bounded, we immediately extend the left multiplication map to obtain a unital $*$-homomorphism from $A$ into $\mathcal{B}(L_2(A))$.  Thus $a\xi$ is a well-defined element of $L_2(A,\tau)$ for all $a\in A$ and $\xi\in L_2(A,\tau)$.  
\item The requirement of the left multiplication inducing bounded operators is immediate in the case when $A$ is a $\mathrm{C}^*$-algebra,  however it also holds in more general situations.  For instance, when $A$ is a unital $*$- generated its partial isometries, the left multiplication map is automatically bounded (see \cite{NS2006}*{Exercise 7.22}).
\item Since the state $\tau_B$ is assumed to be tracial, right multiplication of $B$ on $B/N_{\tau_B}$ is also bounded.  Thus, for any $b_1,b_2\in B$ and $\zeta\in L_2(B,\tau_B)$, we have that $b_1 \zeta b_2$ is a well-defined element of $L_2(B,\tau_B)$ and, in $L_2(A, \tau)$, $L_{b_1} R_{b_2} \zeta = b_1 \zeta b_2$.   Furthermore, note that left and right multiplication of $B$ on $L_2(B, \tau_B)$ are commuting $*$-homomorphisms.
\item For all $a \in A$ and $b \in B$, we automatically have $\tau(a L_b) = \tau(a R_b)$, as $\tau$ is compatible with $E$.  Indeed 
\[
\tau(a L_b) = \tau(L_{E(aL_b)}) = \tau(L_{E(aR_b)}) = \tau(aR_b),
\]
as desired.  Hence $L_b+N_\tau = R_b+N_\tau$ for all $b\in B$. 
\end{enumerate}
\end{remark}

In some cases, property (v) of Definition \ref{enhancedbbdef} is redundant.

\begin{lemma}\label{lem:expectations-CP}
Let $(A, E, \varepsilon, \tau)$ satisfy assumptions (i), (ii), (iii), and (iv) of Definition \ref{enhancedbbdef}.  If $B$ is a C$^*$-algebra and $\tau_B$ is faithful, then property (v) of Definition \ref{enhancedbbdef} holds.
\end{lemma}
\begin{proof}
To see that $E$ is completely positive on $A_\ell$, let $d \in \bN$ and $A = [a_{i,j}] \in M_d(A_\ell)$.  To verify that $E_d(A^*A) \geq 0$ in $B$, as $B$ is a C$^*$-algebra and $\tau_B$ is faithful, it suffices to show for all $h = (b_1, \ldots, b_d) \in B^d$ that
\[
\langle E_d(A^*A) h, h\rangle_{L_2(B, \tau_B)^{\oplus d}} \geq 0.
\]
Note that
\begin{align*}
\left\langle E_d(A^*A) h, h\right\rangle_{L_2(B, \tau_B)^{\oplus d}} &= \sum^d_{i,j,k=1} \tau_B\left(b_i^* E(a^*_{k,i} a_{k,j}) b_j\right)   \\
&= \sum^d_{i,j,k=1} \tau_B\left(E(R_{b_j} L_{b_i^*} a^*_{k,i} a_{k,j}) \right)\\
&= \sum^d_{i,j,k=1} \tau_B\left(E(L_{b_i^*} a^*_{k,i} a_{k,j} R_{b_j} ) \right)\\
&= \sum^d_{i,j,k=1} \tau_B\left(E(L_{b_i}^* a^*_{k,i} a_{k,j} L_{b_j} ) \right)\\
&= \sum^d_{i,j,k=1} \tau\left(L_{b_i}^* a^*_{k,i} a_{k,j} L_{b_j} \right)\\
&= \sum^d_{k=1} \tau\left(c_k^*c_k \right)\\
\end{align*}
where $c_k = \sum^d_{j=1} a_{k,j} L_{b_j}$.  Hence, as $\tau$ is positive and the computation for $A_r$ is similar, the result follows.
\end{proof}

At this point, let us revisit Examples \ref{exam:tensor-up} and \ref{exam:factors} to provide the canonical examples of analytical $B$-$B$-non-commutative probability spaces.

\begin{example}\label{canonicalexample}
Let $\A$ and $B$ be unital $\mathrm{C}^*$-algebras and let $\phi: \A \to\C$ be a state. Recall from Example \ref{exam:tensor-up} that $(A, E, \varepsilon)$ is a $B$-$B$-non-commutative probability space where $A=\A \otimes B\otimes B^{\op}$, $\varepsilon:B\otimes B^{\op}\to A$ is the natural embedding, and $E:A\to B$ is defined by
\[
E(Z\otimes b_1\otimes b_2) = \phi(Z)b_1b_2,
\]
for all $Z\in\A$ and $b_1,b_2\in B$.

Let $\tau_B:B\to\C$ be any tracial state.  Extend $\tau_B$ to a linear map $\tau : A \to \bC$ by defining
\[
\tau(Z\otimes b_1\otimes b_2) = \tau_B (E(Z\otimes b_1\otimes b_2)) = \phi(Z)\tau_B(b_1b_2),
\]
for all $Z\otimes b_1\otimes b_2\in A$.  We claim that $(A, E, \varepsilon, \tau)$ is an analytical $B$-$B$-non-commutative probability space.  To see this, it suffices to prove that $\tau$ is a state that is compatible with $E$, since $\A$ and $B$ being unital C$^*$-algebras automatically implies that left multiplication will be bounded on $L_2(A, \tau)$, and Lemma \ref{lem:expectations-CP} implies that $E$ is completely positive when restricted to $A_\ell$ or $A_r$ (or one may simply use the fact that states are completely positive).  

Clearly $\tau$ is a unital, linear map that is compatible with $E$.  To see that $\tau$ is positive, let ${(Z_i)}_{i=1}^n\subseteq \A$, ${(b_k)}_{k=1}^n,{(c_k)}_{k=1}^n\subseteq B$, and
 \[
a = \sum_{k=1}^n Z_k\otimes b_k\otimes c_k\in A.
\]
To see that  $\tau(a^*a)\geq 0$, note that
\begin{align*}
\tau(a^*a) &= \sum_{i,j=1}^n\tau(Z_i^* Z_j\otimes b_i^* b_j\otimes c_j c_i^*)\\
&= \sum_{i,j=1}^n \phi(Z_i^* Z_j)\tau_B(b_i^*b_j c_j c_i^*)\\
&=\sum_{i,j=1}^n \phi(Z_i^* Z_j)\tau_B(c_i^* b_i^*b_j c_j)\\
&= \sum_{i,j=1}^n \phi(Z_i^* Z_j)\tau_B({(b_i c_i)}^*(b_j c_j)),
\end{align*}
with the third equality being due to the fact that $\tau_B$ is tracial.  Observe that the matrices 
\[
{[Z_i^* Z_j]}    \qqand  {[{(b_i c_i)}^* (b_j c_j)]} 
\]
are positive in $M_n(\A)$ and $M_n(B)$ respectively.  Therefore, as states on C$^*$-algebras are completely positive, this implies that the matrices
\[
{[\phi(Z_i^* Z_j)]}   \qqand {[{\tau_B((b_i c_i)}^* (b_j c_j))]}
\]
are positive in $M_n(\bC)$. Consequently
\[
{[\phi(Z_i^*Z_j){\tau_B((b_i c_i)}^* (b_j c_j))]} 
\]
is also positive being the Schur product of positive matrices (see, for instance, \cite{NS2006}*{Lemma 6.11}). Therefore, as the sum of all entries of a positive matrix equals a positive scalar, we obtain that $\tau(a^*a)\geq 0$. Hence $(A,E,\varepsilon,\tau)$ is an analytical $B$-$B$-non-commutative probability space.
\end{example}

\begin{remark}
Note Example \ref{canonicalexample} demonstrates $E$ need not be a positive map on $A$ since the product of two positive matrices need not be positive.  Thus, even if $\tau_B : B \to \bC$ is defined to be a state, $\tau_B \circ E$ may not be for an arbitrary $A$.
\end{remark}

\begin{example}
\label{exam:vonNeumannAnalytical}
For a finite von Neumann algebra $\fM$ with a unital von Neumann subalgebra $B$ and tracial state $\tau$, let $(A, E, \varepsilon)$ be the $B$-$B$-non-commutative probability space as in Example \ref{exam:factors}.  Note that $\tau$ extends to a unital linear map $\tau_A : A \to \bC$ defined by
\[
\tau_A(T) = \langle T1_\fM, 1_\fM\rangle_{L_2(\fM,\tau)}
\]
for all $T \in A$.  Clearly $\tau_A$ is a state as $A \subseteq \mathcal{B}(L_2(\fM, \tau))$ and $\tau_A$ is a vector state.  Furthermore, notice that
\[
\tau_A(T) =  \langle P(T1_\fM), 1_\fM\rangle_{L_2(\fM,\tau)} = \langle L_{E(T)} 1_\fM, 1_\fM\rangle_{L_2(\fM,\tau)} = \tau_A(L_{E(T)})  
\]
for all $T \in A$ and $\tau_A(T) = \tau_A(R_{E(T)})$ by a similar computation.  Finally as
\[
\tau_A(L_b) = \langle b, 1_\fM\rangle_{L_2(\fM,\tau)} = \tau(b)
\]
for all $b \in B$, we see that $\tau_B = \tau_A \circ E$ is tracial on $B$ as $\tau$ is tracial.  Again, we automatically have that left multiplication will be bounded on $L_2(A, \tau)$ and that $E$ is completely positive when restricted (as they are the conditional expectation of a copy of $\fM$ onto $B$).   Hence $(A, E, \varepsilon, \tau_A)$ is an analytical $B$-$B$-non-commutative probability space.
\end{example}

As motivated by Example \ref{exam:vonNeumannAnalytical}, it is natural in an analytical $B$-$B$-non-commutative probability space to extend the expectation $E:A\to B$ to a map from $L_2(A,\tau)$ to $L_2(B,\tau_B)$ via orthogonal projection.  From this point onwards, for $a \in A$ we will often denote the coset $a+N_\tau$ simply by $a$ and, for $b \in B$ we will often denote the coset $b+N_{\tau_b}$ by $\widehat{b}$.  Note that if $\tau_B$ is faithful, then the map $b \mapsto \widehat{b}$ is a bijection.

\begin{proposition}\label{orthproj}
Let $(A,E,\varepsilon,\tau)$ be an analytical $B$-$B$-non-commutative probability space.  If $\widetilde{E}:L_2(A,\tau)\to L_2(B,\tau_B)$ denotes the orthogonal projection, then
\[
\widetilde{E}(a) = \widehat{E(a)}
\]
for all $a \in A$.  In particular, when $\tau_B$ is faithful, $\widetilde{E}$ extends $E$.
\end{proposition}
\begin{proof}
Notice for all $a \in A$ and $b \in B$ that
\begin{align*}
{\left\langle a-\widehat{E(a)}, L_b\right\rangle }_{L_2(B,\tau_B)} &= \langle L_{b^*}(a - L_{E(a)}) 1_A, 1_A \rangle_{L_2(A,\tau)} \\
&= \tau\left(L_{b^*}(a - L_{E(a)}) \right) \\
&= \tau(L_{b^*}a) - \tau\left(L_{b*} L_{E(a)}\right) \\
&= \tau\left(L_{E(L_{b^*}a)}\right) - \tau\left(L_{b^*E(a)}\right) = 0.
\end{align*}
Since $b$ was arbitrary, the element $a-\widehat{E(a)}$ is orthogonal to $L_2(B,\tau_B)$ and hence $\widetilde{E}(a) = \widehat{E(a)}$.
\end{proof}

\begin{remark}
Notice in Proposition \ref{orthproj} that if $B$ is finite-dimensional and the trace $\tau_B:B\to\C$ is faithful, then $L_2(B,\tau_B)\cong B$, so $E:A\to B$  extends to a map from $L_2(A,\tau)$ into $B$.
\end{remark}

Of course $\widetilde{E}$ inherits many properties that $E$ is required to have.

\begin{proposition}\label{propE}
Let $(A,E,\varepsilon,\tau)$ be an analytical $B$-$B$-non-commutative probability space and let $\widetilde{E}:L_2(A,\tau)\to L_2(B,\tau_B)$ denote the orthogonal projection. For $a\in A$, $b,b_1,b_2\in B$, $\xi, \xi_1,\xi_2\in L_2(A,\tau)$, and $\zeta\in L_2(B,\tau_B)$, the following hold:
\begin{enumerate}[(i)]
\item $\tau(\xi) = \tau_B\left(\widetilde{E}(\xi)\right)$,
\item $\widetilde{E}(aL_b) = \widetilde{E}(aR_b)$,
\item $\widetilde{E}(L_{b_1}R_{b_2}\xi) = b_1\widetilde{E}(\xi)b_2$,
\item if $a\in A_\ell$, then $\widetilde{E}(a\zeta) = E(a)\zeta$,
\item if  $a\in A_r$, then $\widetilde{E}(a\zeta) = \zeta E(a)$, and
\item if $\tau(L_b \xi_1) = \tau(L_b \xi_2)$ for all $b\in B$, then $\widetilde{E}(\xi_1)=\widetilde{E}(\xi_2)$.
\item if $\tau(R_b \xi_1) = \tau(R_b \xi_2)$ for all $b\in B$, then $\widetilde{E}(\xi_1)=\widetilde{E}(\xi_2)$.
\end{enumerate}
\end{proposition}
\begin{proof}
For (i), since  $L_{1_B} = 1_A$ as $\varepsilon$ is unital, note that
\[
\tau_B\left(\widetilde{E}(\xi)\right) = \left\langle \widetilde{E}(\xi), 1_B \right\rangle_{L_2(B, \tau_B)} =\left\langle \xi, \widetilde{E}(1_B) \right\rangle_{L_2(A, \tau)}  = \left\langle \xi, 1_A \right\rangle_{L_2(A, \tau)} = \tau(\xi)
\]
as desired.

For (ii), note for all $b' \in B$ that
\begin{align*}
\left\langle \widetilde{E}(aL_b), \widehat{b'}\right\rangle_{L_2(B, \tau_B)} &= \left\langle aL_b, L_{b'}\right\rangle_{L_2(A, \tau)} = \tau\left(L_{(b')^*} a L_b\right) = \tau\left(L_{(b')^*} a R_b\right) = \left\langle \widetilde{E}(aR_b), \widehat{b'}\right\rangle_{L_2(B, \tau_B)}.
\end{align*}
Hence $\widetilde{E}(aL_b) = \widetilde{E}(aR_b)$.

For (iii), let $(a_n)_{n\geq 1}$ be a sequence of elements of $A$ that converge to $\xi$ in $L_2(A, \tau)$. Since left multiplication in $L_2(A, \tau)$ by elements of $A$ are bounded and thus continuous, and since left and right multiplication in $L_2(B, \tau_B)$ by elements of $B$ are bounded and thus continuous,  we obtain that
\begin{align*}
\widetilde{E}(L_{b_1}R_{b_2}\xi) = \lim_{n\to \infty} E(L_{b_1}R_{b_2}a_n) + N_{\tau_B} 
= \lim_{n\to \infty} b_1 E(a_n) b_2 + N_{\tau_B}   = b_1 \widetilde{E}(\xi) b_2
\end{align*}
as desired.

For (iv) and (v), let $(c_n)_{n\geq 1}$ be a sequence of elements of $B$ that converge to $\zeta$ in $L_2(B, \tau_B)$.  Thus, by the inclusion of $L_2(B, \tau_B)$ into $L_2(A, \tau)$, we have that $(L_{c_n})_{n\geq 1}$ is a sequence of elements of $L_2(A, \tau)$ that converge to $\zeta$ in $L_2(A, \tau)$.  Thus, if $a \in A_\ell$, then
\begin{align*}
\widetilde{E}(a\zeta)
&= \lim_{n\to \infty} E(aL_{c_n}) + N_{\tau_B} \\
&= \lim_{n\to \infty} E(aR_{c_n}) + N_{\tau_B} \\
&= \lim_{n\to \infty} E(R_{c_n}a) + N_{\tau_B} \\
&= \lim_{n\to \infty} E(a)c_n + N_{\tau_B} = E(a) \zeta
\end{align*}
thereby proving (iv).  Note (v) is similar using $(R_{c_n})_{n\geq 1}$ in place of $(L_{c_n})_{n\geq 1}$.

As (vi) and (vii) are similar, we prove (vii).  Note by (iii) and the fact that $\tau_B$ is tracial that 
\begin{align*}
\left\langle \widetilde{E}(\xi_1)-\widetilde{E}(\xi_2), \widehat{b} \right\rangle_{L_2(B,\tau_B)} &= \left\langle \widetilde{E}(\xi_1)b^*, \widetilde{E}(1_A)\right\rangle_{L_2(B,\tau_B)} - \left\langle \widetilde{E}(\xi_2)b^*, \widetilde{E}(1_A)\right\rangle_{L_2(B,\tau_B)}\\
&= \left\langle \widetilde{E}(R_{b^*}\xi_1),\widetilde{E}(1_A)\right\rangle_{L_2(B,\tau_B)} - \left\langle \widetilde{E}(R_{b^*} \xi_2), \widetilde{E}(1_A)\right\rangle_{L_2(B,\tau_B)}\\
&= \left\langle R_{b^*}\xi_1,1_A\right\rangle_{L_2(A,\tau)} - \left\langle R_{b^*} \xi_2, 1_A\right\rangle_{L_2(A,\tau)}\\
&= \tau(R_{b^*}\xi_1)-\tau(R_{b^*}\xi_2) = 0
\end{align*}
As the above holds for all $b\in B$, (vii) follows.
\end{proof}

\section{Analytical Bi-Multiplicative Functions}
\label{sec:Bi-Multi}

In this section, we extend the notion of bi-multiplicative functions on analytical $B$-$B$-non-commutative probability spaces in order to permit the last entry to be an element of $L_2(A, \tau)$.  This is possible as the last entry can be treated as a left or right operator as \cite{S2016-2} shows, or can be treated as a mixture of left and right operators as \cite{CS2020} shows.  Extending the operator-valued bi-free cumulant function to permit the last entry to be an element of $L_2(A, \tau)$ is necessary in order to permit the simple development of conjugate variable systems in the next section.

We advise the reader that familiarity with specifics of bi-multiplicative functions, the construction of the operator-valued bi-free moment function, and the construction of the operator-valued bi-free cumulant function from \cite{CNS2015-1} would be of great aid in comprehension of this section.  As the proofs are nearly identical, to avoid clutter we will focus on that which is different and why the results of \cite{CNS2015-1} extend.

\begin{definition}\label{defn:anal-bi-multiplicative}
Let $(A,E,\varepsilon,\tau)$ be an analytical $B$-$B$-non-commutative probability space and let $\Phi$ be a bi-multiplicative function on $(A, E, \varepsilon)$.  A function
\[
\widetilde{\Phi}  :\bigcup_{n\in\bN}\bigcup_{\chi\in \{\ell,r\}^n}\BNC(\chi)\times A_{\chi(1)}\times\ldots\times A_{\chi(n-1)}\times L_2(A,\tau)\to L_2(B,\tau_B)
\]
is said to be \emph{analytical extension of $\Phi$} if $\widetilde{\Phi}_\pi$ is $\C$-multilinear function that does not change values if the last entry of $\chi$ is changed from an $\ell$ to an $r$ and satisfies the following three properties:   For all $n\in\bN$, $\chi\in\{\ell,r\}^n$, $\pi\in\BNC(\chi)$, $\xi\in L_2(A,\tau)$, $\zeta \in L_2(B, \tau_B)$, $b\in B$, and $Z_k\in A_{\chi(k)}$:
\begin{enumerate}[(i)]
\item If $\chi(k) = \ell$ for all $k \in \{1,2, \ldots, n\}$ then
\[
\widetilde{\Phi}_{1_\chi}(Z_1,\ldots, Z_{n-1}, Z_n\zeta) = \Phi_{1_\chi}(Z_1, \ldots, Z_n) \zeta,
\]
and if $\chi(k) = r$ for all $k \in \{1,2,\ldots, n\}$, then 
\[
\widetilde{\Phi}_{1_\chi}(Z_1,\ldots,Z_{n-1}, Z_n\zeta) = \zeta \Phi_{1_\chi}(Z_1, \ldots, Z_n).
\]
In particular, by setting $\zeta = 1_B = 1_A$, we see $\widetilde{\Phi}$ does extend $\Phi$.
\item[(ii)] Let $p\in\{1,\ldots,n\}$ and let
\[
q = \max_\leq\left\{k\in\{1,\ldots,n\}\, \mid \, \chi(k)=\chi(p),  k<p\right\}.
\]
If $\chi(p)=\ell$, then
\[
\widetilde{\Phi}_{1_\chi}(Z_1,\ldots,Z_{p-1},L_b Z_p, Z_{p+1}, \ldots,Z_{n-1},\xi)= \begin{cases}
       \widetilde{\Phi}_{1_\chi}(Z_1,\ldots,Z_{q-1},Z_q L_b, Z_{q+1}, \ldots,Z_{n-1},\xi)  &\text{if }q\neq -\infty,\\
      b \widetilde{\Phi}_{1_\chi}(Z_1,\ldots,Z_{n-1},\xi) &   \text{if } q= - \infty \\
     \end{cases},
\]
and if $\chi(p)=r,$ then
\[
\widetilde{\Phi}_{1_\chi}(Z_1,\ldots,Z_{p-1},R_b Z_p, Z_{p+1}, \ldots,Z_{n-1},\xi)= \begin{cases}
       \widetilde{\Phi}_{1_\chi}(Z_1,\ldots,Z_{q-1},Z_q R_b,Z_{q+1}, \ldots,Z_{n-1},\xi) & \text{if }q\neq -\infty,\\
      \widetilde{\Phi}_{1_\chi}(Z_1,\ldots,Z_{n-1},\xi)b &   \text{if } q= - \infty. \\
     \end{cases}
\]
\item[(iii)] Suppose $V_1,\ldots, V_m$ are unions of blocks of $\pi$ that partition $\{1,\ldots,n\}$, with each being a $\chi$-interval. Moreover, assume that the sets $V_1,\ldots,V_m$ are ordered by $\pleq$ (i.e. $(\min_{\pleq}V_k)\ple (\min_{\pleq}V_{k+1})$). Let $q\in\{1,\ldots,m\}$ be such that $n\in V_q$ and for each $k\neq q$ let
\[
b_k = \Phi_{\pi|_{V_k}}\left({(Z_1,\ldots,Z_{n-1},\xi)|}_{V_k}\right).
\]
Then $b_k\in B$ for $k \neq q$ and
\[
\widetilde{\Phi}_\pi(Z_1,\ldots,Z_{n-1},\xi) = b_1b_2 \cdots b_{q-1} \widetilde{\Phi}_{\pi|_{V_i}}({(Z_1,\ldots,Z_{n-1},\xi)|}_{V_q}) b_{q+1} \cdots b_m.
\]
\item[(iv)] Suppose that $V$ and $W$ are unions of blocks of $\pi$ that partition $\{1,\ldots,n\}$ such that $V$ is a $\chi$-interval and $\sx(1),\sx(n)\in W$. Let
\[
p = \max_{\pleq} \left \{k\in W\, \left| \, k \pleq \min_{\pleq}V \right.\right\}\qqand q = \min_{\pleq}\left\{k\in W\, \left| \, \max_{\pleq}V \pleq k \right. \right\}.
\]
Then one of the following four cases hold:
\begin{enumerate}[a)]
\item If $n\in V$ and $k=\max_\leq W$, then
\[
\widetilde{\Phi}_\pi(Z_1,\ldots,Z_{n-1},\xi) =  \widetilde{\Phi}_{\pi|_W}\left({(Z_1,\ldots,Z_{k-1}, Z_k \widetilde{\Phi}_{\pi|_V}{({(Z_1,\ldots,Z_{n-1},\xi)|}_V)} ,\ldots,Z_{n-1},\xi)|}_W\right).
\]
\item If $n\in W$ then
\begin{align*} 
\widetilde{\Phi}_\pi(Z_1,\ldots,Z_n) &=  \begin{cases}
       \widetilde{\Phi}_{\pi|_W}\left({(Z_1,\ldots,Z_{p-1},Z_p L_{\Phi_{\pi|_V}({(Z_1,\ldots,Z_n)|}_V)},Z_{p+1},\ldots,Z_n) |}_W\right) &\text{if }\chi(p)=\ell,\\
       \widetilde{\Phi}_{\pi|_W}\left({(Z_1,\ldots,Z_{p-1},R_{\Phi_{\pi|_V}({(Z_1,\ldots,Z_n)|}_V)} Z_p ,Z_{p+1},\ldots,Z_n) |}_W\right) &  \text{if }\chi(p)=r, \\
     \end{cases} \\ 
 &=  \begin{cases}
       \widetilde{\Phi}_{\pi|_W}\left({(Z_1,\ldots,Z_{q-1},L_{\Phi_{\pi|_V}({(Z_1,\ldots,Z_n)|}_V)} Z_q, Z_{q+1},\ldots,Z_n) |}_W\right)  &\text{if }\chi(q)=\ell,\\
       \widetilde{\Phi}_{\pi|_W}\left({(Z_1,\ldots,Z_{q-1}, Z_q R_{\Phi_{\pi|_V}({(Z_1,\ldots,Z_n)|}_V)} ,Z_{q+1},\ldots,Z_n) |}_W\right) &  \text{if }  \chi(q)=r. \\
     \end{cases}
\end{align*}
(Recall we can set $\chi(n) = \ell$ or $\chi(n) = r$.)
\end{enumerate}
\end{enumerate}
\end{definition}
\begin{remark}
Note that the pair of a bi-multiplicative function and its extension are very reminiscent of the two expectation extensions of bi-multiplicative functions used for operator-valued conditional bi-free independence from \cite{GS2018}.  The main difference is that the notion in \cite{GS2018} looks at interior versus exterior blocks of the partition whereas Definition \ref{defn:anal-bi-multiplicative} looks at the blocks containing the last entry.  This is due to the fact that the $L_2(A, \tau)$ element is always the last entry and must be treated differently being a generalization of a mixture of left and right operators.

It is worth pointing out that treating the last entry as an element of $L_2(A, \tau)$ is no issue.  In particular, the properties in Definition \ref{defn:anal-bi-multiplicative} are well-defined. Indeed, properties (i) and (ii) of Definition \ref{defn:anal-bi-multiplicative} are clearly well-defined and properties (iii) and (iv) in Definition \ref{defn:anal-bi-multiplicative} are well-defined as all terms where $\Phi$ is used over $\widetilde{\Phi}$ never involve an element of $L_2(A, \tau)$ and as elements from $B$ have left and right actions on $L_2(B, \tau_B)$.
\end{remark}
\begin{remark}
Note that property (i) of Definition \ref{defn:anal-bi-multiplicative} is clearly  the correct generalization of property (i) from Definition \ref{defn:bi-multiplicative-function}, as an element of $L_2(B, \tau_B)$ is playing the role of $L_b$ and $R_b$ in this generalization and thus should be able to escape these expressions if only left operators or right operators are present.  The absence of the full property (i) from Definition \ref{defn:bi-multiplicative-function}  causes no issues when attempting to reduce or rearrange the value of $\widetilde{\Phi}_\pi$ to an expression involving only $\widetilde{\Phi}_{1_\chi}$'s, as the last entry of any sequence input into $\widetilde{\Phi}$ is always in $L_2(A, \tau)$, is reduced to an element of $L_2(B, \tau)$, and an element of $A_\ell$ or $A_r$ then acts on it via the left action of $A$ on $L_2(A, \tau)$. Thus there is never any need to move the $L_2(A, \tau)$ entry to another position.

If property (i) is ever used, we note that if $\zeta \in L_2(B, \tau_B)$ is viewed as an element of $L_2(A, \tau)$, then $L_b \zeta$ is simply the element $b\zeta  \in L_2(B, \tau_B)$ and $R_b \zeta$ is simply the element $\zeta b \in L_2(B,\tau_B)$.  Thus, using (i) does not pose problems when trying to ``move around $L_b$ and $R_b$ elements'' in proofs when trying to show the equivalence of any reductions as the following example demonstrates.
\end{remark}

\begin{example}\label{exam:analytic-bi-mult-functions-work}
Let $\chi\in\{\ell,r\}^{8}$ be such that $\chi^{-1}(\{\ell\})=\{5,6\}$, let $\xi\in L_2(A,\tau)$, let $Z_k\in A_{\chi(k)}$, and let $\pi\in\BNC(\chi)$ be the partition 
\[
\pi = \{\{1,2\}, \{3,5\}, \{4,7\}, \{6,8\}\}.
\]
Note the bi-non-crossing diagram of $\pi$ can be represented as the following (with the convention now that the last entry is at the bottom instead of on its respective side):
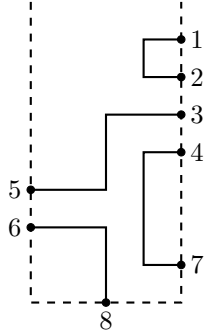
\begin{align*}
\begin{tikzpicture}[baseline]
    \draw[thick,dashed] (-1,4) -- (-1,0) -- (1,0) -- (1,4);
	\node[right] at (1, 3.5) {1};
	\draw[fill=black] (1,3.5) circle (0.05);
	\node[right] at (1, 3) {2};
	\draw[fill=black] (1,3) circle (0.05);
	\node[right] at (1, 2.5) {3};
	\draw[fill=black] (1,2.5) circle (0.05);
	\node[right] at (1, 2) {4};
	\draw[fill=black] (1,2) circle (0.05);
	\node[left] at (-1, 1.5) {5};
	\draw[fill=black] (-1,1.5) circle (0.05);
	\node[left] at (-1, 1) {6};
	\draw[fill=black] (-1,1) circle (0.05);
	\node[right] at (1, .5) {7};
	\draw[fill=black] (1,.5) circle (0.05);
	\node[below] at (0,0) {8};
	\draw[fill=black] (0,0) circle (0.05);
	\draw[thick,black] (1,3.5) -- (0.5,3.5) -- (0.5, 3) -- (1,3);
	\draw[thick, black] (1,2.5) -- (0,2.5) -- (0, 1.5) -- (-1, 1.5);
	\draw[thick, black] (1,2) -- (0.5,2) -- (0.5, 0.5) -- (1,0.5);
	\draw[thick, black] (-1,1) -- (0,1) -- (0, 0);
\end{tikzpicture}
\end{align*}
When reducing $\widetilde{\Phi}_\pi(Z_1, Z_2, \ldots, Z_7, \xi)$, we can clearly use property (iii) of Definition \ref{defn:anal-bi-multiplicative} first to obtain with $U = \{3,4,\ldots, 8\}$ that
\[
\widetilde{\Phi}_\pi(Z_1, Z_2, \ldots, Z_7, \xi) = \widetilde{\Phi}_{\pi|_U}(Z_3, Z_4, \ldots, Z_7, \xi) \Phi_{1_{(r,r)}}(Z_1, Z_2).
\]
To reduce the expression fully, we have to simply reduce $\widetilde{\Phi}_{\pi|_U}(Z_3, Z_4, \ldots, Z_7, \xi)$ using  property (iv) of Definition \ref{defn:anal-bi-multiplicative} of which there are three ways to do so.

The first way to reduce is to use $V = \{4,6,7,8\}$ and $W = \{3,5\}$.  By applying property (iv) of Definition \ref{defn:anal-bi-multiplicative} we obtain that
\begin{align*}
\widetilde{\Phi}_{\pi|_U}(Z_3, Z_4, \ldots, Z_7, \xi) &= \widetilde{\Phi}_{\pi|_W}\left(Z_3, Z_5\widetilde{\Phi}_{\pi|_V}(Z_4, Z_6, Z_7, \xi)\right).
\end{align*}
Finally, by applying property (iii) of Definition \ref{defn:anal-bi-multiplicative} to $\widetilde{\Phi}_{\pi|_V}(Z_4, Z_6, Z_7, \xi)$ we obtain that
\begin{align*}
\widetilde{\Phi}_{\pi|_U}(Z_3, Z_4, \ldots, Z_7, \xi) &= \widetilde{\Phi}_{1_{(r,\ell)}}\left(Z_3, Z_5\left(\widetilde{\Phi}_{1_{(\ell, \ell)}}(Z_6, \xi )   \Phi_{1_{(r,r)}}(Z_4, Z_7) \right)   \right)\\
&= \widetilde{\Phi}_{1_{(r,\ell)}}\left(Z_3, Z_5R_{\Phi_{1_{(r,r)}}(Z_4, Z_7)} \widetilde{\Phi}_{1_{(\ell, \ell)}}(Z_6, \xi )      \right).
\end{align*}

The second way to reduce is to use $V = \{6,8\}$ and $W = \{3,4, 5, 7\}$.  By applying property (iv) of Definition \ref{defn:anal-bi-multiplicative} we obtain that
\begin{align*}
\widetilde{\Phi}_{\pi|_U}(Z_3, Z_4, \ldots, Z_7, \xi) &= \widetilde{\Phi}_{\pi|_W}\left(Z_3, Z_4, Z_5, Z_7\widetilde{\Phi}_{\pi|_V}(Z_6, \xi)\right).
\end{align*}
By applying property (iv) of Definition \ref{defn:anal-bi-multiplicative} again as $\{4,7\}$ is now a $\chi|_{W}$-interval), we obtain that 
\begin{align*}
\widetilde{\Phi}_{\pi|_U}(Z_3, Z_4, \ldots, Z_7, \xi) &= \widetilde{\Phi}_{1_{r,\ell}}\left(Z_3, Z_5\left(\widetilde{\Phi}_{1_{(r, r)}}(Z_4, Z_7\widetilde{\Phi}_{1_{(\ell, \ell)}}(Z_6, \xi)  \right)  \right).
\end{align*}
However, as $Z_4, Z_7 \in A_r$, we obtain by property (i) that
\begin{align*}
\widetilde{\Phi}_{\pi|_U}(Z_3, Z_4, \ldots, Z_7, \xi) &= \widetilde{\Phi}_{1_{(r,\ell)}}\left(Z_3, Z_5\left(\widetilde{\Phi}_{1_{(\ell, \ell)}}(Z_6, \xi )   \Phi_{1_{(r,r)}}(Z_4, Z_7) \right)   \right)\\
&= \widetilde{\Phi}_{1_{(r,\ell)}}\left(Z_3, Z_5R_{\Phi_{1_{(r,r)}}(Z_4, Z_7)} \widetilde{\Phi}_{1_{(\ell, \ell)}}(Z_6, \xi )      \right),
\end{align*}
thereby agreeing with the above expression.

The third way to reduce is to use $V = \{4,7\}$ and $W = \{3,5,6,8\}$. By applying property (iv) of Definition \ref{defn:anal-bi-multiplicative} we obtain that
\begin{align*}
\widetilde{\Phi}_{\pi|_U}(Z_3, Z_4, \ldots, Z_7, \xi) &= \widetilde{\Phi}_{\pi|_W}\left(Z_3, Z_5, Z_6, R_{\Phi_{\pi|_V}(Z_4, Z_7)} \xi\right) \\
&= \widetilde{\Phi}_{\pi|_W}\left(Z_3 R_{\Phi_{\pi|_V}(Z_4, Z_7)}, Z_5, Z_6,  \xi\right),
\end{align*}
by using the two expressions in property (iv).  Using either expression, we will now again property (iv) of Definition \ref{defn:anal-bi-multiplicative} as $\{6,8\}$ is now $\chi|_{W}$-interval.  For the first, we obtain that
\begin{align*}
\widetilde{\Phi}_{\pi|_U}(Z_3, Z_4, \ldots, Z_7, \xi) &= \widetilde{\Phi}_{1_{(r, \ell)}}  \left(Z_3, Z_5  \widetilde{\Phi}_{1_{(\ell, \ell)}} \left( Z_6, R_{\Phi_{1_{(r,r)}}(Z_4, Z_7)} \xi \right)\right) \\
&= \widetilde{\Phi}_{1_{(r,\ell)}}\left(Z_3, Z_5 \left(   \widetilde{\Phi}_{1_{(\ell, \ell)}}(Z_6, \xi )\Phi_{1_{(r,r)}}(Z_4, Z_7)  \right)   \right)\\
&= \widetilde{\Phi}_{1_{(r,\ell)}}\left(Z_3, Z_5R_{\Phi_{1_{(r,r)}}(Z_4, Z_7)} \widetilde{\Phi}_{1_{(\ell, \ell)}}(Z_6, \xi )      \right)
\end{align*}
where the second equality follows from applying property (ii) of Definition \ref{defn:anal-bi-multiplicative}, as $Z_6 \in A_\ell$.  For the second expression, we obtain that
\begin{align*}
\widetilde{\Phi}_{\pi|_U}(Z_3, Z_4, \ldots, Z_7, \xi) &= \widetilde{\Phi}_{1_{(r, \ell)}}\left(Z_3 R_{\Phi_{1_{(r,r)}}(Z_4, Z_7)}, Z_5\widetilde{\Phi}_{1_{(\ell, \ell)}}(Z_6, \xi ) \right)\\
&= \widetilde{\Phi}_{1_{(r, \ell)}}\left(Z_3, R_{\Phi_{1_{(r,r)}}(Z_4, Z_7)} Z_5\widetilde{\Phi}_{1_{(\ell, \ell)}}(Z_6, \xi ) \right)\\
&= \widetilde{\Phi}_{1_{(r, \ell)}}\left(Z_3, Z_5 R_{\Phi_{1_{(r,r)}}(Z_4, Z_7)}\widetilde{\Phi}_{1_{(\ell, \ell)}}(Z_6, \xi ) \right),
\end{align*}
where the second equality follows from applying property (ii) of Definition \ref{defn:anal-bi-multiplicative} as the last entry is now the $L_2(A, \tau)$ entry, and the third equality holds as $Z_5 \in A_\ell$ and thus commutes with $R_b$.

Hence Definition \ref{defn:anal-bi-multiplicative} is consistent in this example (and will be in all examples due to similar computations).
\end{example}

Using similar reductions for arbitrary expressions, one can prove the following.

\begin{lemma}\label{lem:analytic-properties-extend-from-1-to-pi}
Let $(A,E,\varepsilon,\tau)$ be an analytical $B$-$B$-non-commutative probability space, let $\Phi$ be a bi-multiplicative function on $(A, E, \varepsilon)$, and let $\widetilde{\Phi}$ be an analytic extension of $\Phi$.  Then, properties (i) and (ii) of Definition \ref{defn:anal-bi-multiplicative}  hold when $1_\chi$ is replaced with any $\pi \in \BNC(\chi)$.
\end{lemma}
\begin{proof}
The proof is essentially the same as the proof that properties (i) and (ii) of Definition \ref{defn:bi-multiplicative-function} hold for $\Phi$ when $1_\chi$ is replaced with any $\pi \in \BNC(\chi)$ as in \cite{CNS2015-1}*{Proposition 4.2.5}.  To see that property (i) of Definition \ref{defn:anal-bi-multiplicative} extends, note when using (iii) and (iv) to reduce the expression for $\widetilde{\Phi}_\pi(Z_1, Z_2, \ldots, Z_{n-1}, Z_n\zeta)$ that one is effectively using the bi-multiplicative properties of $\Phi$ and including $\zeta$ in the appropriate spot.  To see that property (ii) of Definition \ref{defn:anal-bi-multiplicative} extends, indices that are always adjacent in the $\chi$-ordering will remain in the correct ordering so that when $L_b$ or $R_b$ operators are considered, we can always move them outside the $\Phi$- and $\widetilde{\Phi}$-expressions on the correct side to move them to the next operator (that is, things will always move around as they do in the free multiplicative functions from \cite{S1998} after reordering by the $\chi$-order).  For example, in Example \ref{exam:analytic-bi-mult-functions-work}, we showed that
\[
\widetilde{\Phi}_\pi(Z_1, Z_2, \ldots, Z_7, \xi) =   \widetilde{\Phi}_{1_{(r, \ell)}}\left(Z_3, Z_5 R_{\Phi_{1_{(r,r)}}(Z_4, Z_7)}\widetilde{\Phi}_{1_{(\ell, \ell)}}(Z_6, \xi ) \right)   \Phi_{1_{(r,r)}}(Z_1, Z_2).
\]
If $Z_3$ were replaced with $R_b Z_3$, we would have
\begin{align*}
\widetilde{\Phi}_\pi(Z_1, Z_2, R_b Z_3, Z_4 \ldots, Z_7, \xi) &=  \widetilde{\Phi}_{1_{(r, \ell)}}\left(R_b Z_3, Z_5 R_{\Phi_{1_{(r,r)}}(Z_4, Z_7)}\widetilde{\Phi}_{1_{(\ell, \ell)}}(Z_6, \xi ) \right)   \Phi_{1_{(r,r)}}(Z_1, Z_2)\\
 &=  \widetilde{\Phi}_{1_{(r, \ell)}}\left(Z_3, Z_5 R_{\Phi_{1_{(r,r)}}(Z_4, Z_7)}\widetilde{\Phi}_{1_{(\ell, \ell)}}(Z_6, \xi ) \right) b   \Phi_{1_{(r,r)}}(Z_1, Z_2)\\
  &=  \widetilde{\Phi}_{1_{(r, \ell)}}\left(Z_3, Z_5 R_{\Phi_{1_{(r,r)}}(Z_4, Z_7)}\widetilde{\Phi}_{1_{(\ell, \ell)}}(Z_6, \xi ) \right)  \Phi_{1_{(r,r)}}(Z_1, Z_2R_b)\\
  &= \widetilde{\Phi}_\pi(Z_1, Z_2R_b , Z_3, Z_4 \ldots, Z_7, \xi) .
\end{align*}
If $\xi$ were replaced with $L_b \xi$, then clearly the $L_b$ can be moved to give $Z_6 L_b$ via (ii) with a $1_\chi$  as the expression $\widetilde{\Phi}_{1_{(\ell, \ell)}}(Z_6, \xi )$ is present.  If $\xi$ were replaced with $R_b \xi$, then 
\begin{align*}
\widetilde{\Phi}_\pi(Z_1, Z_2, \ldots, Z_7, R_b\xi) &= \widetilde{\Phi}_{1_{(r, \ell)}}\left(Z_3, Z_5 R_{\Phi_{1_{(r,r)}}(Z_4, Z_7)}\widetilde{\Phi}_{1_{(\ell, \ell)}}(Z_6, R_b\xi ) \right)   \Phi_{1_{(r,r)}}(Z_1, Z_2)\\
&= \widetilde{\Phi}_{1_{(r, \ell)}}\left(Z_3, Z_5 R_{\Phi_{1_{(r,r)}}(Z_4, Z_7)} \left(  \widetilde{\Phi}_{1_{(\ell, \ell)}}(Z_6, \xi )b  \right) \right)   \Phi_{1_{(r,r)}}(Z_1, Z_2)\\
&= \widetilde{\Phi}_{1_{(r, \ell)}}\left(Z_3, Z_5 R_{\Phi_{1_{(r,r)}}(Z_4, Z_7)} R_b  \widetilde{\Phi}_{1_{(\ell, \ell)}}(Z_6, \xi )  \right)   \Phi_{1_{(r,r)}}(Z_1, Z_2)\\
&= \widetilde{\Phi}_{1_{(r, \ell)}}\left(Z_3, Z_5 R_{b\Phi_{1_{(r,r)}}(Z_4, Z_7)}  \widetilde{\Phi}_{1_{(\ell, \ell)}}(Z_6, \xi )  \right)   \Phi_{1_{(r,r)}}(Z_1, Z_2)\\
&= \widetilde{\Phi}_{1_{(r, \ell)}}\left(Z_3, Z_5 R_{\Phi_{1_{(r,r)}}(Z_4, Z_7R_b)}  \widetilde{\Phi}_{1_{(\ell, \ell)}}(Z_6, \xi )  \right)   \Phi_{1_{(r,r)}}(Z_1, Z_2)\\
&= \widetilde{\Phi}_\pi(Z_1, \ldots, Z_6,  Z_7R_b, \xi),
\end{align*}
as desired.  Thus the result follows.
\end{proof}

\subsection*{Analytical Operator-Valued Bi-Moment Function}

We will now construct the analytical extension of the operator-valued bi-moment function  via recursion and the map $\widetilde{E}:L_2(A,\tau)\to L_2(B,\tau_B)$ from Section \ref{sec:Analytical-B-B-NCPS}.  Note that the recursive process in the following definition is different than that from \cite{CNS2015-1}*{Definition 5.1} and \cite{GS2018}*{Definition 4.4}, in order  to facilitate the introduction of the $L_2(A, \tau)$ element.  The same recursive process could have been used in \cite{CNS2015-1}*{Definition 5.1} and \cite{GS2018}*{Definition 4.4}, as these processes are equivalent in those settings.  Note we use $\Psi$ in the following to avoid confusion with $\widetilde{E}$ in Section \ref{sec:Analytical-B-B-NCPS}, although $\Psi$ is a multi-entry extension of $\widetilde{E}$.

\begin{definition}\label{enhancedbimomentdef}
Let $(A,E,\varepsilon,\tau)$ be an analytical $B$-$B$-non-commutative probability space. The \textit{analytical bi-moment function}
\[
\Psi:\bigcup_{n\in\bN}\bigcup_{\chi\in \{\ell,r\}^n}\BNC(\chi)\times A_{\chi(1)}\times\ldots\times A_{\chi(n-1)}\times L_2(A,\tau)\to L_2(B,\tau_B)
\]
is defined recursively as follows:  Let $n\in\bN$, $\chi\in\{\ell,r\}^n$, $\pi\in\BNC(\chi)$, $\xi\in L_2(A,\tau)$, and $Z_k \in A_{\chi(k)}$. 
\begin{itemize}
\item If $\pi=1_\chi$, then 
\[
\Psi_{1_\chi}(Z_1,Z_2,\ldots,Z_{n-1},\xi) = \widetilde{E}(Z_1 Z_2\cdots Z_{n-1}\xi).
\]
\item If $\pi\neq 1_\chi$, let $V$ be the block in $\pi$ such that $n\in V$. We divide discussion into two cases:
\begin{itemize}
\item Suppose that  $\min_{\pleq}V=\sx(1)$ and $\max_{\pleq}V=\sx(n)$ and let 
\[
p = \min_{\pleq}\{i\in\{1,\ldots,n\} \, \mid \,  i\notin V\}, \quad q = \min_{\pleq}\{j\in V \, \mid \,  p \ple j\} \qand m = \max_{\pleq}\{i \in V \, \mid \, i \prec_\chi p \}.
\]
Set
\[
W = \{i\in\{1,\ldots,n\} \, \mid \,  p \pleq  i   \ple  q\}.
\]
Note by construction and the fact that $\pi \in \BNC(\chi)$ that $W$ is equal to a union of blocks of $\pi$ and $\chi(p)=\chi(j)$ for all $j \in W$.   Thus we define
\[
\Psi_\pi(Z_1,\ldots,Z_{n-1},\xi) = \begin{cases}
       \Psi_{\pi|_{W^\mathsf{c}}}\left({(Z_1,\ldots,Z_{p-1}, Z_m L_{E_{\pi|_W}({(Z_1,\ldots,Z_{n-1},\xi)|}_W)} ,\ldots,Z_{n-1},\xi)|}_{W^\mathsf{c}}\right)  &\text{if }\chi(p)=\ell,\\
     \Psi_{\pi|_{W^\mathsf{c}}}\left({(Z_1,\ldots,Z_{q-1}, Z_q R_{E_{\pi|_W}({(Z_1,\ldots,Z_{n-1},\xi)|}_W)} ,\ldots,Z_{n-1},\xi)|}_{W^\mathsf{c}}\right)  &  \text{if } \chi(p)=r. \\
     \end{cases}
\]
[Note as $n\notin W$, the quantity $E_{\pi|_W}({(Z_1,\ldots,Z_{n-1},\xi)|}_W)$ is always a  well-defined element of $B$ in this case.  Note in the case that $\chi(p) = \ell$ that $m \ple p \ple n$ and thus $Z_m \neq \xi$, so $\Psi_\pi(Z_1,\ldots,Z_{n-1},\xi)$ is well-defined.  Also,  in the case when $\chi(p)=r$ note that $n\ple q$ and thus $Z_q\neq\xi$, so $\Psi_\pi(Z_1,\ldots,Z_{n-1},\xi)$ is well-defined.]
\item Otherwise, set
\[
\widetilde{V} = \left\{i\in\{1,\ldots,n\}\, \left| \, \min_{\pleq}V \pleq i \pleq \max_{\pleq}V \right.\right\}.
\]
Note $\widetilde{V}$ is a proper subset of $\{1,\ldots,n\}$ that is a union of blocks of $\pi$ and is such that $n\in V\subseteq\widetilde{V}$. For $q=\max_\leq {\widetilde{V}^\mathsf{c}}$ and define
\[
\Psi_\pi(Z_1,\ldots,Z_{n-1},\xi) =  \Psi_{\pi|_{\widetilde{V}^\mathsf{c}}}\left({(Z_1,\ldots,Z_{q-1}, Z_q \Psi_{\pi|_{\widetilde{V}}}{({(Z_1,\ldots,Z_{n-1},\xi)|}_{\widetilde{V}})} ,\ldots,Z_{n-1},\xi)|}_{\widetilde{V}^\mathsf{c}}\right).
\]
[Note that the quantity $\Psi_{\pi|_{\widetilde{V}}}{({(Z_1,\ldots,Z_{n-1},\xi)|}_{\widetilde{V}})}$ is a well-defined element of $L_2(B,\tau_B)$ due to the recursive nature of our definition. Moreover,the last element of the sequence
\[
{(Z_1,\ldots,Z_{q-1}, Z_q \Psi_{\pi|_{\widetilde{V}}}{({(Z_1,\ldots,Z_{n-1},\xi)|}_{\widetilde{V}})} ,\ldots,Z_{n-1},\xi)|}_{\widetilde{V}^\mathsf{c}}
\]
is equal to $Z_q \Psi_{\pi|_{\widetilde{V}}}{({(Z_1,\ldots,Z_{n-1},\xi)|}_{\widetilde{V}})}$, which is an element of $L_2(A,\tau)$ so $\Psi_\pi(Z_1,\ldots,Z_{n-1},\xi) $ is well-defined.]
\end{itemize}
\end{itemize}
\end{definition}

To aid in the comprehension of Definition \ref{enhancedbimomentdef}, we provide an example using of bi-non-crossing diagrams to show the recursive construction.  We note that $\xi$ will always appear last in a sequence of operators and is an element of $L_2(A,\tau)$ and thus neither a left nor right operator.  As such, we treat it as neither.  This is reminiscent of \cite{S2016-2}*{Lemma 2.17} where it does not matter whether we treat the last operator in a list as a left or as a right operator, and of \cite{CS2020}*{Lemma 2.29 and Proposition 2.30} where the last entry can be a mixture of left and right operators.

\begin{example}\label{exampledef}
Let $\chi\in\{\ell,r\}^{12}$ be such that $\chi^{-1}(\{\ell\})=\{1,5,8,9,11,12\}$, let $\xi\in L_2(A,\tau)$, let $Z_k\in A_{\chi(k)}$, and let $\pi\in\BNC(\chi)$ be the partition with blocks
\[
V_1 = \{1,3\}, \quad V_2 = \{2\}, \quad V_3 = \{4,5,11,12\}, \quad V_4 = \{6,10\}, \quad V_5 = \{7\}, \qand V_6 = \{8,9\}.
\]
To compute $\Psi_\pi (Z_1, \ldots, Z_{11}, \xi)$, we note the second part of the second step of the recursive definition from Definition \ref{enhancedbimomentdef} applies first.  In particular
\[
\widetilde{V} = \bigcup_{k=3}^6 V_k.
\]  
Thus, if
\[
X = \Psi_{\pi|_{\widetilde{V}}}(Z_4,Z_5,Z_6,Z_7,Z_8,Z_9,Z_{10},Z_{11},\xi),
\]
then 
\[
\Psi_\pi(Z_1,\ldots,Z_{11},\xi) = \Psi_{\pi|_{{\widetilde{V}^{\mathsf{c}}}}}(Z_1,Z_2,Z_3 X).
\]
Diagrammatically, this first reduction is seen as follows:	
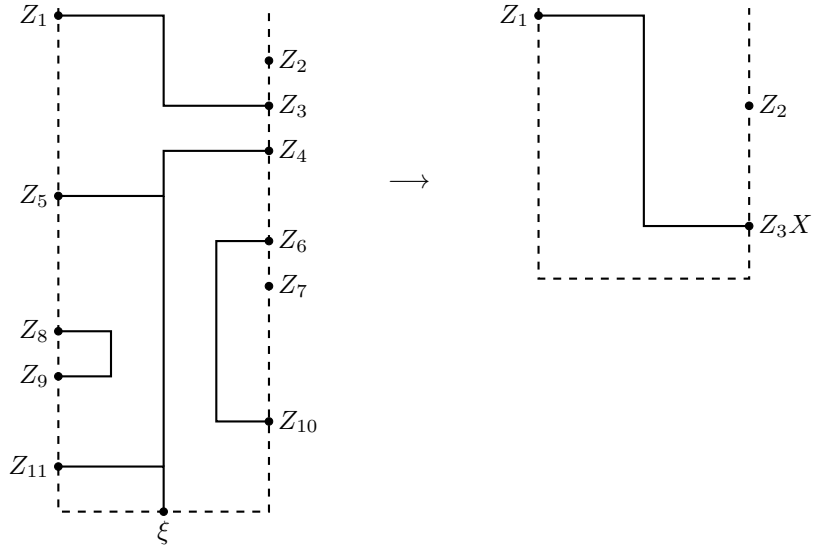
\begin{align*}
\begin{tikzpicture}[baseline]
			\draw[thick,dashed] (-.5,2.4) -- (-.5,-4.3) -- (2.3,-4.3) -- (2.3,2.4);
\node[left] at (-.5,2.3) {$Z_1$};
			\draw[black,fill=black] (-.5,2.3) circle (0.05);
\node[right] at (2.3,1.7) {$Z_2$};
\draw[black,fill=black] (2.3,1.7) circle (0.05);	
\node[right] at (2.3,1.1) {$Z_3$};
\draw[black,fill=black] (2.3,1.1) circle (0.05);	
\node[right] at (2.3,0.5) {$Z_4$};
\draw[black,fill=black] (2.3,0.5) circle (0.05);	
\node[left] at (-0.5,-0.1) {$Z_5$};
\draw[black,fill=black] (-0.5,-0.1) circle (0.05);	
\node[right] at (2.3,-0.7) {$Z_6$};
\draw[black,fill=black] (2.3,-0.7) circle (0.05);	
\node[right] at (2.3,-1.3) {$Z_7$};
\draw[black,fill=black] (2.3,-1.3) circle (0.05);	
\node[left] at (-0.5,-1.9) {$Z_8$};
\draw[black,fill=black] (-0.5,-1.9) circle (0.05);	
\node[left] at (-0.5,-2.5) {$Z_9$};
\draw[black,fill=black] (-0.5,-2.5) circle (0.05);	
\node[right] at (2.3,-3.1) {$Z_{10}$};
\draw[black,fill=black] (2.3,-3.1) circle (0.05);	
\node[left] at (-0.5,-3.7) {$Z_{11}$};
\draw[black,fill=black] (-0.5,-3.7) circle (0.05);	
\node[anchor=north] at (0.9,-4.3) {$\xi$};
\draw[black,fill=black] (0.9,-4.3) circle (0.05);
\draw[thick, black] (-.5,2.3) -- (0.9,2.3) -- (0.9,1.1) -- (2.3,1.1);
\draw[thick, black] (2.3,0.5) -- (0.9,0.5) -- (0.9,-0.1) -- (-0.5,-0.1);
\draw[thick, black] (-0.5,-3.7) -- (0.9,-3.7) -- (0.9,-4.3);
\draw[thick, black] (0.9,-0.1) -- (0.9,-3.7);
\draw[thick, black] (-0.5,-1.9) -- (0.2,-1.9) -- (0.2,-2.5) -- (-0.5,-2.5);
\draw[thick, black] (2.3,-0.7) -- (1.6,-0.7) -- (1.6,-3.1) -- (2.3,-3.1);
\end{tikzpicture}
	\qquad\longrightarrow	\qquad 
\begin{tikzpicture}[baseline]
		\draw[thick,dashed] (-.5,2.4) -- (-.5,-1.2) -- (2.3,-1.2) -- (2.3,2.4);
\node[left] at (-.5,2.3) {$Z_1$};
			\draw[black,fill=black] (-.5,2.3) circle (0.05);
\node[right] at (2.3,1.1) {$Z_2$};
\draw[black,fill=black] (2.3,1.1) circle (0.05);	
\node[right] at (2.3,-0.5) {$Z_3 X$};
\draw[black,fill=black] (2.3,-0.5) circle (0.05);	
\draw[thick, black] (-.5,2.3) -- (0.9,2.3) -- (0.9,-0.5) -- (2.3,-0.5);
\end{tikzpicture}
	\qquad
	\end{align*}
Note
\[
\Psi_{\pi|_{{\widetilde{V}^{\mathsf{c}}}}}(Z_1,Z_2,Z_3 X) = \Psi_{\pi|_{V_2}}(Z_2 \Psi_{\pi|_{V_1}}(Z_1,Z_3X))= \widetilde{E}\left(Z_2 \widetilde{E}(Z_1Z_3 X)\right),
\]	
where the first equality holds by the same recursive idea, whereas the second equality holds by the first step of Definition \ref{enhancedbimomentdef}.

When computing the value of $X$, the minimal and maximal elements of $\{4,5,\ldots,11,12\}$ in the ${\chi |}_{\widetilde{V}}$-order are $5$ and $4$ respectively and the block that contains the index corresponding to $\xi$ contains both 5 and 4.  Thus the first part of the second step of Definition \ref{enhancedbimomentdef} should be used.  The algorithm in Definition \ref{enhancedbimomentdef} then calculates the value of $X$ by ``stripping out" the $\chi$-intervals $V_6$ and $V_4\cup V_5$ successively and this is seen via the following two diagrammatic reductions:
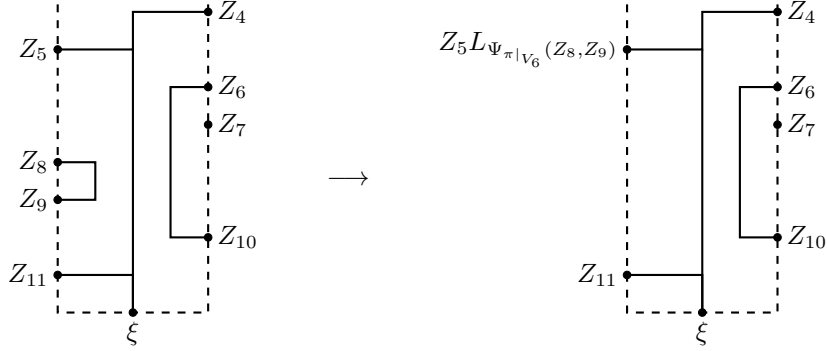
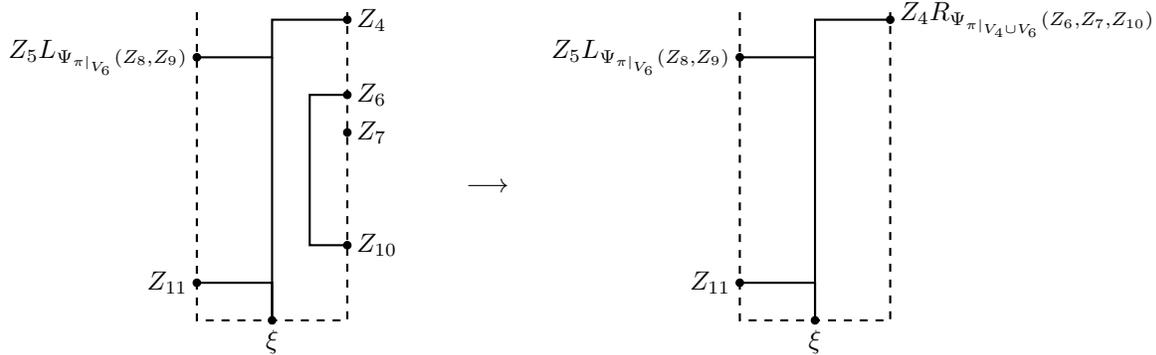
\begin{align*}
\begin{tikzpicture}[baseline]
			\draw[thick,dashed] (-.5,2.4) -- (-.5,-1.7) -- (1.5,-1.7) -- (1.5,2.4);
\node[right] at (1.5,2.3) {$Z_4$};
			\draw[black,fill=black] (1.5,2.3) circle (0.05);
			\node[left] at (-.5,1.8) {$Z_5$};
			\draw[black,fill=black] (-.5,1.8) circle (0.05);
			\node[right] at (1.5,1.3) {$Z_6$};
			\draw[black,fill=black] (1.5,1.3) circle (0.05);
			\node[right] at (1.5,0.8) {$Z_7$};
			\draw[black,fill=black] (1.5,0.8) circle (0.05);
			\node[left] at (-.5,0.3) {$Z_8$};
			\draw[black,fill=black] (-.5,0.3) circle (0.05);
			\node[left] at (-.5,-0.2) {$Z_9$};
			\draw[black,fill=black] (-.5,-0.2) circle (0.05);
			\node[right] at (1.5,-0.7) {$Z_{10}$};
			\draw[black,fill=black] (1.5,-0.7) circle (0.05);
			\node[left] at (-.5,-1.2) {$Z_{11}$};
			\draw[black,fill=black] (-.5,-1.2) circle (0.05);
			\node[anchor=north] at (0.5,-1.7) {$\xi$};
			\draw[black,fill=black] (0.5,-1.7) circle (0.05);
			\draw[thick, black] (1.5,2.3) -- (0.5,2.3) -- (0.5,1.8) -- (-0.5,1.8);
			\draw[thick, black] (-0.5,-1.2) -- (0.5,-1.2) -- (0.5,-1.7);
			\draw[thick, black] (0.5,-1.7) -- (0.5,1.8);
			\draw[thick, black] (-.5,0.3) -- (0,0.3) -- (0,-0.2) -- (-0.5,-0.2);
			\draw[thick, black] (1.5,1.3) -- (1,1.3) -- (1,-0.7) -- (1.5,-0.7);
\end{tikzpicture}
	\qquad
	\longrightarrow	\qquad
	\begin{tikzpicture}[baseline]
			\draw[thick,dashed] (-.5,2.4) -- (-.5,-1.7) -- (1.5,-1.7) -- (1.5,2.4);
\node[right] at (1.5,2.3) {$Z_4$};
			\draw[black,fill=black] (1.5,2.3) circle (0.05);
			\node[left] at (-.5,1.8) {$Z_5 L_{\Psi_{\pi|_{V_6}}(Z_8,Z_9)} $};
			\draw[black,fill=black] (-.5,1.8) circle (0.05);
			\node[right] at (1.5,1.3) {$Z_6$};
			\draw[black,fill=black] (1.5,1.3) circle (0.05);
			\node[right] at (1.5,0.8) {$Z_7$};
			\draw[black,fill=black] (1.5,0.8) circle (0.05);
			\node[right] at (1.5,-0.7) {$Z_{10}$};
			\draw[black,fill=black] (1.5,-0.7) circle (0.05);
			\node[left] at (-.5,-1.2) {$Z_{11}$};
			\draw[black,fill=black] (-.5,-1.2) circle (0.05);
			\node[anchor=north] at (0.5,-1.7) {$\xi$};
			\draw[black,fill=black] (0.5,-1.7) circle (0.05);
			\draw[thick, black] (1.5,2.3) -- (0.5,2.3) -- (0.5,1.8) -- (-0.5,1.8);
			\draw[thick, black] (-0.5,-1.2) -- (0.5,-1.2) -- (0.5,-1.7);
			\draw[thick, black] (0.5,-1.7) -- (0.5,1.8);
			\draw[thick, black] (1.5,1.3) -- (1,1.3) -- (1,-0.7) -- (1.5,-0.7);
\end{tikzpicture}
	\qquad
	\end{align*}
and 
\begin{align*}
\begin{tikzpicture}[baseline]
			\draw[thick,dashed] (-.5,2.4) -- (-.5,-1.7) -- (1.5,-1.7) -- (1.5,2.4);
\node[right] at (1.5,2.3) {$Z_4$};
			\draw[black,fill=black] (1.5,2.3) circle (0.05);
			\node[left] at (-.5,1.8) {$Z_5L_{\Psi_{\pi|_{V_6}}(Z_8,Z_9)} $};
			\draw[black,fill=black] (-.5,1.8) circle (0.05);
			\node[right] at (1.5,1.3) {$Z_6$};
			\draw[black,fill=black] (1.5,1.3) circle (0.05);
			\node[right] at (1.5,0.8) {$Z_7$};
			\draw[black,fill=black] (1.5,0.8) circle (0.05);
			\node[right] at (1.5,-0.7) {$Z_{10}$};
			\draw[black,fill=black] (1.5,-0.7) circle (0.05);
			\node[left] at (-.5,-1.2) {$Z_{11}$};
			\draw[black,fill=black] (-.5,-1.2) circle (0.05);
			\node[anchor=north] at (0.5,-1.7) {$\xi$};
			\draw[black,fill=black] (0.5,-1.7) circle (0.05);
			\draw[thick, black] (1.5,2.3) -- (0.5,2.3) -- (0.5,1.8) -- (-0.5,1.8);
			\draw[thick, black] (-0.5,-1.2) -- (0.5,-1.2) -- (0.5,-1.7);
			\draw[thick, black] (0.5,-1.7) -- (0.5,1.8);
			\draw[thick, black] (1.5,1.3) -- (1,1.3) -- (1,-0.7) -- (1.5,-0.7);
\end{tikzpicture}
	\qquad
	\longrightarrow	 \quad 
	\begin{tikzpicture}[baseline]
			\draw[thick,dashed] (-.5,2.4) -- (-.5,-1.7) -- (1.5,-1.7) -- (1.5,2.4);
\node[right] at (1.5,2.3) {$Z_4 R_{\Psi_{\pi|_{V_4\cup V_6}}(Z_6,Z_7,Z_{10})}$};
			\draw[black,fill=black] (1.5,2.3) circle (0.05);
			\node[left] at (-.5,1.8) {$Z_5 L_{\Psi_{\pi|_{V_6}}(Z_8,Z_9)}$};
			\draw[black,fill=black] (-.5,1.8) circle (0.05);
			\node[left] at (-.5,-1.2) {$ Z_{11}$};
			\draw[black,fill=black] (-.5,-1.2) circle (0.05);
			\node[anchor=north] at (0.5,-1.7) {$\xi$};
			\draw[black,fill=black] (0.5,-1.7) circle (0.05);
			\draw[thick, black] (1.5,2.3) -- (0.5,2.3) -- (0.5,1.8) -- (-0.5,1.8);
			\draw[thick, black] (-0.5,-1.2) -- (0.5,-1.2) -- (0.5,-1.7);
			\draw[thick, black] (0.5,-1.7) -- (0.5,1.8);		
\end{tikzpicture}
	\qquad
	\end{align*}
It is readily verified using the fact that the operator-valued bi-free  moment function is bi-multiplicative that $E_{\pi|_{V_6}}(Z_8,Z_9)= E(Z_8Z_9)$ and $E_{\pi|_{V_4\cup V_6}}(Z_6,Z_7,Z_{10}) = E(Z_6R_{E(Z_7)}Z_{10})$. Thus, using the fact that $\pi|_{V_3} = 1_{{\chi |}_{V_3}}$, the first step in Definition \ref{enhancedbimomentdef} yields
\[
X = \widetilde{E}\left(Z_4 R_{E(Z_6R_{E(Z_7)}Z_{10})} Z_5 L_{E(Z_8Z_9)} Z_{11}\xi \right).
\]
Hence
\begin{align*}
\Psi_\pi(Z_1,\ldots,Z_{11},\xi)  &= \widetilde{E}\left(Z_2 \widetilde{E}\left(Z_1Z_3 \widetilde{E}\left(Z_4 R_{E\left(Z_6R_{E(Z_7)}Z_{10}\right)} Z_5 L_{E(Z_8Z_9)} Z_{11}\xi\right)\right)\right)\\
&=  \widetilde{E}\left(Z_1Z_3 \widetilde{E}\left(Z_4 R_{E\left(Z_6R_{E(Z_7)}Z_{10}\right)} Z_5 L_{E(Z_8Z_9)} Z_{11}\xi\right)\right)E(Z_2),
\end{align*}
with the last equality following from Proposition \ref{propE}.
\end{example}

Before investigating the bi-multiplicative properties inherited by the analytical bi-moment function, we note it is truly an extension of the operator-valued bi-moment function.

\begin{theorem}\label{thmPsi=E}
Let $(A,E,\varepsilon,\tau)$ be an analytical $B$-$B$-non-commutative probability space.  For any $n\in\bN$, $\chi\in\{\ell,r\}^n$, $\pi\in\BNC(\chi)$, and $Z_k\in A_{\chi(k)}$,
\[
\Psi_\pi(Z_1,\ldots,Z_{n-1},Z_n+N_\tau) = E_\pi(Z_1,\ldots,Z_{n-1},Z_n) + N_{\tau_B}.
\]
\end{theorem}
\begin{proof}
Note that each step in the recursive definition of Definition \ref{enhancedbimomentdef} is a step that can be performed to the operator-valued bi-free moment function as the operator-valued bi-free moment function is bi-multiplicative (see Definition \ref{defn:bi-multiplicative-function}).  Therefore, as Proposition \ref{orthproj} implies that 
\[
\widetilde{E}(a) = E(a) + N_{\tau_B}
\]
for all $a \in A$, by applying the same recusive properties to $E_\pi(Z_1,\ldots,Z_{n-1},Z_n)$ as used to compute $\Psi_\pi(Z_1,\ldots,Z_{n-1},Z_n+N_\tau)$, the result follows.
\end{proof}

Like with the construction of the operator-valued bi-free moment function in \cite{CNS2015-1}, although the construction of the analytical bi-moment function is done using specific rules from the operator-valued bi-free moment function in a specific order, we desire more flexibility in the reductions that can be done and the order they can be done in.  In particular, we desire to show that the analytical bi-moment function is an analytic extension of the operator-valued bi-free moment function.   

The main ideas used to prove this are similar to those utilized in the proof of \cite{CNS2015-1}*{Theorem 5.1.4} and hence we shall be concerned with demonstrating that the inclusion of the $L_2(A,\tau)$ term and the slightly modified recursive definition are not issue and pose next to no changes.  In particular,  it may appear that $\Psi$ behaves differently than the   operator-valued bi-free moment function   as entries in $L_2(A, \tau)$ can also act as mixtures of left and right operators, which was not dealt with in \cite{CNS2015-1}.  However, using the properties of $\widetilde{E}$ as developed in  Proposition \ref{propE}, one familiar with \cite{CNS2015-1} can easily see that the desired results will hold with simple adaptations.  We note that similar adaptations were done in \cite{GS2018} without issue.

When examining the proof of the following, Example \ref{exampledef} serves as a good example to keep in mind, just as Example \ref{exam:analytic-bi-mult-functions-work} aided in comprehending why analytic extensions of bi-multiplicative functions work.

\begin{theorem}\label{thm:enhanced1}
Let $(A,E,\varepsilon,\tau)$ be an analytical $B$-$B$-non-commutative probability space and let 
\[
\Psi:\bigcup_{n\in\bN}\bigcup_{\chi\in \{\ell,r\}^n}\BNC(\chi)\times A_{\chi(1)}\times\ldots\times A_{\chi(n-1)}\times L_2(A,\tau)\to L_2(B,\tau_B)
\]
be the analytical bi-moment function.  Then $\Psi$ is an analytically extension of the operator-valued bi-free moment function.
\end{theorem}

Clearly the map $\Psi_{1_\chi}$ is $\C$-multilinear and it does not matter whether $\chi(n) = \ell $ or $\chi(n) = r$.  A straightforward induction argument using the definition of $\Psi$ shows that the map $\Psi_\pi$ will be $\C$-multilinear.  Thus we focus on the remaining four properties.

\begin{proof}[Proof of Theorem \ref{thm:enhanced1} property (i)]
This immediately follows from parts (ii) and (iii) of Proposition \ref{propE}.
\end{proof}

\begin{proof}[Proof of Theorem \ref{thm:enhanced1} property (ii)]
To see (ii), fix $n\in\bN$, $\chi\in\{\ell,r\}^n$, $\pi\in\BNC(\chi)$, $\xi\in L_2(A,\tau)$, $b\in B$, and $Z_k\in A_{\chi(k)}$, and let $p$ and $q$ be as in the statement of (ii).  In the case that $\chi(p) = \ell$, note that
\[
\Psi_{1_\chi}(Z_1,\ldots,Z_{p-1},L_b Z_p, Z_{p+1}, \ldots,Z_{n-1},\xi) = \widetilde{E}(Z_1 \cdots Z_{p-1}L_bZ_p Z_{p+1} \cdots Z_{n-1} \xi).
\]

If $q \neq \infty$, then $Z_{q+1},\ldots, Z_{p-1} \in A_r$ and thus commute with $L_b$.  Hence
\begin{align*}
\Psi_{1_\chi}(Z_1,\ldots,Z_{p-1},L_b Z_p, Z_{p+1}, \ldots,Z_{n-1},\xi) &= \widetilde{E}(Z_1 \cdots Z_{q-1}Z_qL_b Z_{q+1} \cdots Z_{n-1} \xi) \\
&= \Psi_{1_\chi}(Z_1,\ldots,Z_{q-1},Z_q L_b, Z_{q+1}, \ldots,Z_{n-1},\xi)
\end{align*}
(and note $Z_qL_b \in A_\ell$).

If $q = \infty$, then $Z_1, \ldots, Z_{p-1} \in A_r$ and thus commute with $L_b$.  Hence
\begin{align*}
\Psi_{1_\chi}(Z_1,\ldots,Z_{p-1},L_b Z_p, Z_{p+1}, \ldots,Z_{n-1},\xi) &= \widetilde{E}(L_b Z_1  \cdots Z_{n-1} \xi) \\
&= b \widetilde{E}(Z_1  \cdots Z_{n-1} \xi) \\
&= b \Psi_{1_\chi}(Z_1,\ldots,Z_{n-1},\xi)
\end{align*}
by Proposition \ref{propE}.  The case $\chi(p) = r$ is similar.
\end{proof}

\begin{proof}[Proof of Theorem \ref{thm:enhanced1} property (iii)]
To highlight how the proof works, we begin with the case that $\pi$ consists of exactly three blocks that are $\chi$-intervals and $n$ is contained in the middle block under the $\chi$-ordering. Suppose that $\pi = \{V_1,V_2,V_3\}$ and thus there exists $i\in\{2,\ldots,n-1\}$ and $j\in\{0,\ldots,n-2\}$ such that
\begin{align*}
V_1 &= \{\sx(1),\sx(2),\ldots,\sx(i-1)\},\\
V_2 &= \{\sx(i),\sx(i+1),\ldots,\sx(i+j)\}, \text{ and } \\
V_3 &=\{\sx(i+j+1),\sx(i+j+2),\ldots,\sx(n)\}.
\end{align*}
Thus $n=\sx(k)$ for some  $i\leq k\leq i+j$.  This implies that $\chi(p)=\ell$ for all $p\in V_1$ and $\chi(p)=r$ for all $p\in V_3$.   If $W=V_1\cup V_3$, observe that the definition of the permutation $\sx$ yields that
\[
q = \max_\leq W=\max_\leq \{\sx(i-1),\sx(i+j+1)\}.
\]

Consider the case $q=\sx(i-1)$ so that $q \in V_1$, and let $m= \sx(i+j+1)$.  Notice that if
\[
X = \Psi_{\pi|_{V_2}}({(Z_1,\ldots,Z_{n-1},\xi)|}_{V_2}),
\]
then by    Definition \ref{enhancedbimomentdef} we obtain that
\[
\Psi_\pi(Z_1,\ldots,Z_{n-1},\xi) = \Psi_{\pi|_W}({(Z_1,\ldots,Z_{q-1},Z_qX,\ldots,Z_{n-1},\xi)|}_W).
\]
Thus, if
\[
Y = \Psi_{\pi|_{V_1}}({(Z_1,\ldots,Z_{q-1},Z_qX,\ldots,Z_{n-1},\xi)|}_{V_1}),
\]
then  again  Definition \ref{enhancedbimomentdef} implies that
\begin{align*}
\Psi_\pi(Z_1,\ldots,Z_{n-1},\xi) &= \Psi_{\pi|_{V_3}}({(Z_1,\ldots,Z_{m-1},Z_mY,\ldots,Z_{n-1},\xi)|}_{V_3})\\
&= \widetilde{E}\left(Z_{\sx(n)} Z_{\sx(n-1)}\ldots Z_{\sx(i+j+1)} Y\right).
\end{align*}
Therefore, since $Z_{\sx(n)} Z_{\sx(n-1)}\ldots Z_{\sx(i+j+1)}  \in A_r$, Proposition \ref{propE} implies that
\[
\Psi_\pi(Z_1,\ldots,Z_{n-1},\xi) =  YE\left(Z_{\sx(n)} Z_{\sx(n-1)}\ldots Z_{\sx(i+j+1)}\right) = Y E_{\pi|_{V_3}}((Z_1,\ldots,Z_{n-1},\xi)|_{V_3}).
\]
Since $Z_{\sx(1)}, Z_{\sx(2)},\ldots, Z_{\sx(i-1)} \in A_r$, Proposition \ref{propE} implies that
\begin{align*}
Y = E(Z_{\sx(1)} Z_{\sx(2)}\ldots Z_{\sx(i-1)})X = E_{\pi|_{V_1}}((Z_1,\ldots,Z_{n-1},\xi)|_{V_1})X,
\end{align*}
so the result follows. Note the case $q=\sx(i+j+1)$ is handled similarly  by interchanging the orders of the $\chi$- intervals $V_1$ and $V_3$ .  This argument can be extended via induction to any bi-non-crossing partition $\pi$ all of whose blocks are $\chi$-interval.  

By the same argument as \cite{CNS2015-1}*{Lemma 5.2.1}, one need only consider the case in property (iii) that for each $\chi$-interval, the $\chi$-maximal and $\chi$-minimal elements belong to the same block.  When using the recursive procedure in Definition \ref{enhancedbimomentdef} to reduce $\Psi_\pi$, one of the $\chi$-intervals (which will either be entirely on the left or entirely on the right) will have the $L_2(A, \tau)$ term added to the last entry as above.  This $L_2(A, \tau)$ entry can be pulled out on the appropriate side leaving only the bi-moment function expression for the $\chi$-interval, which can be undone as usual.  By repetition, eventually all that remains is the expression for the $\chi$-interval containing $n$ as desired.
\end{proof}

\begin{proof}[Proof of Theorem \ref{thm:enhanced1} property (iv)]
The proof that property (iv) holds for the operator-valued bi-free moment function is one of the longest of \cite{CNS2015-1} consisting of \cite{CNS2015-1}*{Lemma 5.3.1},  \cite{CNS2015-1}*{Lemma 5.3.2},  \cite{CNS2015-1}*{Lemma 5.3.3}, and \cite{CNS2015-1}*{Lemma 5.3.4}.  As such, we will only sketch the details here.

First one proceeds to show that properties (i) and (ii) of Definition \ref{defn:anal-bi-multiplicative} hold for $\Psi$ when $1_\chi$ is replaced with an arbitrary bi-non-crossing partition.  This effectively makes use of the same arguments as in Lemma \ref{lem:analytic-properties-extend-from-1-to-pi}; that is, one uses the recursive algorithm to reduce down and then note the proofs of properties (i) and (ii) above still apply and lets one move elements around as needed.  In particular, the same arguments used in \cite{CNS2015-1}*{Lemma 5.3.2} and  \cite{CNS2015-1}*{Lemma 5.3.3} transfer with the use of Proposition \ref{propE}.

Next, using property (iii), we need only prove property (iv) under the assumption that $s_\chi(1)$ and $s_\chi(n)$ are in the same block $W_0$ of $W$.  One then follows many of the same ideas as \cite{CNS2015-1}*{Lemma 5.3.1} and \cite{CNS2015-1}*{Lemma 5.3.4} by applying the recursive definition from Definition \ref{enhancedbimomentdef}, moving around the appropriate $B$-elements using the more general (i) and (ii), and combining the appropriate elements using (iii) as needed.
\end{proof}

\subsection*{Analytical Operator-Valued Bi-Free Cumulant Function}

By convolving the analytical bi-moment function with the bi-non-crossing M\"{o}bius function, we obtain the following which is essential to our study of conjugate variables in the subsequent section.

\begin{definition}\label{encumdef}
Let $(A,E,\varepsilon,\tau)$ be an analytical $B$-$B$-non-commutative probability space and denote by $\Psi$ the analytical bi-moment function. The \textit{analytical bi-cumulant function}
\[
\widetilde{\kappa}: \bigcup_{n\in\bN}\bigcup_{\chi\in \{\ell,r\}^n}\BNC(\chi)\times A_{\chi(1)}\times\ldots\times A_{\chi(n-1)}\times L_2(A,\tau)\to L_2(B,\tau_B)
\]
is defined by
\[
\widetilde{\kappa}_\pi(Z_1,\ldots,Z_{n-1},\xi) = 
\sum_{\substack{\sigma\in\BNC(\chi)\\
\sigma\leq\pi\\}}\Psi_\sigma(Z_1,\ldots,Z_{n-1},\xi)\mu_{\BNC}(\sigma,\pi),
\]
for all $n\in\bN$, $\chi\in\{\ell,r\}^n$, $\pi\in\BNC(\chi)$, $\xi\in L_2(A,\tau)$, and $Z_k\in A_{\chi(k)}$. 
\end{definition}
\begin{remark}
\begin{enumerate}[(i)]
\item In the case when $\pi=1_\chi$, we will denote the map $\widetilde{\kappa}_{1_\chi}$ simply by $\widetilde{\kappa}_\chi$.  By M\"{o}bius inversion, we obtain that
\[
\Psi_\sigma(Z_1,\ldots,Z_n) = 
\sum_{\substack{\pi\in\BNC(\chi)\\
\pi\leq\sigma\\}}\widetilde{\kappa}_\pi(Z_1,\ldots,Z_{n-1},\xi)
\]
for all $n\in\bN$, $\chi\in \{\ell,r\}^n$, $\sigma\in\BNC(\chi)$, and $Z_k\in A_{\chi(k)}$.
\item In the case that $B = \bC$, the analytical bi-cumulant function are precisely the $L_2(A, \tau)$-valued bi-free cumulants that were used in \cite{CS2020}.
\end{enumerate}
\end{remark}

Unsurprisingly, the analytic extension of the operator-valued bi-free cumulant function lives upto its name.

\begin{theorem}\label{thm:analytic-cumulants-bi-multiplicative}
Let $(A,E,\varepsilon,\tau)$ be an analytical $B$-$B$-non-commutative probability space.  Then the analytical bi-cumulant function is the analytic extension of the operator-valued bi-free cumulant function.
\end{theorem}
\begin{proof}
Recall by \cite{CNS2015-1}*{Theorem 6.2.1} that the convolution of a bi-multiplicative function with a scalar-valued multiplicative function on the lattice of non-crossing partitions (e.g. the bi-free M\"{o}bius function) produces a bi-multiplicative function.  As the properties of an analytic extension of a bi-multiplicative function are analogous to those of a bi-multiplicative function, we obtain that the convolution of an analytic extension of a bi-multiplicative function with a scalar-valued multiplicative function on the lattice of non-crossing partitions (e.g. the bi-free M\"{o}bius function) produces the analytic extension of the corresponding bi-multiplicative function obtained via the same convolution.  Hence the result follows.
\end{proof}

\begin{theorem}\label{thmkappaPsi=kappaB}
Let $(A,E,\varepsilon,\tau)$ be an analytical $B$-$B$-non-commutative probability space.  For all $n\in\bN$, $\chi\in\{\ell,r\}^n$, $\pi\in\BNC(\chi)$, and $Z_k\in A_{\chi(k)}$, we have that
\[
\widetilde{\kappa}_\pi(Z_1,\ldots,Z_{n-1}, Z_n + N_{\tau}) = \kappa^B_\pi(Z_1,\ldots,Z_{n-1},Z_n) + N_{\tau_B}.
\]
\end{theorem}
\begin{proof}
By Theorem \ref{thmPsi=E}, we know that
\[
\widetilde{\Psi}_\pi(Z_1,\ldots,Z_{n-1}, Z_n + N_{\tau}) = E_\pi(Z_1,\ldots,Z_{n-1},Z_n) + N_{\tau_B}
\]
for all $\pi \in \BNC(\chi)$.  Therefore as $\widetilde{\kappa}$ and $\kappa^B$ are the convolution of $\widetilde{\Psi}$ and $E$ against the bi-free M\"{o}bius function respectively, the result follows.
\end{proof}

Of course, as \cite{CNS2015-1}*{Theorem 8.1.1} demonstrated that bi-freeness with amalagmation over $B$ is equivalent to the mixed operator-valued bi-free cumulants vanishing, Theorem \ref{thmkappaPsi=kappaB} immediately implies the following.

\begin{corollary}\label{cor:bi-free-is-vanishing-mixed-cumulants}
Let $(A,E,\varepsilon,\tau)$ be an analytical $B$-$B$-non-commutative probability space containing a family of pairs of $B$-algebras $\{(C_k,D_k)\}_{k\in K}$.  Consider the following two conditions:
\begin{enumerate}
\item The family $\{(C_k,D_k)\}_{k\in K}$ is bi-free with amalgamation over $B$ with respect to $E$.
\item For all $n\in\bN$, $\chi\in\{\ell,r\}^n$, $Z_1,\ldots,Z_n\in A$, and non-constant maps $\gamma:\{1,\ldots,n\}\to K$ such that
\[
Z_k\in \begin{cases}
       C_{\gamma(k)} &\text{if } \chi(k)=\ell \\
        D_{\gamma(k)} &\text{if }\chi(k) = r\\   
     \end{cases},
\]
it follows that
\[
\widetilde{\kappa}_\chi(Z_1,\ldots,Z_{n-1},Z_n + N_\tau)=0.
\]
\end{enumerate}
Then (1) implies (2).  In the case that $\tau_B : B \to \bC$ is faithful, (2) implies (1).
\end{corollary}
\begin{proof}
Note (1) implies (2) follows from \cite{CNS2015-1}*{Theorem 8.1.1} and Theorem \ref{thmkappaPsi=kappaB}.  In the case that $\tau_B$ is faithful, (2) immediately implies $\kappa^B_\pi(Z_1,\ldots,Z_{n-1},Z_n)=0$ where $\{Z_k\}^n_{k=1}$ are as in (2) via Theorem \ref{thmkappaPsi=kappaB}.  Hence \cite{CNS2015-1}*{Theorem 8.1.1} completes the argument.
\end{proof}

\subsection*{Vanishing Analytical Cumulants}

However, something stronger than Corollary \ref{cor:bi-free-is-vanishing-mixed-cumulants} holds.  Indeed, note that the analytic operator-valued bi-free cumulant function has the added benefit that the last entry can be an element of $L_2(A, \tau)$ and thus the $L_2$-image of a product of left and right operators.  As such, it is possible to verify that additional analytic bi-cumulants vanish.  

The desired result is analogous to the scalar-valued result demonstrated in \cite{CS2020}*{Proposition 2.30} and proved in a similar manner.  Thus we begin with a generalization of \cite{CNS2015-1}*{Theorem 9.1.5} where we can expand out a cumulant involving products of operators.  In \cite{CNS2015-1}*{Theorem 9.1.5} only products of left and right operators were considered in the operator-valued setting whereas \cite{CS2020}*{Lemma 2.29} expanded out scalar-valued cumulants involving a product of left and right operator in the last entry.

To begin, fix $m<n\in\bN$, $\chi\in {\{\ell,r\}}^m$, integers
\[
k(0)=0<k(1)<  \ldots<k(m)=n,
\]
and any function $\xh\in {\{\ell,r\}}^n$ such that for all $q \in \{1,\ldots, n\}$ for which there exists a (necessarily unique) $p_q \in \{1,\ldots, m-1\}$ with $k(p_q-1)<q\leq k(p_q)$, we have
\[
\xh(q)=\chi(p_q).
\]
Thus $\widehat{\chi}$ is constant from $k(p-1)+1$ to $k(p)$ whereas $\widehat{\chi}$ does not need to be constant from $k(m-1)+1$ to $k(m)$.

We may embed $\BNC(\chi)$ into $\BNC(\widehat{\chi})$ via $\pi \mapsto \widehat{\pi}$ where the blocks of $\widehat{\pi}$ are formed by taking each block $V$ of $\pi$ and forming a block 
\[
\widehat{V} = \bigcup_{p \in V} \{k(p - 1) + 1, \ldots,k(p)\}
\]
of $\widehat{\pi}$.  It is not difficult to see that $\widehat{\pi} \in \BNC(\widehat{\chi})$ as $\widehat{\chi}$ is constant on $\{k(p - 1) + 1, \ldots,k(p)\}$ for all $p \in V \setminus \{m\}$ and although the block containing $\{k(n - 1) + 1, \ldots,k(n)\}$ has both left and right entries, it occurs at the bottom of the bi-non-crossing diagram and thus poses no problem.  Alternatively, this map can be viewed as an analogue of the map on non-crossing partitions from \cite{NS2006}*{Notation 11.9} after applying $s^{-1}_\chi$.

It is easy to see that $\widehat{1_\chi} = 1_{\widehat{\chi}}$, 
\[
\widehat{0_\chi} = \bigcup_{p=1}^m \{k(p - 1) + 1, \ldots,k(p)\},
\]
and that the map $\pi \mapsto \widehat{\pi}$ is injective and order-preserving.  Furthermore, the image of $\BNC(\chi)$ under this map is
\[
\widehat{\BNC}(\chi) = \left[\widehat{0_\chi}, \widehat{1_\chi}\right] = \left[\widehat{0_\chi}, 1_{\widehat{\chi}}\right] \subseteq \BNC(\widehat{\chi}).
\]
\begin{remark}
	\label{rem:partial-mobius-inversion}
	Recall that since $\mu_{\BNC}$ is the bi-non-crossing M\"{o}bius function, we have for each $\sigma,\pi \in \BNC(\chi)$ with $\sigma \leq \pi$ that
	\[
		\sum_{\substack{ \upsilon \in \BNC(\chi) \\ \sigma \leq \upsilon \leq \pi  }} \mu_{\BNC}(\upsilon, \pi) =  \left\{
			\begin{array}{ll}
				1 & \mbox{if } \sigma = \pi  \\
				0 & \mbox{otherwise }
		\end{array} \right. .
	\]
Since the lattice structure is preserved under the map defined above, we see that $\mu_{\BNC}(\sigma, \pi) = \mu_{\BNC}(\widehat{\sigma}, \widehat{\pi})$.

	It is also easy to see that the partial M\"{o}bius inversion from \cite{NS2006}*{Proposition 10.11} holds in the bi-free setting; that is, if $f, g : \BNC(\chi) \to \bC$ are such that 
	\[
		f(\pi) = \sum_{\substack{\sigma \in \BNC(\chi) \\ \sigma \leq \pi}} g(\sigma)
	\]
	for all $\pi \in \BNC(\chi)$, then for all $\pi, \sigma \in \BNC(\chi)$ with $\sigma \leq \pi$, we have the relation
	\[
		\sum_{\substack{\upsilon \in \BNC(\chi) \\ \sigma \leq \upsilon \leq \pi }} f(\upsilon) \mu_{\BNC}(\upsilon, \pi) = \sum_{\substack{\omega \in \BNC(\chi) \\ \omega \vee \sigma = \pi }} g(\omega)
	\]
	where $\pi \vee \sigma$ denotes the smallest element of $\BNC(\chi)$ greater than $\pi$ and $\sigma$.
\end{remark}

Thus, by following the proofs of either \cite{NS2006}*{Theorem 11.12}, \cite{CNS2015-1}*{Theorem 9.1.5}, or \cite{CS2020}*{Lemma 2.29}, we arrive at the following.

\begin{proposition}\label{prop:analytic-cumulants-with-product-entries}
Let $(A,E,\varepsilon,\tau)$ be  be an analytical $B$-$B$-non-commutative probability space.  Under the above notation, if $\pi \in \BNC(\chi)$ and $Z_k\in A_{\xh(k)}$, then
\[
\widetilde{\kappa}_{\pi}(Z_1\cdots Z_{k(1)},\ldots, Z_{k(m-2)+1}\cdots Z_{k(m-1)}, Z_{k(m-1)+1}\cdots Z_{k(m)} + N_\tau) = \sum_{\substack{\sigma\in \BNC(\xh)\\
\sigma \vee\oxh = \widehat{\pi}}}
        \widetilde{\kappa}_\sigma(Z_1,\ldots,Z_{n-1},Z_n+ N_\tau).
\]
In particular, when $\sigma = 1_\chi$, we have
\[
\widetilde{\kappa}_{\chi}(Z_1\cdots Z_{k(1)},\ldots, Z_{k(m-2)+1}\cdots Z_{k(m-1)}, Z_{k(m-1)+1}\cdots Z_{k(m)} + N_\tau) = \sum_{\substack{\sigma\in \BNC(\xh)\\
\sigma\vee\oxh = 1_{\xh}\\}}
        \widetilde{\kappa}_\sigma(Z_1,\ldots,Z_{n-1},Z_n+ N_\tau),
\]
\end{proposition}
\begin{proof}
First, it is not difficult to verify using the recursive definition of the analytic operator-valued bi-moment function that
\[
\Psi_{\upsilon}(Z_1\cdots Z_{k(1)},\ldots, Z_{k(m-2)+1}\cdots Z_{k(m-1)}, Z_{k(m-1)+1}\cdots Z_{k(m)} + N_\tau)  = \Psi_{\widehat{\upsilon}} (Z_1,\ldots,Z_{n-1},Z_n+ N_\tau)
\]
for all $\upsilon \in \BNC(\chi)$.  Therefore, we have
\begin{align*}
&\widetilde{\kappa}_{\pi}(Z_1\cdots Z_{k(1)},\ldots, Z_{k(m-2)+1}\cdots Z_{k(m-1)}, Z_{k(m-1)+1}\cdots Z_{k(m)} + N_\tau) \\
&= \sum_{\substack{\upsilon \in \BNC(\chi) \\ \upsilon \leq \pi}} \Psi_\upsilon(Z_1\cdots Z_{k(1)},\ldots, Z_{k(m-2)+1}\cdots Z_{k(m-1)}, Z_{k(m-1)+1}\cdots Z_{k(m)} + N_\tau) \mu_{\BNC}(\upsilon, \pi) \\
		& = \sum_{\substack{\upsilon \in \BNC(\chi)\\  \upsilon \leq \pi}}  \Psi_{\widehat{\upsilon}} (Z_1,\ldots,Z_{n-1},Z_n+ N_\tau) \mu_{\BNC}(\widehat{\upsilon}, \widehat{\pi}) \\
		& = \sum_{\substack{\sigma \in \BNC(\widehat{\chi}) \\ \widehat{0_\chi} \leq \sigma \leq  \widehat{\pi}}}  \Psi_{\sigma} (Z_1,\ldots,Z_{n-1},Z_n+ N_\tau) \mu_{\BNC}(\sigma, \widehat{\pi}) \\
		& = \sum_{\substack{\sigma \in \BNC(\widehat{\chi})\\ \sigma \vee \widehat{0_\chi} = \widehat{\pi}}} \widetilde{\kappa}_\sigma(Z_1,\ldots,Z_{n-1},Z_n+ N_\tau)
	\end{align*}
where the last line following from Remark \ref{rem:partial-mobius-inversion}.
\end{proof}

\begin{theorem}\label{thm:vanishing-cumulants-L2-entropy}
Let $(A,E,\varepsilon,\tau)$ be an analytical $B$-$B$-non-commutative probability space containing a family of pairs of $B$-algebras $\{(C_k,D_k)\}_{k\in K}$ that are bi-free with amalgamation over $B$ with respect to $E$.  For each $k \in K$, let $L_2(A_k, \tau)$ be the closed subspace of $L_2(A, \tau)$ generated by
\[
\alg(C_k, D_k) + N_\tau.
\]
Then for all $n\in\bN$, $\chi\in\{\ell,r\}^n$, non-constant maps $\gamma:\{1,\ldots,n\}\to K$, $\xi \in L_2(A_{\gamma(n)}, \tau)$, and $Z_1,\ldots,Z_{n-1}\in A$ such that
\[
Z_k\in \begin{cases}
       C_{\gamma(k)} &\text{if } \chi(k)=\ell \\
        D_{\gamma(k)} &\text{if }\chi(k) = r\\   
     \end{cases},
\]
it follows that
\[
\widetilde{\kappa}_\chi(Z_1,\ldots,Z_{n-1},\xi)=0.
\]
\end{theorem}
\begin{proof}
Fix an $m\in\bN$, $\chi\in\{\ell,r\}^m$, non-constant map $\gamma:\{1,\ldots,m\}\to K$, and $Z_1,\ldots,Z_{m-1}\in A$. For any $n \geq m$, if $Z_m, \ldots, Z_n \in C_{\gamma(m)} \cup D_{\gamma(m)}$ and $\widehat{\chi}$ is defined by
\[
\widehat{\chi}(k) = \begin{cases}
\chi(k) & \text{if } k \leq m \\
\ell & \text{if } k > m \text{ and }Z_k \in C_{\gamma(m)} \\
r & \text{if } k > m \text{ and }Z_k \in D_{\gamma(m)} \\
\end{cases},
\]
then by Proposition \ref{prop:analytic-cumulants-with-product-entries} and Theorem \ref{thmPsi=E} implies that 
\begin{align*}
\widetilde{\kappa}_{\chi}(Z_1,\ldots, Z_{m-1}, Z_{m}\cdots Z_{n} + N_\tau) & = \sum_{\substack{\sigma\in \BNC(\xh)\\  \sigma\vee\oxh = 1_{\xh}\\}}   \widetilde{\kappa}_\sigma(Z_1,\ldots,Z_{n-1},Z_n+ N_\tau)\\
& = \sum_{\substack{\sigma\in \BNC(\xh)\\  \sigma\vee\oxh = 1_{\xh}\\}}   \kappa^B_\sigma(Z_1,\ldots,Z_{n-1},Z_n) + N_{\tau_B}.
\end{align*}
As the conditions $\sigma \in \BNC(\widehat{\chi})$ and $\sigma\vee\oxh = 1_{\xh}$ automatically imply that each block of $\sigma$ containing one of $\{1,2,\ldots, m-1\}$ must also contain an element of $\{m, \ldots, n\}$,  the bi-multiplicative properties of the operator-valued bi-free cumulant function imply that each cumulant  $\kappa^B_\sigma(Z_1,\ldots,Z_{n-1},Z_n)$ appearing in the sum above can be reduced down to an expression involving a mixed $\kappa^B$ term which must be 0 by \cite{CNS2015-1}*{Theorem 8.1.1}.  Hence 
\[
\widetilde{\kappa}_{\chi}(Z_1,\ldots, Z_{m-1}, Z_{m}\cdots Z_{n} + N_\tau) = 0.
\]

Since $\widetilde{E}$ is a continuous function and left multiplication of $A$ on $L_2(A, \tau)$ yields bounded operators,  due to the recursive nature of $\Psi$ we see that $\Psi$ is continuous in the $L_2(A, \tau)$ entry.  Therefore, by M\"{o}bius inversion, $\widetilde{\kappa}$ is continuous in the $L_2(A, \tau)$ entry.  Hence the result follows.
\end{proof}

\begin{corollary}
Let $(A,E,\varepsilon,\tau)$ be an analytical $B$-$B$-non-commutative probability space and let $n\geq 2$, $\chi\in\{\ell,r\}^n$, $\xi\in L_2(A,\tau)$, $b\in B$, and $Z_k\in A_{\chi(k)}$. Suppose that either there exists $p\in\{1,\ldots,n-1\}$ such that  
\[
Z_p= \begin{cases}
      L_b &\text{if }  \chi (p)=\ell\\
       R_b &\text{if } \chi (p) = r
     \end{cases}
\]
or that $\xi\in L_2(B, \tau_B)$.  Then
\[
\widetilde{\kappa}_\chi(Z_1,\ldots,Z_{n-1},\xi)=0.
\]
\end{corollary}
\begin{proof}
If $Z_p = L_b$ or $Z_p = R_b$ for some $p$, then we may proceed as in the proof of Theorem \ref{thm:vanishing-cumulants-L2-entropy} by assuming that $\xi$ is an element of $A$, expanding out the analytic operator-valued bi-free cumulant function with the aid of Proposition \ref{prop:analytic-cumulants-with-product-entries}, and using the fact that non-singleton operator-valued bi-free cumulants involving $L_b$ or $R_b$ terms are zero by \cite{CNS2015-1}*{Proposition 6.4.1} and then taking a limit at the end.

In the case where $\xi \in L_2(B, \tau_B)$, $\xi$ is a limit of terms of the form $L_b + N_{\tau}$.  As
\[
\widetilde{\kappa}_\chi(Z_1,\ldots,Z_{n-1}, L_b +  N_{\tau}) = \kappa_\chi(Z_1,\ldots,Z_{n-1}, L_b ) + N_{\tau_B} = 0 + N_{\tau_B},
\]
by Theorem \ref{thmkappaPsi=kappaB} and \cite{CNS2015-1}*{Proposition 6.4.1} the result follows by taking a limit.
\end{proof}

\section{Bi-Free Conjugate Variables with respect to Completely Positive Maps}
\label{sec:Conj}

In this section we develop the appropriate notion of conjugate variables in order to define bi-free Fisher information and entropy with respect to completely positive maps.  This can be viewed as both an extension of the bi-free conjugate variables developed in \cite{CS2020} and of the free conjugate variables with respect to a completely positive map developed in \cite{S2000}.  We will focus on both the moment and cumulant characterizations of these conjugate variables, whereas \cite{S2000} focused on the moment and derivation characterizations of free conjugate variables.  Although \cite{CS2020} analyzed the moment, cumulant, and bi-free difference quotient characterizations of bi-free conjugate variables, we will forgo trying to generalize the bi-free difference quotient characterization in this setting as it was the cumulant characterization that was found most effective and as the bi-module structures of \cite{S2000} necessary for the derivation characterization using adjoints are less clear in this setting.

We refer the reader to \cite{NSS1999}*{Definition 2.7} as an equivalent description to \cite{S2000} of the free conjugate variables with respect to a completely positive map that we mimic below.

\begin{definition}\label{defn:conjugate-variables}
Let $(A,E,\varepsilon,\tau)$ be an analytical $B$-$B$-non-commutative probability space, let $(C_\ell, C_r)$ be a pair of $B$-algebras in $A$, let $X \in A_\ell$, let $Y \in A_r$, and let $\eta:B\to B$ be a completely positive map.  

An element $\xi \in L_2(A, \tau)$ is said to satisfy the \emph{left bi-free conjugate variable relations for $X$ with respect to $\eta$ and $\tau$ in the presence of $(C_\ell, C_r)$} if for all $n \in \bN \cup \{0\}$, $Z_1,\ldots,Z_n\in \{X\} \cup C_\ell \cup C_r$ we have
\[
\tau(Z_1\cdots Z_n \xi) = \sum_{\substack{1\leq k \leq n\\
Z_k =X}} \tau\left( \left( \prod_{p \in V_k^c\setminus \{k,n+1\}} Z_p\right) L_{\eta\left(  E\left( \prod_{p \in V_k} Z_p\right)     \right)} \right),
\]
where $V_k = \{k < m < n+1 \, \mid \, Z_m \in \{X\} \cup C_\ell\}$ and where all products are taken in numeric order (with the empty product being 1).   If, in addition, 
\[
\xi \in \overline{\alg(X, C_\ell, C_r)}^{\left\|\,  \cdot \, \right\|_\tau},
\]
we call $\xi$ the \emph{left bi-free conjugate variable for $X$ with respect to $\eta$ and $\tau$ in the presence of $(C_\ell, C_r)$} and denote $\xi$ by $J_\ell(X : (C_\ell, C_r), \eta)$.

Similarly, an element $\nu \in L_2(A, \tau)$ is said to satisfy the \emph{right bi-free conjugate variable relations for $Y$ with respect to $\eta$ and $\tau$ in the presence of $(C_\ell, C_r)$} if for all $n \in \bN \cup \{0\}$, $Z_1,\ldots,Z_n\in \{Y\} \cup C_\ell \cup C_r$ we have
\[
\tau(Z_1\cdots Z_n \nu) = \sum_{\substack{1\leq k \leq n\\
Z_k =X}} \tau\left( \left( \prod_{p \in V_k^c\setminus \{k,n+1\} } Z_p\right) R_{\eta\left(  E\left( \prod_{p \in V_k} Z_p\right)      \right)} \right),
\]
where $V_k = \{k < m < n+1 \, \mid \, Z_m \in \{Y\} \cup C_r\}$.   If, in addition, 
\[
\nu \in \overline{\alg(Y, C_\ell, C_r)}^{\left\|\,  \cdot \, \right\|_\tau},
\]
we call $\nu$ the \emph{right bi-free conjugate variable for $X$ with respect to $\eta$ and $\tau$ in the presence of $(C_\ell, C_r)$} and denote $\nu$ by $J_r(Y : (C_\ell, C_r), \eta)$.
\end{definition}
\begin{example}
For an example of Definition \ref{defn:conjugate-variables}, consider $X \in A_\ell$, $Y \in A_r$, $Z_2, Z_3 \in C_\ell$, and $Z_1, Z_4 \in C_r$.  If $\xi = J_\ell(X : (C_\ell, \alg(C_r, Y)), \eta)$, then
\begin{align*}
\tau(XZ_1Z_2YXYZ_3XZ_4\xi) &= \tau\left(Z_1YYZ_4 L_{\eta(E(Z_2XZ_3X))}\right) \\
&\qquad + \tau\left(XZ_1Z_2YYZ_4 L_{\eta(E(Z_3X))}\right) \\
& \qquad+ \tau\left(XZ_1Z_2YXYZ_3Z_4 L_{\eta(E(1))}\right)
\end{align*}
This can be observed diagrammatically by drawing $X, Z_1, Z_2, Y, X, Y, Z_3, X, Z_4$ as one would in a bi-non-crossing diagram (i.e. drawing two vertical lines and placing the variables on these lines starting at the top and going down with left variables on the left line and right variables on the right line), drawing all pictures connecting the centre of the bottom of the diagram to any $X$, taking the product of the elements starting from the top and going down in each of the two isolated components of the diagram, taking the expectation of the bounded region and applying $\eta$ to the result to obtain a $b \in B$,  appending $L_b$ to the end of the product of operators from the unbounded region, and applying $\tau$ to the result.
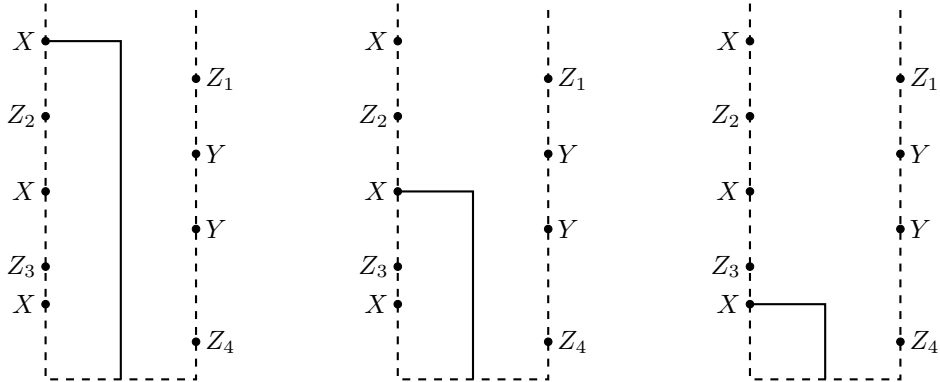
\begin{align*}
		\begin{tikzpicture}[baseline]
			\draw[thick, dashed] (-1,4.5) -- (-1,-.5) -- (1,-.5) -- (1,4.5);
			\draw[thick] (0,-.5) -- (0,4) -- (-1,4);
			\node[right] at (1, 3.5) {$Z_1$};
			\draw[black, fill=black] (1,3.5) circle (0.05);
			\node[right] at (1, 2.5) {$Y$};
			\draw[black, fill=black] (1,2.5) circle (0.05);
			\node[right] at (1, 1.5) {$Y$};
			\draw[black, fill=black] (1,1.5) circle (0.05);	
			\node[right] at (1, 0) {$Z_4$};
			\draw[black, fill=black] (1,0) circle (0.05);
			\node[left] at (-1, 4) {$X$};
			\draw[black, fill=black] (-1,4) circle (0.05);	
			\node[left] at (-1, 3) {$Z_2$};
			\draw[black, fill=black] (-1,3) circle (0.05);
			\node[left] at (-1, 2) {$X$};
			\draw[black, fill=black] (-1,2) circle (0.05);
			\node[left] at (-1, 1) {$Z_3$};
			\draw[black, fill=black] (-1,1) circle (0.05);
			\node[left] at (-1, .5) {$X$};
			\draw[black, fill=black] (-1,.5) circle (0.05);
		\end{tikzpicture}
		\qquad  \qquad
		\begin{tikzpicture}[baseline]
			\draw[thick, dashed] (-1,4.5) -- (-1,-.5) -- (1,-.5) -- (1,4.5);
			\draw[thick] (0,-.5) -- (0,2) -- (-1,2);
			\node[right] at (1, 3.5) {$Z_1$};
			\draw[black, fill=black] (1,3.5) circle (0.05);
			\node[right] at (1, 2.5) {$Y$};
			\draw[black, fill=black] (1,2.5) circle (0.05);
			\node[right] at (1, 1.5) {$Y$};
			\draw[black, fill=black] (1,1.5) circle (0.05);	
			\node[right] at (1, 0) {$Z_4$};
			\draw[black, fill=black] (1,0) circle (0.05);
			\node[left] at (-1, 4) {$X$};
			\draw[black, fill=black] (-1,4) circle (0.05);	
			\node[left] at (-1, 3) {$Z_2$};
			\draw[black, fill=black] (-1,3) circle (0.05);
			\node[left] at (-1, 2) {$X$};
			\draw[black, fill=black] (-1,2) circle (0.05);
			\node[left] at (-1, 1) {$Z_3$};
			\draw[black, fill=black] (-1,1) circle (0.05);
			\node[left] at (-1, .5) {$X$};
			\draw[black, fill=black] (-1,.5) circle (0.05);
		\end{tikzpicture}
			\qquad  \qquad
		\begin{tikzpicture}[baseline]
			\draw[thick, dashed] (-1,4.5) -- (-1,-.5) -- (1,-.5) -- (1,4.5);
			\draw[thick] (0,-.5) -- (0,.5) -- (-1,.5);
			\node[right] at (1, 3.5) {$Z_1$};
			\draw[black, fill=black] (1,3.5) circle (0.05);
			\node[right] at (1, 2.5) {$Y$};
			\draw[black, fill=black] (1,2.5) circle (0.05);
			\node[right] at (1, 1.5) {$Y$};
			\draw[black, fill=black] (1,1.5) circle (0.05);	
			\node[right] at (1, 0) {$Z_4$};
			\draw[black, fill=black] (1,0) circle (0.05);
			\node[left] at (-1, 4) {$X$};
			\draw[black, fill=black] (-1,4) circle (0.05);	
			\node[left] at (-1, 3) {$Z_2$};
			\draw[black, fill=black] (-1,3) circle (0.05);
			\node[left] at (-1, 2) {$X$};
			\draw[black, fill=black] (-1,2) circle (0.05);
			\node[left] at (-1, 1) {$Z_3$};
			\draw[black, fill=black] (-1,1) circle (0.05);
			\node[left] at (-1, .5) {$X$};
			\draw[black, fill=black] (-1,.5) circle (0.05);
		\end{tikzpicture}
\end{align*}
This is analogous to applying the left bi-free difference quotient $\partial_{\ell, X}$ defined in \cite{CS2020} on a suitable algebraic free product to $XZ_1Z_2YXYZ_3XZ_4$ to obtain
\[
Z_1 YYZ_4 \otimes Z_2XZ_3Z + XZ_1Z_2YYZ_4 \otimes Z_3X + XZ_1Z_2YXYZ_3Z_4 \otimes 1,
\]
applying $\mathrm{Id} \otimes (\eta \circ E)$, collapsing the tensor, and applying $\tau$ to the result.

Similarly, if $\nu = J_r(Y : (\alg(C_\ell, X), C_r), \eta)$, then
\begin{align*}
\tau(XZ_1Z_2YXYZ_3XZ_4\nu) &= \tau\left(XZ_1Z_2XZ_3X R_{\eta(E(YZ_4))}\right) + \tau\left(XZ_1Z_2YXZ_3X R_{\eta(E(Z_4))}\right) .
\end{align*}
This can be observed diagrammatically in a similar fashion by drawing all pictures connecting the centre of the bottom of the diagram to any $Y$ on the right.
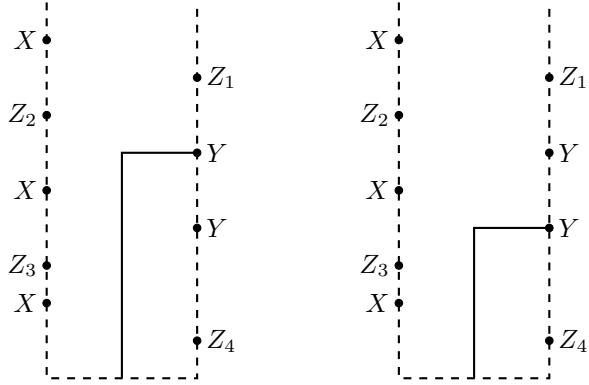
\begin{align*}
		\begin{tikzpicture}[baseline]
			\draw[thick, dashed] (-1,4.5) -- (-1,-.5) -- (1,-.5) -- (1,4.5);
			\draw[thick] (0,-.5) -- (0,2.5) -- (1,2.5);
			\node[right] at (1, 3.5) {$Z_1$};
			\draw[black, fill=black] (1,3.5) circle (0.05);
			\node[right] at (1, 2.5) {$Y$};
			\draw[black, fill=black] (1,2.5) circle (0.05);
			\node[right] at (1, 1.5) {$Y$};
			\draw[black, fill=black] (1,1.5) circle (0.05);	
			\node[right] at (1, 0) {$Z_4$};
			\draw[black, fill=black] (1,0) circle (0.05);
			\node[left] at (-1, 4) {$X$};
			\draw[black, fill=black] (-1,4) circle (0.05);	
			\node[left] at (-1, 3) {$Z_2$};
			\draw[black, fill=black] (-1,3) circle (0.05);
			\node[left] at (-1, 2) {$X$};
			\draw[black, fill=black] (-1,2) circle (0.05);
			\node[left] at (-1, 1) {$Z_3$};
			\draw[black, fill=black] (-1,1) circle (0.05);
			\node[left] at (-1, .5) {$X$};
			\draw[black, fill=black] (-1,.5) circle (0.05);
		\end{tikzpicture}
		\qquad  \qquad
		\begin{tikzpicture}[baseline]
			\draw[thick, dashed] (-1,4.5) -- (-1,-.5) -- (1,-.5) -- (1,4.5);
			\draw[thick] (0,-.5) -- (0,1.5) -- (1,1.5);
			\node[right] at (1, 3.5) {$Z_1$};
			\draw[black, fill=black] (1,3.5) circle (0.05);
			\node[right] at (1, 2.5) {$Y$};
			\draw[black, fill=black] (1,2.5) circle (0.05);
			\node[right] at (1, 1.5) {$Y$};
			\draw[black, fill=black] (1,1.5) circle (0.05);	
			\node[right] at (1, 0) {$Z_4$};
			\draw[black, fill=black] (1,0) circle (0.05);
			\node[left] at (-1, 4) {$X$};
			\draw[black, fill=black] (-1,4) circle (0.05);	
			\node[left] at (-1, 3) {$Z_2$};
			\draw[black, fill=black] (-1,3) circle (0.05);
			\node[left] at (-1, 2) {$X$};
			\draw[black, fill=black] (-1,2) circle (0.05);
			\node[left] at (-1, 1) {$Z_3$};
			\draw[black, fill=black] (-1,1) circle (0.05);
			\node[left] at (-1, .5) {$X$};
			\draw[black, fill=black] (-1,.5) circle (0.05);
		\end{tikzpicture}
\end{align*}
This is analogous to applying the right bi-free difference quotient $\partial_{r, Y}$ defined in \cite{CS2020} on a suitable algebraic free product to $XZ_1Z_2YXYZ_3XZ_4$ to obtain
\[
XZ_1Z_2XZ_3X \otimes Y Z_4 + XZ_1Z_2YXZ_3X \otimes Z_4,
\]
applying $\mathrm{Id} \otimes (\eta \circ E)$, collapsing the tensor, and applying $\tau$ to the result.
\end{example}
\begin{remark}
\label{rem:conjugate-variable-immediate-remarks}
\begin{enumerate}[(i)]
\item As $\tau(aL_b) = \tau(aR_b)$ for all $a \in A$ and $b \in B$, one may use either $L_{\eta \circ E}$ or $R_{\eta \circ E}$ in either part of Definition \ref{defn:conjugate-variables}.  In fact, one may simply use $\eta \circ E$ if one views the resulting element of $B$ as an element of $L_2(B,\tau_B) \subseteq L_2(A,\tau)$, since $\tau(aL_b) = \tau(a (b + N_\tau))$ for all $a \in A$ and $b\in B$ by construction.
\item The element $J_\ell(X : (C_\ell, C_r), \eta)$ is unique in the sense that if $\xi_0 \in  \overline{\alg(X, C_\ell, C_r)}^{\left\|\,  \cdot \, \right\|_\tau}$ satisfies the left bi-free conjugate variable relations for $X$ with respect to $(C_\ell, C_r)$, then $\xi_0 = J_\ell(X : (C_\ell, C_r), \eta)$ as the left bi-free conjugate variable relations causes the inner products in $L_2(A,\tau)$ of both $\xi_0$ and $J_\ell(X : (C_\ell, C_r), \eta)$ against any element of $\alg(X, C_\ell, C_r)$ to be equal.
\item In the case where $B = \bC$, $E$ reduces down to a unital, linear map $\varphi : A\to \bC$ and, as $\tau$ is compatible with $E$, one obtains that $\tau = \varphi$.  As $\varphi$ is linear,  Definition \ref{defn:conjugate-variables} immediately reduces down to the left and right conjugate variables with respect to $\varphi$ in the presence of $(C_\ell, C_r)$ as in \cite{CS2020}, provided $\eta$ is unital.
\item In the setting of Example \ref{exam:factors}, we note that $J_\ell(X : (B_\ell, B_r), \eta)$ exists if and only if the free conjugate variable of $X$ with respect to $(B, \eta)$ from \cite{S2000} exists.  This immediately follows as $B_r$ commutes with $X$ and $B_\ell$ and $\tau$ is tracial, so the expressions for either conjugate variable can be modified into the expressions of the other conjugate variable.
\end{enumerate}
\end{remark}

As with the bi-free conjugate variables in \cite{CS2020}, any moment expression should be equivalent to certain cumulant expressions via M\"{o}bius inversion.  Thus we obtain the following equivalent characterization of conjugate variables.  Note in that which follows, it does not matter whether the last entry in the analytical operator-valued bi-free cumulant function is treated as a left or as a right operator by Definition \ref{defn:anal-bi-multiplicative}.

\begin{theorem}\label{thm:conjugate-variables-via-cumulants}
Let $(A,E,\varepsilon,\tau)$ be an analytical $B$-$B$-non-commutative probability space, let $(C_\ell, C_r)$ be a pair of $B$-algebras in $A$, let $X \in A_\ell$, let $Y \in A_r$, and let $\eta:B\to B$ be a completely positive map.  For $\xi \in L_2(A, \tau)$, the following are equivalent:
\begin{enumerate}[(i)]
\item $\xi$ satisfies the left bi-free conjugate variables relations for $X$ (respectively $\xi$ satisfies the right bi-free conjugate variables relations for $Y$) with respect to $\eta$ and $\tau$ in the presence of $(C_\ell, C_r)$,
\item the following four cumulant conditions hold:
\begin{enumerate}[(a)] 
\item  $\widetilde{\kappa}_{1_{(\ell)}}(\xi)=0 + N_{\tau_B}$,
\item $\widetilde{\kappa}_{1_{(\ell, \ell)}}(XL_b,\xi) =  \eta(b) + N_{\tau_B}$ (respectively $\widetilde{\kappa}_{1_{(r, r)}}(YR_b,\xi) =  \eta(b) + N_{\tau_B}$) for all $b\in B$,
\item $\widetilde{\kappa}_{1_{(\ell, \ell)}}(c_1, \xi ) = \widetilde{\kappa}_{1_{(r, \ell)}}(c_2, \xi ) = 0 + N_{\tau_B}$ for all $c_1 \in C_\ell$ and $c_2 \in C_r$, 
\item for all $n\geq 3$, $\chi\in\{\ell,r\}^n$, and all $Z_1, Z_2, \ldots, Z_{n-1} \in A$ such that
\[
Z_k \in  \begin{cases}
       \{X\} \cup C_\ell &\text{if } \chi(k)=\ell\\
       C_r &\text{if }\chi(k)=r\\
     \end{cases} \qquad \left(   \text{respectively} \quad Z_k \in  \begin{cases}
       C_\ell &\text{if } \chi(k)=\ell\\
       \{Y\} \cup C_r &\text{if }\chi(k)=r\\
     \end{cases} \right)
\]
we have that
\[
\widetilde{\kappa}_\chi(Z_1, \ldots, Z_{n-1},\xi) = 0 + N_{\tau_B}.
\]
\end{enumerate}
\end{enumerate}
\end{theorem}
\begin{proof}
We will prove the result for the left bi-free conjugate variable as the proof for the right bi-free conjugate variable is analogous.  

Suppose that $\xi$ satisfies (ii).  To see that $\xi$ satisfies the left bi-free conjugate variables relations, let $n \in \bN \cup \{0\}$ and let $Z_1,\ldots,Z_n\in \{X\} \cup C_\ell \cup C_r$.  Fix $\chi \in \{\ell, r\}^{n+1}$ such that
\begin{align*}
\chi(k) = \begin{cases}
\ell & \text{if }Z_k \in \{X\} \cup C_\ell \\
r & \text{if } Z_k \in C_r
\end{cases}
\end{align*}
(note the value of $\chi(n+1)$ does not matter in that which follows).  By the relation between the analytic extensions of the bi-moment and bi-cumulant functions, we obtain that
\begin{align*}
\widetilde{E}(Z_1 \cdots Z_n \xi) = \sum_{\pi \in \BNC(\chi)} \widetilde{\kappa}_\pi (Z_1, \ldots, Z_n, \xi).
\end{align*}
Due to the cumulant conditions in (ii), the only way $\widetilde{\kappa}_\pi (Z_1, \ldots, Z_n, \xi)$ is non-zero is if the block of $\pi$ containing $n+1$ contains a single other index $k$ with $Z_k = X$.  Moreover, there is a bijection between such partitions and partitions of the form 
\[
\pi = \{k, n+1\} \cup \pi_1 \cup \pi_2
\]
where $\pi_1$ is a bi-non-crossing partition on $V_k = \{k < m < n+1 \, \mid \, Z_m \in \{X\} \cup C_\ell\}$ with respect to $\chi|_{V_k}$ and where $\pi_2$ is a bi-non-crossing partition on $W_k = V_k \setminus \{k,n+1\}$ with respect to $\chi|_{W_k}$.  Using this decomposition, the properties of bi-analytic extensions of bi-multiplicative functions and the moment-cumulant formulas yield that
\begin{align*}
&\widetilde{E}(Z_1 \cdots Z_n \xi) \\
&= \sum_{\substack{1 \leq k \leq n  \\  Z_k = X}} \sum_{\pi_2 \in \BNC(\chi|_{W_k})}  \sum_{\pi_1 \in \BNC(\chi|_{V_k})}  \widetilde{\kappa}_{\{k, n+1\} \cup \pi_1 \cup \pi_2} (Z_1, \ldots, Z_n, \xi) \\
 &= \sum_{\substack{1 \leq k \leq n  \\  Z_k = X}} \sum_{\pi_2 \in \BNC(\chi|_{W_k})}  \widetilde{\kappa}_{\{k, n+1\} \cup \pi_1} \left( \left. \left(Z_1, \ldots, Z_{k-1}, Z_kL_{  E\left( \prod_{p \in V_k} Z_p\right) } Z_{k+1}, \ldots, Z_n, \xi\right)\right|_{W_k \cup \{k,n+1\}}\right)\\
 &= \sum_{\substack{1 \leq k \leq n  \\  Z_k = X}} \widetilde{\kappa}_{\pi_1} \left( \left. \left(Z_1, \ldots, Z_{\max_{\leq}(W_k)-1}, Z_{\max_{\leq}(W_k)} \eta\left(E\left( \prod_{p \in V_k} Z_p\right)\right)\right)\right|_{W_k}\right)\\
&= \sum_{\substack{1\leq k \leq n\\
Z_k =X}} \tilde{E}\left( \left( \prod_{p \in V_k^c\setminus \{k,n+1\}} Z_p\right)\eta\left(  E\left( \prod_{p \in V_k} Z_p\right)     \right)\right).
\end{align*}
Hence, by applying $\tau_B$ to both sides of this equation, the  left bi-free conjugate variables relations from Definition \ref{defn:conjugate-variables} are obtained via part (i) of Proposition \ref{propE}.

For the converse direction, suppose $\xi$ satisfies the left bi-free conjugate variables relations for $X$.   Thus for all $b \in B$, $\tau(L_b \xi) = 0$ by the conjugate variable relations.  Hence $\widetilde{\kappa}_{1_{(\ell)}}(\xi) = \widetilde{E}(\xi) = 0$ by part (vi) of  Proposition \ref{propE} and therefore (a) holds.

To see that (b) holds, note for all $b_0, b\in B$ that
\begin{align*}
b_0 \widetilde{\kappa}_{1_{(\ell, \ell)}}(X L_b, \xi) = \widetilde{\kappa}_{1_{(\ell, \ell)}}(L_{b_0}X L_b, \xi)  &=  \Psi_{1_{(\ell, \ell)}}(L_{b_0}XL_b, \xi) - \Psi_{0_{(\ell, \ell)}}(L_{b_0}XL_b, \xi)\\
&= \widetilde{E}(L_{b_0}XL_b \xi) - E(L_{b_0}XL_b) \widetilde{E}(\xi) =  \widetilde{E}(L_{b_0}XL_b \xi).
\end{align*}
Therefore, by applying $\tau_B$ to both sides, we obtain that
\[
\tau_B\left(b_0 \widetilde{\kappa}_{1_{(\ell, \ell)}}(X L_b, \xi)\right) = \tau_B\left(\widetilde{E}(L_{b_0}XL_b \xi)\right) = \tau(L_{b_0} XL_b \xi).
\]
By the left bi-free conjugate variable relations we obtain that 
\begin{align*}
\tau_B\left(b_0 \widetilde{\kappa}_{1_{(\ell, \ell)}}(X L_b, \xi)\right)= \tau\left(L_{b_0} L_{\eta(E(L_b))}\right) =\tau(L_{b_0} L_{\eta(b)}) = \tau(L_{b_0\eta(b)}) = \tau_B(b_0\eta(b) ).
\end{align*}
As this holds for all $b_0 \in B$, we obtain that $\widetilde{\kappa}_{1_{(\ell, \ell)}}(X L_b, \xi) = \eta(b) + N_{\tau_B}$ as desired.   

To see that (c) holds, note for all $b \in B$ and $c_1 \in C_\ell$ that
\[
\tau_B\left(b \widetilde{\kappa}_{1_{(\ell, \ell)}}(c_1, \xi)\right) = \tau(L_b c_1 \xi) = 0
\]
by similar computations as above.  Since this holds for all $b \in B$, we see that $\widetilde{\kappa}_{1_{(\ell, \ell)}}(c_1, \xi) = 0 + N_{\tau_B}$.  Similarly, for all $c_2 \in C_r$ we see that
\begin{align*}
\widetilde{\kappa}_{1_{(r, \ell)}}(c_2, \xi) b = \widetilde{\kappa}_{1_{(r, \ell)}}(R_{b}c_2, \xi)  &=  \Psi_{1_{(r, \ell)}}(R_b c_2, \xi) - \Psi_{0_{(r, \ell)}}(R_b c_2, \xi)\\
&= \widetilde{E}(R_b c_2 \xi) - \widetilde{E}(\xi)E(R_b c_2) =  \widetilde{E}(R_b c_2 \xi).
\end{align*}
Therefore, by applying $\tau_B$ to both sides, we obtain that
\[
\tau_B\left(\widetilde{\kappa}_{1_{(r, \ell)}}(c_2, \xi) b \right) = \tau_B\left(\widetilde{E}(R_b c_2 \xi)\right) = \tau(R_b c_2 \xi).
\]
By the left bi-free conjugate variable relations we obtain that 
\begin{align*}
\tau_B\left(\widetilde{\kappa}_{1_{(r, \ell)}}(c_2, \xi) b\right) =0.
\end{align*}
Therefore, as $\tau_B$ is tracial and the above holds for all $b \in B$, we obtain that $\widetilde{\kappa}_{1_{(r, \ell)}}(c_2, \xi) = 0 + N_{\tau_B}$ as desired.

For (d), we proceed by induction on $n$.  To do so, we will prove the base case $n = 3$ and the inductive step simultaneously.  Fix $n\geq 3$ and suppose when $n > 3$ that (d) holds for all $m < n$.  Let $\chi \in \{\ell, r\}^n$ and let $Z_1, Z_2, \ldots, Z_{n-1} \in A$ be as in the assumptions of (d).  We will assume that $\chi(1) = r$ as the case $\chi(1) = \ell$ will be handled similarly.  Thus for all $b \in B$ we know that
\begin{align*}
\widetilde{\kappa}_\chi(Z_1,\ldots, Z_{n-1},\xi)b &= \widetilde{\kappa}_\chi(R_bZ_1, Z_2, \ldots, Z_{n-1},\xi)\\
& = \widetilde{E}(R_bZ_1Z_2 \cdots Z_{n-1} \xi) - \sum_{\substack{\pi\in\BNC(\chi)\\
\pi\neq 1_\chi\\}}\widetilde{\kappa}_\pi(R_bZ_1, Z_2, \ldots, Z_{n-1}, \xi) \\
& = \widetilde{E}(R_bZ_1Z_2 \cdots Z_{n-1} \xi) - \sum_{\substack{\pi\in\BNC(\chi)\\
\pi\neq 1_\chi\\}}\widetilde{\kappa}_\pi(Z_1, Z_2, \ldots, Z_{n-1}, \xi)b.
\end{align*}
Using the fact that (a), (b), (c) hold and that (d) holds for all $m < n$, we obtain using the same arguments used in the other direction of the proof that 
\[
\sum_{\substack{\pi\in\BNC(\chi)\\ \pi\neq 1_\chi\\}}\widetilde{\kappa}_\pi(Z_1, \ldots, Z_{n-1}, \xi)b = \sum_{\substack{1\leq k < n\\ Z_k =X}} \tilde{E}\left(R_b \left( \prod_{p \in V_k^c\setminus \{k,n\}} Z_p\right)\eta\left(  E\left( \prod_{p \in V_k} Z_p\right)     \right)\right),
\]
where $V_k = \{k < m < n \, \mid \, Z_m \in \{X\} \cup C_\ell\}$.  Therefore, by applying $\tau_B$ to both sides of our initial equation, we obtain that
\begin{align*}
&\tau_B\left( \widetilde{\kappa}_\chi(Z_1,\ldots, Z_{n-1},\xi)b \right) \\
& = \tau_B\left(\widetilde{E}(R_bZ_1Z_2 \cdots Z_{n-1} \xi)  - \sum_{\substack{1\leq k < n\\ Z_k =X}} \tilde{E}\left(R_b \left( \prod_{p \in V_k^c\setminus \{k,n\}} Z_p\right)\eta\left(  E\left( \prod_{p \in V_k} Z_p\right)     \right)\right)      \right)\\
& = \tau (R_bZ_1Z_2 \cdots Z_{n-1} \xi)  - \sum_{\substack{1\leq k < n\\ Z_k =X}} \tau\left(R_b \left( \prod_{p \in V_k^c\setminus \{k,n\}} Z_p\right)\eta\left(  E\left( \prod_{p \in V_k} Z_p\right)     \right)\right)  =0,
\end{align*}
by the left bi-free conjugate variable relations.  Therefore, as the above holds for all $b\in B$ and $\tau_B$ is tracial, the result follows.
\end{proof}

The cumulant approach to conjugate variables has merits as it is very simple to check that most cumulants vanish and the values of others.  For instance, an observant reader might have noticed that the operators $X$ and $Y$ in Definition \ref{defn:conjugate-variables} were not required to be self-adjoint.  This is for later use in the paper and can be converted to studying self-adjoint operators as follows.

\begin{lemma}
\label{lem:non-self-adjoint-conjugate-variables-to-self-adjoint}
Let $(A,E,\varepsilon,\tau)$ be an analytical $B$-$B$-non-commutative probability space, let $(C_\ell, C_r)$ be a pair of $B$-algebras in $A$, let $X \in A_\ell$, and let $\eta:B\to B$ be a completely positive map.  The left bi-free conjugate variables
\[
J_\ell(X : (\alg(C_\ell, X^*), C_r), \eta) \qqand J_\ell(X^* : (\alg(C_\ell, X), C_r), \eta)
\]
exist if and only if
\[
J_\ell(\Re(X) : (\alg(C_\ell, \Im(X)), C_r), \eta) \qqand J_\ell(\Im(X) : (\alg(C_\ell, \Re(X)), C_r), \eta)
\]
exist where $\Re(X) = \frac{1}{2}(X+X^*)$ and $\Im(X)= \frac{1}{2i}(X-X^*)$.  Furthermore,
\begin{align*}
J_\ell(\Re(X) : (\alg(C_\ell, \Im(X)), C_r), \eta) &= J_\ell(X : (\alg(C_\ell, X^*), C_r), \eta)  + J_\ell(X^* : (\alg(C_\ell, X), C_r), \eta) \\
J_\ell(\Im(X) : (\alg(C_\ell, \Re(X)), C_r), \eta) &= iJ_\ell(X : (\alg(C_\ell, X^*), C_r), \eta)  - iJ_\ell(X^* : (\alg(C_\ell, X), C_r), \eta).
\end{align*}
A similar result holds for right bi-free conjugate variables.
\end{lemma}
\begin{proof}
Suppose
\[
\xi_1 = J_\ell(X : (\alg(C_\ell, X^*), C_r), \eta) \qqand \xi_2 = J_\ell(X^* : (\alg(C_\ell, X), C_r), \eta)
\]
exist.  Hence $\xi_1$ and $\xi_2$ satisfy the appropriate analytic cumulant equations from Theorem \ref{thm:conjugate-variables-via-cumulants}.  Let
\[
h_1 = \xi_1 + \xi_2 \qqand h_2 = i \xi_1 - i \xi_2.
\]
As
\[
\xi_1 \in \overline{\alg(X, \alg(C_\ell, X^*), C_r)}^{\left\|\, \cdot \, \right\|_\tau} \qqand \xi_2 \in \overline{\alg(X^*, \alg(C_\ell, X), C_r)}^{\left\|\, \cdot \, \right\|_\tau} ,
\]
we easily see that 
\[
h_1 \in \overline{\alg(\Re(X), \alg(C_\ell, \Im(X)), C_r)}^{\left\|\, \cdot \, \right\|_\tau} \qqand \xi_2 \in \overline{\alg(\Im(X), \alg(C_\ell, \Re(X)), C_r)}^{\left\|\, \cdot \, \right\|_\tau}.
\]
Thus, by Theorem \ref{thm:conjugate-variables-via-cumulants}, it suffices to show that $h_1$ and $h_2$ satisfy the appropriate conjugate variable formulae.  Indeed, property (a) of  Theorem \ref{thm:conjugate-variables-via-cumulants} holds as
\[
\widetilde{\kappa}_{1_{(\ell)}}(h_1) = \widetilde{\kappa}_{1_{(\ell)}}(h_2) = 0 + N_{\tau_B}.
\]
Next, notice for all $b \in B$ that
\begin{align*}
&\widetilde{\kappa}_{1_{(\ell, \ell)}}(\Re(X)L_b, h_1) \\
&= \widetilde{\kappa}_{1_{(\ell, \ell)}}\left(\frac{1}{2} X L_b, \xi_1\right) +  \widetilde{\kappa}_{1_{(\ell, \ell)}}\left(\frac{1}{2} X^* L_b, \xi_1\right) +  \widetilde{\kappa}_{1_{(\ell, \ell)}}\left(\frac{1}{2} X L_b, \xi_2\right) +  \widetilde{\kappa}_{1_{(\ell, \ell)}}\left(\frac{1}{2} X^* L_b, \xi_2\right) \\
&= \frac{1}{2} \eta(b) + 0 + 0 + \frac{1}{2}\eta(b) = \eta(b),
\end{align*}
and
\begin{align*}
&\widetilde{\kappa}_{1_{(\ell, \ell)}}(\Im(X)L_b, h_2) \\
&= \widetilde{\kappa}_{1_{(\ell, \ell)}}\left(\frac{1}{2i} X L_b, i\xi_1\right) +  \widetilde{\kappa}_{1_{(\ell, \ell)}}\left(-\frac{1}{2i} X^* L_b, i\xi_1\right) +  \widetilde{\kappa}_{1_{(\ell, \ell)}}\left(\frac{1}{2i} X L_b, -i\xi_2\right) +  \widetilde{\kappa}_{1_{(\ell, \ell)}}\left(-\frac{1}{2i} X^* L_b, -i\xi_2\right) \\
&= \frac{1}{2i} i\eta(b) - 0 + 0 - \frac{1}{2i}(-i)\eta(b) = \eta(b).
\end{align*}
Hence property (b) of Theorem \ref{thm:conjugate-variables-via-cumulants} holds.

To see  properties (c) and (d) of Theorem \ref{thm:conjugate-variables-via-cumulants} hold, note for all $b \in B$ that 
\begin{align*}
&\widetilde{\kappa}_{1_{(\ell, \ell)}}(\Re(X)L_b, h_2) \\
&= \widetilde{\kappa}_{1_{(\ell, \ell)}}\left(\frac{1}{2} X L_b, i\xi_1\right) +  \widetilde{\kappa}_{1_{(\ell, \ell)}}\left(\frac{1}{2} X^* L_b, i\xi_1\right) +  \widetilde{\kappa}_{1_{(\ell, \ell)}}\left(\frac{1}{2} X L_b, -i\xi_2\right) + \widetilde{\kappa}_{1_{(\ell, \ell)}}\left(\frac{1}{2} X^* L_b, -i\xi_2\right) \\
&= \frac{1}{2}i  \eta(b) + 0 + 0 + \frac{1}{2}(-i)\eta(b) = 0,
\end{align*}
and similarly $\widetilde{\kappa}_{1_{(\ell, \ell)}}(\Im(X)L_b, h_1) =0$.  Therefore,  Proposition \ref{prop:analytic-cumulants-with-product-entries} along with the linearity of the cumulants in each entry yield properties (c) and (d).

The converse direction is proved analogously.  
\end{proof}

Of course, many other results follow immediately from the cumulant definition of the conjugate variables.

\begin{lemma}
\label{lem:conj-var-scaling}
Let $(A,E,\varepsilon,\tau)$ be an analytical $B$-$B$-non-commutative probability space, let $(C_\ell, C_r)$ be a pair of $B$-algebras in $A$, let $X \in A_\ell$, and let $\eta:B\to B$ be a completely positive map.  If
\[
\xi = J_\ell(X : (C_\ell, C_r), \eta)   
\]
exists, then for all $\lambda \in \bC \setminus \{0\}$   the conjugate variable $J_\ell(\lambda X : (C_\ell, C_r), \eta) $ exists and is equal to $\frac{1}{\lambda} \xi$.

A similar result holds for right bi-free conjugate variables.
\end{lemma}

\begin{proposition}
Let $(A,E,\varepsilon,\tau)$ be an analytical $B$-$B$-non-commutative probability space, let $(C_\ell, C_r)$ be a pair of $B$-algebras in $A$, let $X \in A_\ell$, and let $\eta_1, \eta_2:B\to B$ be completely positive maps.  If
\[
\xi_1 = J_\ell(X : (C_\ell, C_r), \eta_1) \qqand \xi_2 = J_\ell(X : (C_\ell, C_r), \eta_2)
\]
exist, then $\xi = J_\ell(X : (C_\ell, C_r), \eta_1 + \eta_2)$ exists and $\xi = \xi_1 + \xi_2$.

A similar result holds for right bi-free conjugate variables.
\end{proposition}

\begin{proposition}
Let $(A,E,\varepsilon,\tau)$ be an analytical $B$-$B$-non-commutative probability space, let $(C_\ell, C_r)$ be a pair of $B$-algebras in $A$, let $X \in A_\ell$, and let $\eta :B\to B$ be a completely positive map. For  fixed $b_1, b_2 \in B$, define $\eta_{\ell, b_1, b_2} : B \to B$ by
\[
\eta_{\ell, b_1, b_2}(b) = \eta(b b_2) b_1
\]
for all $b \in B$.  If $\xi = J_\ell(X : (C_\ell, C_r), \eta)$ exists and $\eta_{\ell, b_1, b_2}$ is completely positive, then $J_\ell(X : (C_\ell, C_r), \eta_{\ell, b_1, b_2})$ exists and
\[
J_\ell(X : (C_\ell, C_r), \eta_{\ell, b_1, b_2}) = R_{b_1} L_{b_2} \xi.
\]

Similarly, if $Y \in A_r$ and $\eta_{r, b_1, b_2} : B \to B$ is defined by
\[
\eta_{r, b_1, b_2} = b_2 \eta(b_1 b)
\]
for all $b \in B$ is completely positive, and $J_r(Y : (C_\ell, C_r), \eta)$ exists, then $J_r(Y : (C_\ell, C_r), \eta_{r, b_1, b_2})$ exists and
\[
J_r(Y : (C_\ell, C_r), \eta_{r, b_1, b_2}) = R_{b_1} L_{b_2}J_r(Y : (C_\ell, C_r), \eta).
\]
\end{proposition}
\begin{proof}
By Theorem \ref{thm:conjugate-variables-via-cumulants}, it suffices to show that $R_{b_1} L_{b_2} \xi$ satisfies the appropriate analytical operator-valued bi-free cumulant formula.  Indeed, clearly
\[
\widetilde{\kappa}_{1_{(\ell)}}(R_{b_1} L_{b_2} \xi) = b_2 \widetilde{\kappa}_{1_{(\ell)}}(\xi) b_1 = 0
\]
and for all $b \in B$
\[
\widetilde{\kappa}_{1_{(\ell, \ell)}}(X L_b, R_{b_1} L_{b_2} \xi) = \widetilde{\kappa}_{1_{(\ell, \ell)}}(X L_bL_{b_2},  \xi)b_1 = \widetilde{\kappa}_{1_{(\ell, \ell)}}(X L_{b b_2},  \xi)b_1 = \eta(bb_2) b_1 = \eta_{b_1, b_2}(b).
\]
To show that the other analytical operator-valued bi-free cumulants from Theorem \ref{thm:conjugate-variables-via-cumulants} vanish, one simply needs to use the analytical extension properties of bi-multiplicative functions together with Proposition \ref{prop:analytic-cumulants-with-product-entries}.  The result for right bi-free conjugate variables is analogous.
\end{proof}

Similarly, many results pertaining to conjugate variables from \cites{V1998-2,S2000, CS2020} immediately generalize to the conjugate variables in Definition \ref{defn:conjugate-variables}.  However, one result from \cite{S2000} requires additional set-up.  In the context of Example \ref{exam:factors}, one can always consider a further von Neumann subalgebra $D$ of $B$ and ask how the conjugate variables react.  To analyze the comparable situation in our setting, we need the following example.

\begin{example}\label{exam:reducing-analytic-spaces-to-smaller-algebra}
Let $(A, E, \varepsilon, \tau)$ be an analytical $B$-$B$-non-commutative probability space and let $D$ be a unital $*$-subalgebra of $B$ (with $1_D = 1_B$).  If $F : B \to D$ is a conditional expectation in the sense that $F(d) = d$ for all $d \in D$ and $F(d_1bd_2) = d_1 F(b) d_2$ for all $d_1, d_2 \in D$ and $b \in B$, then $(A, F \circ E, \varepsilon|_{D \otimes D^{\op}})$ is a $D$-$D$-non-commutative probability space by \cite{S2016-2}*{Section 3}.

Note $\tau_D = \tau_B|_D : D \to \bC$ is a tracial state being the restriction of a tracial state.  Moreover, if $\tau_B$ is compatible with $F$ in the sense that $\tau_B(F(b)) = \tau_B(b)$ for all $b \in B$, we easily see that $\tau$ is compatible with $F \circ E$, as for all $a \in A$ we have that
\[
\tau(a) = \tau\left(L_{E(a)}\right) = \tau_B(E(a)) = \tau_B(F(E(a))) = \tau\left(L_{F(E(a))}\right)
\]
and similarly $\tau(a) = \tau\left(R_{F(E(a))}\right)$.  Hence $(A, F \circ E, \varepsilon|_{D \otimes D^{\op}}, \tau)$ is an analytical $D$-$D$-non-commutative probability space.
\end{example} 

\begin{proposition}\label{prop:conj-variables-smaller-amalgamating-algebra}
Let $(A, E, \varepsilon, \tau)$ be an analytical $B$-$B$-non-commutative probability space and  $X \in A_\ell$.  In addition, let $D$ and $F$ be as in Example \ref{exam:reducing-analytic-spaces-to-smaller-algebra} and  let $(C_\ell, C_r)$ be a pair of $B$-algebras (and thus automatically a pair of $D$-algebras) and $\eta : D \to D$ be a completely positive map. Moreover, suppose that $F$ is completely positive (and hence $\eta \circ F : B \to D \subseteq B$ is also completely positive).

Then, the conjugate variable $J_\ell(X : (C_\ell, C_r), \eta \circ F)$ exists in the  analytical $B$-$B$-non-commutative probability space $(A, E, \varepsilon, \tau)$ if and only if the conjugate variable $J_\ell(X : (C_\ell, C_r), \eta)$ exists in the  analytical $D$-$D$-non-commutative probability space $(A, F \circ E, \varepsilon|_{D \otimes D^{\op}}, \tau)$, in which case they are the same element of $L_2(A, \tau)$.

A similar result holds for right bi-free conjugate variables.
\end{proposition}
\begin{proof}
As $(\eta \circ F) \circ E = \eta \circ (F \circ E)$, the bi-free conjugate variable relations from Definition \ref{defn:conjugate-variables} are precisely the same and thus there is nothing to prove.
\end{proof}

With Proposition \ref{prop:conj-variables-smaller-amalgamating-algebra} out of the way, we turn our attention to proving that the expected generalizations of conjugate variable properties from \cites{V1998-2,S2000, CS2020} hold.

\begin{lemma}
\label{lem:con-var-reducing-alg-in-presence}
Let $(A,E,\varepsilon,\tau)$ be an analytical $B$-$B$-non-commutative probability space, let $(C_\ell, C_r)$ be a pair of $B$-algebras in $A$, let $X \in A_\ell$, and let $\eta:B\to B$ be a completely positive map.   Suppose further that $D_\ell \subseteq C_\ell$ and $D_r \subseteq C_r$ are such that $(D_\ell, D_r)$ is a pair of $B$-algebras in $A$.  If
\[
\xi = J_\ell(X : (C_\ell, C_r), \eta)   
\]
exists then
\[
\xi' = J_\ell(X : (D_\ell, D_r), \eta)   
\]
exists.  In particular, if $P$ is the orthogonal projection of $L_2(A,\tau)$ onto 
\[
\overline{\alg(X, D_\ell, D_r)}^{\left\|\,\cdot\,\right\|_\tau},
\]
then $\xi' = P(\xi)$.

A similar result holds for right bi-free conjugate variables.
\end{lemma}
\begin{proof}
Notice if $\xi$ satisfies the left bi-free conjugate variable relations for $X$ with respect to $E$ and $\eta$ in the presence of $(C_\ell, C_r)$, $\xi$ satisfies the left bi-free conjugate variable relations for $X$ with respect to $E$ and $\eta$ in the presence of $(D_\ell, D_r)$.  Therefore, since $\tau(ZP(\xi)) = \tau(Z\xi)$ for all $Z \in \alg(X, D_\ell, D_r)$, it follows that $P(\xi) = J_\ell(X : (D_\ell, D_r), \eta)$ as desired.
\end{proof}

The following generalizes \cite{V1998-2}*{Proposition 3.6}, \cite{S2000}*{Proposition 3.8}, and \cite{CS2020}*{Proposition 4.3}.
\begin{proposition}
\label{prop:con-var-do-not-change-adding-bi-free-parts}
Let $(A,E,\varepsilon,\tau)$ be an analytical $B$-$B$-non-commutative probability space, let $(C_\ell, C_r)$ be a pair of $B$-algebras in $A$, let $X \in A_\ell$, and let $\eta:B\to B$ be a completely positive map.   If $(D_\ell, D_r)$ is another pair of $B$-algebras such that
\[
(\alg(X, C_\ell), C_r ) \qqand (D_\ell, D_r)
\]
are bi-free with amalgamation over $B$ with respect to $E$, then
\[
\xi = J_\ell(X : (C_\ell, C_r), \eta)
\]
exists if and only if
\[
\xi' =  J_\ell(X : (\alg(C_\ell, D_\ell),  \alg(C_r, D_r)), \eta)
\]
exists, in which case they are equal.

A similar result holds for right bi-free conjugate variables.
\end{proposition}
\begin{proof}
Note by Lemma \ref{lem:con-var-reducing-alg-in-presence} that if $\xi'$ exists then $\xi$ exists.

Conversely suppose that $\xi$ exists.  Hence $\xi$ is an $\left\|\,\cdot \, \right\|_\tau$-limit of elements from $\alg(X, C_\ell, C_r)$.  Since the analytical operator-valued bi-free cumulants are $\left\|\,\cdot \, \right\|_\tau$-continuous in the last entry, it follows that any analytical operator-valued bi-free cumulant involving $\xi$ at the end and at least one element of $D_\ell$ or $D_r$ must be zero by Theorem \ref{thm:vanishing-cumulants-L2-entropy} as
\[
(\alg(X, C_\ell), C_r ) \qqand (D_\ell, D_r)
\]
are bi-free with amalgamation over $B$ with respect to $E$.  Therefore, as 
\[
\xi \in \overline{\alg(X, C_\ell, C_r)}^{\left\|\,\cdot\,\right\|_\tau} \subseteq \overline{\alg(X, \alg(C_\ell, D_\ell), \alg(C_r, D_r))}^{\left\|\,\cdot\,\right\|_\tau},
\]
it follows that $\xi'$ exists and $\xi = \xi'$.
\end{proof}

The following generalizes \cite{V1998-2}*{Proposition 3.7}, \cite{S2000}*{Proposition 3.11}, and \cite{CS2020}*{Proposition 4.4}.  In that which follows, we will use $\vZ$ to denote a tuple of operators $(Z_1, \ldots, Z_k)$.  Furthermore, given another tuple $\vZ' = (Z'_1, \ldots, Z'_k)$, we will use  $\vZ + \vZ'$ to denote the tuple $(Z_1 + Z'_1, \ldots, Z_k + Z'_k)$ and we will use $\widehat{\vZ_p}$ to denote the tuple $(Z_1, \ldots, Z_{p-1}, Z_{p+1},\ldots, Z_k)$ obtained by removing $Z_p$ from the list.

\begin{proposition}
\label{prop:conj-var-adding-bi-free-sums}
Let $(A,E,\varepsilon,\tau)$ be an analytical $B$-$B$-non-commutative probability space and let $\eta:B\to B$ be a completely positive map. Suppose $\vX$ and $\vX'$ are $n$-tuples of operators from $A_\ell$, $\vY$ and $\vY'$ are $m$-tuples of operators from $A_r$, and $(C_\ell, C_r)$ and $(D_\ell, D_r)$ are pairs of $B$-algebras such that
\[
(\alg(\vX, C_\ell), \alg(\vY, C_r) ) \qqand (\alg(\vX', D_\ell), \alg(\vY', D_r) )
\]
are bi-free with amalgamation over $B$ with respect to $E$.  If
\[
\xi = J_\ell\left(X_1 : \left(\alg\left( \widehat{\vX_1}, C_\ell\right), \alg(\vY, C_r) \right),  \eta, \tau    \right)
\]
exists then
\[
\xi' = J_\ell\left(X_1+X'_1 : \left(\alg\left(\widehat{(\vX + \vX')_1}, C_\ell, D_\ell \right), \alg(\vY + \vY', C_r, D_r)  \right) \right)
\]
exists.  Moreover, if $P$ is the orthogonal projection of $L_2(\A, \varphi)$ onto 
\[
\overline{\alg(\vX + \vX', \vY + \vY', C_\ell, C_r, D_\ell, D_r ) }^{\left\| \, \cdot \, \right\|_\tau},
\]
then
\[
\xi' = P(\xi).
\]

A similar result holds for right bi-free conjugate variables.
\end{proposition}
\begin{proof}
Suppose $\xi$ exists.  For notation purposes, let $\A = \alg(\vX + \vX', \vY + \vY', C_\ell, C_r, D_\ell, D_r)$.  

Since $\tau(L_bZP(\xi)) = \tau(L_bZ\xi)$ for all $Z \in \A$ and $b \in B$ (as $B_\ell \subseteq C_\ell$), we obtain by Proposition \ref{propE} that $\widetilde{E}(ZP(\xi)) = \widetilde{E}(Z\xi)$ for all $Z \in \A$.  Thus, as $B_\ell, B_r \subseteq \A$, we obtain  for all $\chi \in \{\ell,r\}^p$ with $\chi(p) = \ell$, for all $\pi \in \BNC(\chi)$, and for all $Z_k \in\A$ with 
\[
Z_k \in \begin{cases}
\alg(\vX + \vX', C_\ell,  D_\ell) & \text{if } \chi(k) = \ell \\
\alg(\vY + \vY',   C_r,  D_r) & \text{if } \chi(k) = r
\end{cases}
\]
that $
\Psi_\pi(Z_1, \ldots, Z_{p-1} P(\xi)) = \Psi_\pi(Z_1, \ldots, Z_{p-1}, \xi)$ and thus
\[
\widetilde{\kappa}_\chi(Z_1, \ldots, Z_{p-1}, P(\xi)) = \widetilde{\kappa}_\chi(Z_1, \ldots, Z_{p-1}, \xi).
\]

To show that $P(\xi)$ is the appropriate left bi-free conjugate variable, it suffices to consider expressions of the form $\widetilde{\kappa}_\chi(Z_1, \ldots, Z_{p-1}, P(\xi))$ and show they obtain the correct values as dictated in Theorem \ref{thm:conjugate-variables-via-cumulants}.  By the above, said cumulant is equal to an analytic operator-valued bi-free cumulant involving elements from $C_\ell$, $D_\ell$, $C_r$, $D_r$, $\vX + \vX'$, and $\vY + \vY'$ (in the appropriate positions) and a $\xi$ at the end.  By expanding using linearity, said cumulant can be modified to a sum of cumulants involving only elements from $C_\ell$, $D_\ell$, $C_r$, $D_r$, $\vX$, $\vX'$, $\vY$, and $\vY'$ with a $\xi$ at the end.  By a similar argument to that in Lemma \ref{lem:non-self-adjoint-conjugate-variables-to-self-adjoint}, these cumulants then obtain the necessary values for $P(\xi)$ to be the appropriate left bi-free conjugate variable due to Theorem \ref{thm:conjugate-variables-via-cumulants} applied to $\xi$ and the fact that
\[
(\alg(\vX, C_\ell), \alg(\vY, C_r) ) \qqand (\alg(\vX', D_\ell), \alg(\vY', D_r) )
\]
are bi-free with amalgamation over $B$ with respect to $E$, so mixed cumulants vanish by Theorem \ref{thm:vanishing-cumulants-L2-entropy}.
\end{proof}

\section{Bi-Semicircular Operators with Completely Positive Covariance}\label{sec:semis}

One essential example of conjugate variables in \cites{V1998-2, S1998, CS2020} comes from central limit distributions.  Thus this section is devoted to defining the operator-valued bi-semicircular operators with covariance coming from a completely positive map, showing that one may add in certain bi-semicircular operators into analytical $B$-$B$-non-commutative probability spaces, and showing the bi-free conjugate variables behave in the appropriate manner.

To begin, let $B$ be a unital $^*$-algebra and let $K$ be a finite index set.  For each $k \in K$ let $Z_k$ be a symbol.  Recall the \emph{full Fock space} $\F(B, K)$ is the algebraic free product of $B$ and $\{Z_k\}_{k \in K}$; that is
\[
\F(B, K) = B \oplus H_1 \oplus H_2 \oplus \cdots
\]
where
\[
H_m = \{ b_0 Z_{k_1} b_1  \cdots Z_{k_m} b_m \, \mid \, b_0, b_1, \ldots, b_m \in B, k_1, \ldots, k_m \in K\}.
\]
Note $\F(B, K)$ is a $B$-$B$-bimodule with the obvious left and right actions of $B$ on $B$ and $H_m$.  Moreover, as $\F(B, K)$ is a direct sum of $B$ and another $B$-$B$-bimodule, $\F(B, K)$ is a $B$-$B$-bimodule with the specified vector state $p : \F(B, K) \to B$ (as in the sense of \cite{CNS2015-1}*{Definition 3.1.1}) defined by taking the $B$-term in the above direct product.  Therefore, the set $A$ of linear maps on $\F(B, K)$ is a $B$-$B$-non-commutative probability space with respect to the expectation $E : A \to B$ defined by $E(T) = p(T 1_B)$ (see \cite{CNS2015-1}*{Remark 3.2.2}).

Let $\{\eta_{i,j}\}_{i,j \in K}$ be linear maps on $B$.  For each $k \in K$, the \emph{left creation and annihilation operators} $l_k$ and $l^*_k$ are the linear maps defined such that
\begin{align*}
l_k b &= 1_B Z_k b \\
l_k( b_0 Z_{k_1} b_1  \cdots Z_{k_m} b_m ) &= 1_B Z_k b_0 Z_{k_1} b_1  \cdots Z_{k_m} b_m  \\
l^*_kb &= 0 \\
l^*_k(b_0 Z_{k_1} b_1  \cdots Z_{k_m} b_m ) &= \eta_{k,k_1}(b_0) b_1  \cdots Z_{k_m} b_m 
\end{align*}
and the \emph{right creation and annihilation operators} $r_k$ and $r_k^*$ are the linear maps defined such that
\begin{align*}
r_k b &= b Z_k 1_B \\
r_k( b_0 Z_{k_1} b_1  \cdots Z_{k_m} b_m ) &= b_0 Z_{k_1} b_1  \cdots Z_{k_m} b_m Z_k 1_B  \\
r^*_kb &= 0 \\
r^*_k(b_0 Z_{k_1} b_1  \cdots Z_{k_m} b_m ) &= b_0 Z_{k_1} b_1  \cdots b_{m-1} \eta_{k_m, k}(b_m) 
\end{align*}
It is elementary to see that $l_k, l_k^* \in A_\ell$ and $r_k, r_k^* \in A_r$.   With these  operators in hand, we make the following definition.

\begin{definition}\label{defn:bi-semi-circulars}
Using the above notation, write $K$ as the disjoint union of two sets $I$ and $J$.  For each $i \in I$ and $j \in J$, let
\[
S_i = l_i + l^*_i \qqand D_j = r_j + r^*_j.
\]
The pair $(\{S_i\}_{i \in I}, \{D_j\}_{j \in J})$ are called the \emph{operator-valued bi-semicircular operators with covariance $\{\eta_{i,j}\}_{i,j \in K}$}.

In the case that $\eta_{k_1,k_2} = 0$ for all $k_1, k_2 \in K$ with $k_1 \neq k_2$, we say that $(\{S_i\}_{i \in I}, \{D_j\}_{j \in J})$ is a collection of \emph{$(\{\eta_{i,i}\}_{i \in I}, \{\eta_{j,j}\}_{j \in J})$ bi-semicircular operators}.
\end{definition}

\begin{remark}
It is natural to ask what are the necessary conditions for operator-valued bi-semicircular operators with covariance $\{\eta_{i,j}\}_{i,j \in K}$ to sit inside an analytical $B$-$B$-non-commutative probability space.  One may hope that a condition similar to \cite{S1998}*{Theorem 4.3.1} would work; that is, the answer is yes if $\tau_B$ is tracial and $\eta : M_{|K|}(B) \to M_{|K|}(B)$ defined by
\[
\eta([b_{i,j}]_{i,j \in K}) = [\eta_{i,j}(b_{i,j})]_{i,j \in K}
\]
is completely positive.    However, if $i \in I$, $j \in J$, and $b_1, b_2 \in B$, it is not difficult to verify that
\[
\tau_B\left(E((S_i D_j L_{b_1} R_{b_2})^*(S_i D_j L_{b_1} R_{b_2})\right) = \tau_B\left(b_2^* \eta_{i,i}(1_B) b_1b_2 \eta_{j,j}(1_B) b_1^* + b_2^* \eta_{i,j}( \eta_{i,j}(b_1b_2))b_1^*\right)
\]
and it is not clear if this is positive (even if the outer $\eta_{i,j}$ was a $\eta_{j,i}$).  

Only certain operator-valued bi-semicircular operators are required in this paper.  Indeed, we will need only the case where $(\{S_i\}_{i \in I}, \{D_j\}_{j \in J})$ are $(\{\eta_{i,i}\}_{i \in I}, \{\eta_{j,j}\}_{j \in J})$ bi-semicircular operators, as in this setting the pairs of algebras
\[
\{(\alg(B_\ell, S_i), B_r)\}_{i \in I} \cup \{(B_\ell, \alg(B_r, D_j))\}_{j \in J}
\]
are bi-free with amalgamation over $B$ with respect to $E$, as the following result shows.
\end{remark}

\begin{theorem}\label{thm:cumulants-of-semis}
Using the above notation, if $B$ is a $*$-algebra, $\tau_B$ is a tracial state on $B$, $\{\eta_{i}\}_{i \in I} \cup \{\eta_{j}\}_{j \in J}$ are completely positive maps from $B$ to $B$, $(\{S_i\}_{i \in I}, \{D_j\}_{j \in J})$ are $(\{\eta_{i}\}_{i \in I}, \{\eta_{j}\}_{j \in J})$ bi-semicircular operators, $A_\ell = \alg(B_\ell, \{S_i\}_{i \in I})$, $A_r = \alg(B_r, \{D_j\}_{j \in J})$, $A$ is generated as a $*$-algebra by $A_\ell$ and $A_r$ and $\tau : A \to \bC$ is defined by $\tau = \tau_B \circ E$, then $(A, E, \varepsilon, \tau)$ is an analytical $B$-$B$-non-commutative probability space.  Moreover, all operator-valued bi-free cumulants involving $\{S_i\}_{i \in I}$ and $\{D_j\}_{j \in J}$ of order not two are zero and for all $i,i_1, i_2 \in I$ and $j,j_1,j_2 \in J$, 
\begin{align*}
\kappa_{1_{(\ell, \ell)}}(S_{i_1} L_b, S_{i_2}) &= \delta_{i_1, i_2}\eta_{i_1}(b) & \kappa_{1_{(r, r)}}(D_{j_1} R_b, D_{j_2}) &= \delta_{j_1, j_2}\eta_{ j_1}(b)\\
\kappa_{1_{(\ell, r)}}(S_{i} L_b, D_{j}) &= 0 & \kappa_{1_{(r, \ell)}}(D_{j} R_b, S_{i}) &= 0.
\end{align*}
\end{theorem}
\begin{proof}
As shown above, $(A, E, \varepsilon)$ is a $B$-$B$-non-commutative probability space.  Next, note that $E$ is completely positive when restricted $A_\ell$ and $A_r$ by \cite{S1998}*{Remark 4.3.2} as the expectations reduce to the free case.  In fact, this same idea can be used to show that $\tau$ is positive.  Indeed, first note that $S_i$ and $D_j$ commute.  Thus every element of $A$ can be written as sum of elements of the form
\[
Z = L_{b_0} S_{i_1} L_{b_1} \cdots S_{i_n} L_{b_n} R_{b_{n+1}} D_{i_{n+1}} R_{b_{n+2}} \cdots D_{j_{n+m}} R_{b_{n+m}}
\]
where $b_0, \ldots b_{n+m} \in B$.  Moreover,  we can write
\[
\F(B, K) \cong \F(B, I) \otimes_B \F(B, J)
\]
in such a way that $Z$ acts via
\[
L_{b_0} S_{i_1} L_{b_1} \cdots S_{i_n} L_{b_n} \otimes L_{b_{n+m}} S_{j_{n+m}} L_{b_{n+m-1}} \cdots S_{j_{n+1}} L_{b_{n+1}},
\]
so that if $E_I : \mathcal{L}(\F(B, I)) \to B$ and  $E_J : \mathcal{L}(\F(B, J)) \to B$ are the corresponding expectations, then
\[
E(Z) = E_I(L_{b_0} S_{i_1} L_{b_1} \cdots S_{i_n} L_{b_n}) E_J(L_{b_{n+m}} S_{j_{n+m}} L_{b_{n+m-1}} \cdots S_{j_{n+1}} L_{b_{n+1}}).
\]
Therefore, if we have $Z = \sum^d_{k=1} X_k Y_k$ where $X_k$ is a product of $L_b$'s and $S_i$'s and $Y_k$ is a product of $R_b$'s and $D_j$'s, then
\[
\tau(Z^*Z) = \sum^d_{k_1, k_2=1} \tau_B\left(E_I\left(X_{k_1}^* X_{k_2}\right) E_J\left(\widetilde{Y}_{k_2} \widetilde{Y}^*_{k_1}\right)\right),
\]
where $\widetilde{Y}$ represents the monomial obtained by reversing the order and changing $R$'s to $L$'s and $D$'s to $S$'s.  However, as $E_I$ and $E_J$ are completely positive on the algebras generated by $L$'s and $S$'s, we can find $b_{X, k_1, k_3},b_{Y, k_1, k_4} \in B$ such that
\[
E_I\left(X_{k_1}^* X_{k_2}\right) = \sum^{d_2}_{k_3=1} b^*_{X, k_1, k_3}b_{X, k_2, k_3} \qand E_J\left(\widetilde{Y}_{k_2} \widetilde{Y}^*_{k_1}\right) = \sum^{d_3}_{k_4=1} b_{Y, k_2, k_4}b^*_{Y, k_1, k_4},
\]
thus
\begin{align*}
\tau(Z^*Z) & = \sum^d_{k_1, k_2=1}\sum^{d_2}_{k_3=1}\sum^{d_3}_{k_4=1} \tau_B(b^*_{X, k_1, k_3}b_{X, k_2, k_3}b_{Y, k_2, k_4}b^*_{Y, k_1, k_4} ) \\
&= \sum^{d_2}_{k_3=1}\sum^{d_3}_{k_4=1} \tau_B\left(\left(\sum^d_{k_1=1} b_{X, k_1, k_3} b_{Y, k_1, k_4}\right)^*\left(\sum^d_{k_2=1} b_{X, k_2, k_3} b_{Y, k_2, k_4}\right)   \right) \geq 0.
\end{align*}
Hence $\tau$ is positive.

Next, one can verify in $L_2(A,\tau)$ that $S_i$ and $D_j$ are the sum of an isometry and its adjoint (see \cite{S1998}*{Proposition 4.6.9}) and thus define bounded linear operators.  Hence $(A, E, \varepsilon, \tau)$ is an analytical $B$-$B$-non-commutative probability space.

To see the cumulant condition, one can proceed in two ways.  One can immediately realize that
\[
(\alg(B_\ell, \{S_i\}_{i \in I}), B_r) \qqand (B_\ell, \alg(B_r, \{D_j\}_{j \in J})
\]
are bi-free over $B$ due to the above tensor-product relation.  This implies mixed cumulants are zero.  The other cumulants then follow from the free case in \cite{S1998}.  Alternatively, one can analyze the actions of $L_b$, $R_b$, $S_i$, and $D_j$ as one would on the operator-valued reduced free product space in an identical way to the LR-diagrams of \cites{CNS2015-2,CNS2015-1} to obtain a diagrammatic description of the elements of $\F(B, K)$ produced, note that the ones that contribute to a $B$-element are exactly the bi-non-crossing diagrams that correspond to pair bi-non-crossing partitions, and use induction to deduce the values of the operator-valued cumulants.
\end{proof}

We immediately obtain the following using Theorem \ref{thm:cumulants-of-semis} and Theorem \ref{thm:conjugate-variables-via-cumulants}.

\begin{lemma}
\label{lem:conj-variables-for-semis}
Let $(A,E,\varepsilon,\tau)$ be an analytical $B$-$B$-non-commutative probability space, let $\{\eta_{\ell, i}\}^n_{i=1}$ and $\{\eta_{r, j}\}^m_{j=1}$ be completely positive maps from $B$ to $B$, and let $(\{S_i\}^n_{i=1}, \{D_j\}^m_{j=1})$ be $(\{\eta_{\ell, i}\}^n_{i=1}, \{\eta_{r, j}\}^m_{j=1})$ bi-semicircular operators in $A$.  Then
\[
J_\ell\left(S_1 : \left(\alg\left(B_\ell, \{S_i\}^n_{j=2}\right), \alg\left(B_r, \{D_j\}^m_{j=1}\right)\right), \eta_{\ell, 1}\right) = S_1.
\]

A similar result holds for the other left and the right conjugate variables.
\end{lemma}

In order to obtain more examples of bi-free conjugate variables, it would be typical to perturb by bi-semicircular operators and use Proposition \ref{prop:conj-var-adding-bi-free-sums}.  To do this we must have the collection of bi-semicircular operators in the same analytical $B$-$B$-non-commutative probability space.  Thus, it is natural to ask whether given two analytical $B$-$B$-non-commutative probability spaces there is a bi-free product which causes the pairs of left and right algebras to be bi-freely independent over $B$ and preserve the analytical properties.

Unfortunately we do not have an answer to this question.  The proof of positivity in the operator-valued free case requires the characterization of the vanishing of alternating centred moments  in \cite{S1998}*{Proposition 3.3.3} to ensure positivity in the end.  One may attempt to use the bi-free analogue of `alternating centred moments vanish' from \cite{C2019}, however the bi-free formulae generalization of \cite{S1998}*{Proposition 3.3.3} is far more complicated.  In particular, the proof from \cite{S1998} will not immediately generalize, as Example \ref{canonicalexample} shows $E$ will not be positive and the traciality of $\tau_B$ will need to come into play.  

Luckily, if we deal only with bi-semicircular operators, which is all that is required for this paper, there is no issue.  In fact, in the case one is working with von Neumann factors as in Example \ref{exam:factors}, the following is trivial as it we can add the corresponding collection of bi-semicircular operators using factors by \cite{S1998}.

\begin{theorem}\label{thm:can-add-semis}
Let $(A,E,\varepsilon,\tau)$ be an analytical $B$-$B$-non-commutative probability space with $A_\ell$ and $A_r$ generated by isometries, let $\{\eta_{\ell, i}\}^n_{i=1} \cup \{\eta_{r, j}\}^m_{j=1}$ be completely positive maps from $B$ to $B$, and let $(\{S_i\}^n_{i=1}, \{D_j\}^m_{j=1})$ be $(\{\eta_{\ell, i}\}^n_{i=1}, \{\eta_{r, j}\}^m_{j=1})$ bi-semicircular operators.  Then, there exists an analytical $B$-$B$-non-commutative probability space $(A', E', \varepsilon', \tau')$ with $A \subseteq A'$, $E'|_A = E$, $\tau'|_A = \tau$, $A_\ell \subseteq A'_\ell$, $A_r \subseteq A'_r$, $\{S_i\}^n_{i=1} \subseteq A'_\ell$, $\{D_j\}^m_{j=1} \subseteq A'_r$ and such that the pairs of algebras
\[
(A_\ell, A_r) \qqand \left(\alg(B_\ell, \{S_i\}^n_{i=1}),\alg(B_r, \{D_j\}^m_{j=1})\right)
\]
are bi-free with amalgamation over $B$ with respect to $E'$.
\end{theorem}

As for any $(\{\eta_{i,i}\}_{i \in I}, \{\eta_{j,j}\}_{j \in J})$ bi-semicircular operators  $(\{S_i\}_{i \in I}, \{D_j\}_{j \in J})$ we know that
\[
\{(\alg(B_\ell, S_i), B_r)\}_{i \in I} \cup \{(B_\ell, \alg(B_r, D_j))\}_{j \in J}
\]
are bi-free over $B$, to prove Theorem \ref{thm:can-add-semis} it suffices to use the following lemma and an analogous result on the right iteratively, or simply adapt the proof to multiple operators simultaneously.

\begin{lemma}
Let $(A,E,\varepsilon,\tau)$ be an analytical $B$-$B$-non-commutative probability space with $A_\ell$ and $A_r$ generated by isometries, let $\eta : B \to B$ be a completely positive map and let $S$ be an $\eta$-semicircular operator.  Then, there exists an analytical $B$-$B$-non-commutative probability space $(A', E', \varepsilon', \tau')$ such that $A' = \alg(A_\ell, A_r, S)$, $E'|_A = E$, $\tau'|_A = \tau$, $A'_\ell = \alg(A_\ell, S)$, $A'_r = A_r$  and
\[
(A_\ell, A_r) \qqand (\alg(B_\ell, S), B_r)
\]
are bi-free with amalgamation over $B$ with respect to $E'$.
\end{lemma}
\begin{proof}
By taking the operator-valued bi-free product of $B$-$B$-non-commutative probability spaces, we obtain a $B$-$B$-non-commutative probability space $(A', E', \varepsilon')$ such that $A' = \alg(A_\ell, A_r, S)$, $A'_\ell = \alg(A_\ell, S)$, $A'_r = A_r$, $E'|_A = E$ and
\[
(A_\ell, A_r) \qqand (\alg(B_\ell, S), B_r)
\]
are bi-free with amalgamation over $B$ with respect to $E'$.  Note $E'$ restricted to $A_r$ is trivially completely positive and $E'$ restricted to $A'_\ell$ is completely positive by the free result from \cite{S1998}.  Thus, to verify Definition \ref{enhancedbbdef} it suffices to verify that if $\tau' = \tau_B \circ E$, then $\tau'$ is positive and elements of $A'_\ell$ and $A'_r$ define bounded operators on $L_2(A', \tau')$.

By analyzing the reduced free product construction, we can realize $A_\ell$, $A_r$, and $S$ as operators acting on 
\[
\F = B \oplus \H_1 \oplus \H_2 \oplus \cdots,
\]
where
\[
\H_m = \{ a_0 Z a_1  \cdots Z a_m \, \mid \, a_0, a_1, \ldots, a_{m-1} \in A_\ell, a_m \in A\}
\]
and if $p : \F \to B$ is defined by taking the $B$-term in $\F$, then 
\[
E'(T) = p(T 1_B).
\]
Define a function $\langle \, \cdot , \cdot \, \rangle : \F \times \F \to A$ by setting $B$, $\H_1$, $\H_2$, $\ldots$. to be pairwise orthogonal, $\langle b_1, b_2 \rangle = L_{b_2^*b_1}$, and
\[
\langle a'_0 Z a'_1  \cdots Z a'_m, a_0 Z a_1  \cdots Z a_m \rangle = a_m^*L_\eta(a_{m-1}^* \cdots   L_\eta(a_1^* L_\eta(a_0^* a'_0)a'_1) \cdots a'_{m-1})a'_m,
\]
where $L_\eta(T) = L_{\eta(E(T))}$.  As $\eta$ is completely positive and $E$ is completely positive when restricted to $A_\ell$, we obtain that $\langle \, \cdot , \cdot \, \rangle$ is an $A$-valued inner product by the same arguments as \cite{S1998}*{Proposition 4.6.6}.  To elaborate slightly, given $\sum^n_{k=1} a_{k,0} Z a_{k,1}  \cdots Z a_{k,m}$, the matrix $[\eta(a^*_{i,0} a_{j,0})]$ is positive and thus can be written as $[\sum^n_{k=1} b^*_{k,i} b_{k,j}]$ for some $b_{i,j} \in B$.  One then substitutes $L_\eta(a^*_{i,0} a_{j,0}) = \sum^n_{k=1} L_{b_{k,i}}^* L_{b_{k,j}}$ and continues until one ends with a sum of products of elements of $A$ with their adjoints.

As $\tau : A \to \bC$ is positive and as for all $T \in A'$,
\[
\tau'(T^*T) = \tau(\langle T1_B, T1_B\rangle),
\]
we obtain that $\tau'$ is positive as desired.  To see that elements of $A'_\ell$ and $A'_r$ define bounded operators on $L_2(A', \tau')$ note if $T \in A_r$, then, using the above description,
\[
T(a_0 Z a_1  \cdots Z a_m) = a_0 Z a_1  \cdots Z Ta_m.
\]
As any of the terms
\[
L_\eta(a_{m-1}^* \cdots   L_\eta(a_1^* L_\eta(a_0^* a'_0)a'_1) \cdots a'_{m-1})
\]
in the above $A$-valued inner product will be able to be written as sums involving terms of the form $L_{b_1}^*L_{b_2}$ and will then produce terms of the form
\[
a_m^* T^* L_{b_1}^* L_{b_2} T a'_m = a_m^*  L_{b_1}^*T^*T L_{b_2}  a'_m
\]
in the $A$-valued inner product when $T$ acts on the left, the fact that $A_r$ is generated by isometries yields that $A_r$ acts as bounded operators on $L_2(A, \tau)$.  The fact that $A_\ell$ is generated by isometries immediately yields that $A_\ell$ acts as bounded operators on $L_2(A, \tau)$, and it is not difficult to see that $S$ acts as the sum of an isometry and its adjoint on $L_2(A,\tau)$ and thus is bounded.
\end{proof}

With Theorem \ref{thm:can-add-semis} establishing we can always assume our $B$-$B$-non-commutative probability spaces have $(\{\eta_{\ell, i}\}^n_{i=1}, \{\eta_{r, j}\}^m_{j=1})$ bi-semicircular operators, we can proceed with the following.

\begin{theorem}
\label{thm:conj-perturb-by-semis}
Let $(A,E,\varepsilon,\tau)$ be an analytical $B$-$B$-non-commutative probability space, let $\{\eta_{\ell, i}\}^n_{i=1}$ and $\{\eta_{r, j}\}^m_{j=1}$ be completely positive maps from $B$ to $B$, let $\vX \in A_\ell^n$ and $\vY \in A_r^m$ be tuples of self-adjoint operators, let $(\{S_i\}^n_{i=1}, \{D_j\}^m_{j=1})$ be $(\{\eta_{\ell, i}\}^n_{i=1}, \{\eta_{r, j}\}^m_{j=1})$ bi-semicircular operators in $A$ and let $(C_\ell, C_r)$ be pairs of $B$-algebras of $A$ such that
\[
\{(\alg(C_\ell, \vX), \alg(C_r, \vY)\rangle)\} \cup \{(\alg(B_\ell, S_i), B_r)\}^n_{i=1} \cup \{(B_\ell, \alg(B_r, D_j))\}^m_{j=1}
\]
are bi-free.  If $P$ is the orthogonal projection of $L_2(A, \varphi)$ onto 
\[
\overline{\alg\left(C_\ell, C_r, \vX + \sqrt{\epsilon} \vS, \vY + \sqrt{\epsilon} \vD\right)}^{\left\| \, \cdot \, \right\|_\tau},
\]
then
\[
\xi = J_\ell\left(X_1 + \sqrt{\epsilon} S_1 : \left(\alg\left( C_\ell, \widehat{(\vX + \sqrt{\epsilon} \vS)}_1\right), \alg\left(C_r,  \vY + \sqrt{\epsilon} \vD \right) \right), \eta_{\ell, 1}\right) = \frac{1}{\sqrt{\epsilon}} P(S_1).
\]
Thus
\[
 \left\|\xi\right\|_\tau \leq \frac{1}{\sqrt{\epsilon}} \sqrt{\tau_B(\eta(1_B))}.
 \]
A similar computation holds for the other entries of the tuples and the right conjugate variables.
\end{theorem}
\begin{proof}
By Lemma \ref{lem:conj-variables-for-semis} and Lemma \ref{lem:conj-var-scaling}, we have that
\begin{align*} 
J_\ell\left(\sqrt{\epsilon} S_1 : \left(\alg\left(B_\ell,  \sqrt{\epsilon} \hat{\vS}_1 \right), \alg\left( B_r,  \sqrt{\epsilon}\vD   \right)\right) , \eta_{\ell, 1}\right) = \frac{1}{\sqrt{\epsilon}} S_1.
\end{align*}
The conjugate variable result then follows from Propositions \ref{prop:con-var-do-not-change-adding-bi-free-parts} and \ref{prop:conj-var-adding-bi-free-sums}, whereas the $\tau$-norm computation is trivial.
\end{proof}

\section{Bi-Free Fisher Information with Respect to a Completely Positive Map}
\label{sec:Fisher}

With the above technology, the bi-free Fisher information with respect to completely positive maps can be constructed and has similar properties to the bi-free Fisher information from \cite{CS2020} and the free Fisher information with respect to a completely positive map from \cite{S1998}.  We highlight the main results and properties in this section.

\begin{definition}
\label{defn:bi-fisher}
Let $(A,E,\varepsilon,\tau)$ be an analytical $B$-$B$-non-commutative probability space, let $\{\eta_{\ell, i}\}^n_{i=1}$ and $\{\eta_{r, j}\}^m_{j=1}$ be completely positive maps from $B$ to $B$, let $\vX \in A_\ell^n$ and $\vY \in A_r^m$ and let $(C_\ell, C_r)$ be a pair of $B$-algebras of $A$.  The \emph{relative bi-free Fisher information of $(\vX, \vY)$ with respect to $(\{\eta_{\ell, i}\}^n_{i=1}, \{\eta_{r, j}\}^m_{j=1})$ in the presence of $(C_\ell, C_r)$} is
\[
\Phi^*(\vX \sqcup \vY : (C_\ell, C_r), (\{\eta_{\ell, i}\}^n_{i=1}, \{\eta_{r, j}\}^m_{j=1})) = \sum^n_{i=1} \left\|\xi_i\right\|_\tau^2 + \sum^m_{j=1} \left\|\nu_j\right\|_\tau^2
\]
where for $i \in \{1, \ldots, n\}$ and $j \in \{1,\ldots, m\}$
\[
\xi_i = J_\ell\left(X_i : \left(\alg\left(C_\ell, \widehat{\vX}_i\right), \alg\left(C_r,  \vY\right)\right), \eta_{\ell, i}\right)
\qand
\nu_j = J_r\left(Y_{j} : \left(\alg\left(C_\ell, \vX\right), \alg\left(C_r,  \widehat{\vY}_j\right)\right), \eta_{r,j}\right)
\]  
provided these variables exist, and otherwise is defined as $\infty$.

In the case that $\eta_{\ell, i} = \eta_{r, j} = \eta$ for all $i$ and $j$, we use $\Phi^*(\vX \sqcup \vY : (C_\ell, C_r), \eta)$ to denote the above bi-free Fisher information.  In the case that $C_\ell = B_\ell$ and $C_r = B_r$, we use $\Phi^*(\vX \sqcup \vY : (\{\eta_{\ell, i}\}^n_{i=1}, \{\eta_{r, j}\}^m_{j=1}))$.  In the case both occur, we use $\Phi^*(\vX \sqcup \vY :  \eta)$.
\end{definition}

Note the bi-free Fisher with respect to completely positive maps exists in many settings due to Theorem \ref{thm:can-add-semis} and Theorem \ref{thm:conj-perturb-by-semis}.  Furthermore, the properties of the bi-free Fisher with respect to completely positive maps are in analogy with those from \cites{V1998-2, S1998,  CS2020} as the following shows.

\begin{remark}
\label{rem:remarks-about-fisher-info}
\begin{enumerate}[(i)]
\item In the case that $B = \bC$ and $\eta$ is unital, Definition \ref{defn:bi-fisher} immediately reduces down to the bi-free Fisher information in \cite{CS2020}*{Definition 5.1}  by Remark \ref{rem:conjugate-variable-immediate-remarks}.
\item In the case we are in the context of Example \ref{exam:factors} with $m = 0$, $C_\ell = B_\ell$, and $C_r = B_r$, Definition \ref{defn:bi-fisher} immediately reduces down to the free Fisher information with respect to a complete positive map from \cite{S1998}*{Definition 4.1}.
\item Note
\begin{align*}
 \Phi^*(& \vX \sqcup \vY : (C_\ell, C_r), (\{\eta_{\ell, i}\}^n_{i=1}, \{\eta_{r, j}\}^m_{j=1})) \\
&= \sum^n_{i=1}  \Phi^*\left(X_i \sqcup \emptyset :  \left(\alg\left(C_\ell, \widehat{\vX}_i\right), \alg\left(C_r,  \vY\right)\right), \eta_{\ell, i}, \tau  \right) \\
& \quad + \sum^m_{j=1}  \Phi^*\left(\emptyset \sqcup Y_j:  \left(\alg\left(C_\ell, \vX \right), \alg\left(C_r,  \widehat{\vY}_j\right)\right), \eta_{r, j}, \tau  \right).
\end{align*}
\item \label{rem:remarks-about-fisher-info:part:self-adjoint}If $\vX = (X_1, X_1^*, \ldots, X_n, X_n^*)$ and $\vY = (Y_1, Y_1^*, \ldots, Y_m, Y_m^*)$, then Lemma \ref{lem:non-self-adjoint-conjugate-variables-to-self-adjoint} implies
\[
\Phi^*(\vX \sqcup \vY : (C_\ell, C_r), (\{\eta_{\ell, i}\}^n_{i=1}, \{\eta_{r, j}\}^m_{j=1})) = \frac{1}{2} \Phi^*(\vX' \sqcup \vY' : (C_\ell, C_r), (\{\eta_{\ell, i}\}^n_{i=1}, \{\eta_{r, j}\}^m_{j=1}))
\]
where $\vX' = (\Re(X_1), \Im(X_1), \ldots, \Re(X_n), \Im(X_n))$ and $\vY' = (\Re(Y_1), \Im(Y_1), \ldots, \Re(Y_m), \Im(Y_m))$.
\item  \label{rem:remarks-about-fisher-info:part:smaller-amalgamation} In the context of Proposition \ref{prop:conj-variables-smaller-amalgamating-algebra} (i.e. reducing the $B$-$B$-non-commutative probability space to a $D$-$D$-non-commutative probability space),
\[
\Phi^*(\vX \sqcup \vY : (C_\ell, C_r), (\{\eta_{\ell, i}\}^n_{i=1}, \{\eta_{r, j}\}^m_{j=1})) = \Phi^*(\vX \sqcup \vY : (C_\ell, C_r), (\{\eta_{\ell, i} \circ F\}^n_{i=1}, \{\eta_{r, j}  \circ F\}^m_{j=1})).
\]
\item  \label{rem:remarks-about-fisher-info:part:scaling}By Lemma \ref{lem:conj-var-scaling}, for all $\lambda \in \bC \setminus \{0\}$
\[
\Phi^*(\lambda \vX \sqcup \lambda  \vY : (C_\ell, C_r), (\{\eta_{\ell, i}\}^n_{i=1}, \{\eta_{r, j}\}^m_{j=1})) = \frac{1}{|\lambda|^2}\Phi^*(\vX \sqcup \vY : (C_\ell, C_r), (\{\eta_{\ell, i}\}^n_{i=1}, \{\eta_{r, j}\}^m_{j=1})).
\]
\item  \label{rem:remarks-about-fisher-info:part:smaller-algebra} In the context of Lemma \ref{lem:con-var-reducing-alg-in-presence} (i.e. $(D_\ell, D_r)$ a smaller pair of $B$-algebras than $(C_\ell, C_r)$), 
\[
\Phi^*( \vX \sqcup    \vY : (D_\ell, D_r), (\{\eta_{\ell, i}\}^n_{i=1}, \{\eta_{r, j}\}^m_{j=1})) \leq \Phi^*(  \vX \sqcup    \vY : (C_\ell, C_r), (\{\eta_{\ell, i}\}^n_{i=1}, \{\eta_{r, j}\}^m_{j=1})).
\]
\item  \label{rem:remarks-about-fisher-info:part:adding-bi-free-algebra}In the context of Proposition \ref{prop:con-var-do-not-change-adding-bi-free-parts} (i.e. adding in a bi-free pair of $B$-algebras),
\[
\Phi^*(  \vX \sqcup    \vY : (\alg (D_\ell, C_\ell), \alg (D_r, C_r)), (\{\eta_{\ell, i}\}^n_{i=1}, \{\eta_{r, j}\}^m_{j=1})) = \Phi^*(  \vX \sqcup    \vY : (C_\ell, C_r), (\{\eta_{\ell, i}\}^n_{i=1}, \{\eta_{r, j}\}^m_{j=1})).
\]
\item If in addition to the assumptions of Definition \ref{defn:bi-fisher} $\vX' \in A_\ell^{n'}$, $\vY' \in A_r^{m'}$, $(\{\eta'_{\ell, i}\}^{n'}_{i=1}, \{\eta'_{r, j}\}^{m'}_{j=1})$ is a collection of completely positive maps on $B$, and $(D_\ell, D_r)$ is a pair of $B$-algebras then, by (iii) and (vii),
\begin{align*}
\Phi^*( &  \vX, \vX' \sqcup    \vY, \vY' : (\alg(C_\ell, D_\ell) , \alg(C_r, D_r) ), (\{\eta_{\ell, i}\}^n_{i=1} \cup \{\eta'_{\ell, i}\}^{n'}_{i=1}, \{\eta_{r, j}\}^m_{j=1} \cup \{\eta'_{r, j}\}^{m'}_{j=1})) \\
&\geq \Phi^*(  \vX \sqcup    \vY : (C_\ell, C_r), (\{\eta_{\ell, i}\}^n_{i=1}, \{\eta_{r, j}\}^m_{j=1})) + \Phi^*(  \vX' \sqcup    \vY' : (D_\ell, D_r), (\{\eta'_{\ell, i}\}^{n'}_{i=1}, \{\eta'_{r, j}\}^{m'}_{j=1}))
\end{align*}
\item In the context of (ix) with the additional assumption that 
\[
\left(\alg(C_\ell, \vX), \alg(C_r, \vY)    \right) \qqand \left(\alg(D_\ell, \vX'), \alg(D_r, \vY')    \right)
\]
are bi-free with amalgamation over $B$ with respect to $E$,  Proposition \ref{prop:con-var-do-not-change-adding-bi-free-parts} implies that 
\begin{align*}
\Phi^*( &  \vX, \vX' \sqcup    \vY, \vY' : (\alg(C_\ell, D_\ell) , \alg(C_r, D_r) ), (\{\eta_{\ell, i}\}^n_{i=1} \cup \{\eta'_{\ell, i}\}^{n'}_{i=1}, \{\eta_{r, j}\}^m_{j=1} \cup \{\eta'_{r, j}\}^{m'}_{j=1})) \\
&= \Phi^*(  \vX \sqcup    \vY : (C_\ell, C_r), (\{\eta_{\ell, i}\}^n_{i=1}, \{\eta_{r, j}\}^m_{j=1})) + \Phi^*(  \vX' \sqcup    \vY' : (D_\ell, D_r), (\{\eta'_{\ell, i}\}^{n'}_{i=1}, \{\eta'_{r, j}\}^{m'}_{j=1}))
\end{align*}
\end{enumerate}
\end{remark}

Unsurprisingly, more complicated properties of free Fisher information extend.

\begin{proposition}[Bi-Free Stam Inequality]
\label{prop:bi-free-stam-inequality}
Let $(A,E,\varepsilon,\tau)$ be an analytical $B$-$B$-non-commutative probability space, let $\{\eta_{\ell, i}\}^n_{i=1}$ and $\{\eta_{r, j}\}^m_{j=1}$ be completely positive maps from $B$ to $B$, let $\vX, \vX' \in A_\ell^n$ and $\vY, \vY' \in A_r^m$, and let $(C_\ell, C_r)$ and $(D_\ell, D_r)$ be pairs of $B$-algebras of $A$ such that
\[
\left(\alg(C_\ell, \vX), \alg(C_r, \vY)    \right) \qqand \left(\alg(D_\ell, \vX'), \alg(D_r, \vY')    \right)
\]
are bi-free with amalgamation over $B$ with respect to $E$.  Then
\begin{align*}
& \left(   \Phi^*(  \vX + \vX' \sqcup    \vY +  \vY' : (\alg(C_\ell, D_\ell) , \alg(C_r, D_r) ), (\{\eta_{\ell, i}\}^n_{i=1}, \{\eta_{r, j}\}^m_{j=1} )) \right)^{-1}  \\
&\geq \left(  \Phi^*(  \vX \sqcup    \vY : (\alg(C_\ell, D_\ell) , \alg(C_r, D_r) ), (\{\eta_{\ell, i}\}^n_{i=1}, \{\eta_{r, j}\}^m_{j=1} )) \right)^{-1} \\
& \quad+ \left(  \Phi^*(  \vX' \sqcup   \vY' : (\alg(C_\ell, D_\ell) , \alg(C_r, D_r) ), (\{\eta_{\ell, i}\}^n_{i=1}, \{\eta_{r, j}\}^m_{j=1} ))      \right)^{-1}.
\end{align*}
\end{proposition}
\begin{proof}
Let
\begin{align*}
P_0 & : L_2(\A, \varphi) \to L_2(B, \tau_B) \\
P_1 & : L_2(\A, \varphi) \to \overline{\alg(C_\ell, C_r, \vX, \vY)}^{\left\|\, \cdot \, \right\|_\tau}    \\
P_2 & : L_2(\A, \varphi) \to \overline{\alg(D_\ell, D_r, \vX', \vY')}^{\left\|\, \cdot \, \right\|_\tau}
\end{align*}
be the orthogonal projections onto their co-domains.  Note that if
\[
Z \in \alg(C_\ell, C_r, \vX, \vY)  \qand \alg(D_\ell, D_r, \vX', \vY'),
\]
then bi-freeness implies
\[
\widetilde{E}(ZZ') =  \widetilde{E}\left(Z \widetilde{E}(Z')\right).
\]
Indeed, this is easily seen as if $Z$ and $Z'$ are monomials, then any cumulant of the monomial $ZZ'$ corresponding to a bi-non-crossing partition is non-zero if and only if it decomposes into a bi-non-crossing partition on $Z$ union a bi-non-crossing partition on $Z'$.  Thus $P_1P_2 = P_2P_1 = P_0$.

The remainder of the proof can then be read from \cite{S1998}*{Proposition 4.5}, \cite{CS2020}*{Proposition 5.8}, or even \cite{V1998-2}*{Proposition 6.5}.
\end{proof}

\begin{proposition}[Bi-Free Cramer-Rao Inequality]
\label{prop:cramer-rao}
Let $(A,E,\varepsilon,\tau)$ be an analytical $B$-$B$-non-commutative probability space, let $\{\eta_{\ell, i}\}^n_{i=1}$ and $\{\eta_{r, j}\}^m_{j=1}$ be completely positive maps from $B$ to $B$, let $\vX \in A_\ell^n$ and $\vY \in A_r^m$ consist of self-adjoint operators and let $(C_\ell, C_r)$ be a pair of $B$-algebras of $A$.  Then
\[
\Phi^*(\vX \sqcup \vY : (C_\ell, C_r), (\{\eta_{\ell, i}\}^n_{i=1}, \{\eta_{r, j}\}^m_{j=1})) \tau\left(\sum^n_{i=1} X_i^2 + \sum^m_{j=1} Y_j^2 \right) \geq \left(\sum^n_{i=1} \tau_B(\eta_{\ell, i}(1)) + \sum^m_{j=1} \tau_B(\eta_{r, j}(1))  \right)^2.
\]
Moreover, equality holds if $(\vX, \vY)$ are $(\{\eta_{\ell, i}\}^n_{i=1}, \{\eta_{r, j}\}^m_{j=1})$-bi-semicircular elements and 
\[
\{(C_\ell, C_r)\} \cup \{(\alg(B_\ell, X_i), B_r)\}^n_{i=1}\cup \{(B_\ell, \alg(B_r, Y_j))\}^m_{j=1}
\]
are bi-free with amalgamation over $B$ with respect to $E$.  The converse holds when $C_\ell = B_\ell$ and $C_r =  B_r$.
\end{proposition}
\begin{proof}
The result follow from the obvious modifications to \cite{CS2020}*{Proposition 5.10}.  Also see \cite{S1998}*{Proposition 4.6} and \cite{V1998-2}*{Proposition 6.9}.
\end{proof}

Similarly, limits behave as one expects based on \cites{CS2020, V1998-2, S1998}.

\begin{proposition} 
\label{prop:fisher-limits}
Let $(A,E,\varepsilon,\tau)$ be an analytical $B$-$B$-non-commutative probability space, let $\{\eta_{\ell, i}\}^n_{i=1}$ and $\{\eta_{r, j}\}^m_{j=1}$ be completely positive maps from $B$ to $B$, let $\vX \in A_\ell^n$ and $\vY \in A_r^m$ consist of self-adjoint operators and let $(C_\ell, C_r)$ be a pair of $B$-algebras of $A$.   Suppose further  for each $k \in \bN$ that $\vX^{(k)} \in A_\ell^n$ and $ \vY^{(k)} \in A_r^m$ are tuples of self-adjoint elements in $A$ such that
\begin{align*}
& \limsup_{k \to \infty} \left\|X^{(k)}_i\right\| < \infty, \\
& \limsup_{k \to \infty} \left\|Y^{(k)}_j\right\| < \infty, \\
& s\text{-}\lim_{k \to \infty} X^{(k)}_i = X_i, \text{ and} \\
& s\text{-}\lim_{k \to \infty} Y^{(k)}_j = Y_j
\end{align*}
for all $1 \leq i \leq n$ and $1 \leq j \leq m$ (where the strong limit is computed as bounded linear maps acting on $L_2(A, \tau)$).  Then
\begin{align*}
\liminf_{k \to \infty}  \Phi^*\left(\vX^{(k)} \sqcup \vY^{(k)} : (C_\ell, C_r), (\{\eta_{\ell, i}\}^n_{i=1}, \{\eta_{r, j}\}^m_{j=1}) \right)  \geq  \Phi^*(\vX \sqcup \vY : (C_\ell, C_r), (\{\eta_{\ell, i}\}^n_{i=1}, \{\eta_{r, j}\}^m_{j=1})) 
\end{align*}
\end{proposition}

The proof of Proposition \ref{prop:fisher-limits}  becomes identical to \cite{CS2020}*{Proposition 5.12} once the following lemma is established.  Also see \cite{S1998}*{Proposition 4.7} and \cite{V1998-2}*{Proposition 6.10}.

\begin{lemma}
Under the assumptions of Proposition \ref{prop:fisher-limits} along with the additional assumptions that
\[
\xi_k = J_\ell\left(X_1^{(k)} : \left( \alg\left( C_\ell,  \widehat{\vX}^{(k)}_1\right) , \alg(C_r, \vY^{(k)}) \right), \eta_{\ell, 1} \right)
\]
exist and are bounded in $L_2$-norm by some constant $K > 0$, it follows that
\[
\xi = J_\ell\left(X_1 : \left( \alg\left( C_\ell,  \widehat{\vX}_1\right) , \alg(C_r, \vY) \right), \eta_{\ell, 1} \right)
\]
exists and is equal to
\[
w\text{-}\lim_{k \to \infty} P\left(\xi_k\right)
\]
where $P$ is the orthogonal projection of $L_2(A, \tau)$ onto $\overline{\alg(C_\ell,C_r, \vX, \vY)}^{\left\| \, \cdot \, \right\|_\tau}$.

If, in addition,
\[
\limsup_{k \to \infty} \left\|\xi_k\right\|_2 \leq \left\| \xi \right\|_2
\]
then
\[
\lim_{k \to \infty}  \left\|\xi_k- \xi\right\|_2 =0.
\]

The same holds with $X_1$ replaced with $X_i$, and a similar result holds for the right.
\end{lemma}
\begin{proof}
The proof of this result follows from the same sequence of steps as \cite{CS2020}*{Lemma 5.13} using the analytical operator-valued bi-free cumulants.
\end{proof}

\begin{corollary}
Under the assumptions of Proposition \ref{prop:fisher-limits}, if in addition 
\[
(B_\ell(\vX), B_r( \vY) ) \qqand \left(\alg\left( C_\ell, \vX^{(k)}\right) ,  \alg\left( C_r, \vY^{(k)}\right) \right)
\]
are bi-free for all $k$ and
\[
\lim_{k \to \infty} \left\|X^{(k)}_i\right\| = \lim_{k \to \infty} \left\|Y_j^{(k)}\right\| = 0
\]
for all $1 \leq i \leq n$ and $1 \leq j \leq m$, then
\begin{align*}
\lim_{k \to \infty} \Phi^*\left( \vX + \vX^{(k)}\sqcup \vY + \vY^{(k)} : (C_\ell, C_r),  (\{\eta_{\ell, i}\}^n_{i=1}, \{\eta_{r, j}\}^m_{j=1}) \right)=  \Phi^*(\vX \sqcup \vY :  (\{\eta_{\ell, i}\}^n_{i=1}, \{\eta_{r, j}\}^m_{j=1})).
\end{align*}
Furthermore, if $C_\ell = B_\ell$, $C_r = B_r$, and
\[
\Phi^*(\vX \sqcup \vY :  (\{\eta_{\ell, i}\}^n_{i=1}, \{\eta_{r, j}\}^m_{j=1}) ) < \infty,
\]
then 
\[
J_\ell\left( X^{(k)}_i  :   \left( B_\ell \left ( \widehat{(\vX + \vX^{(k)})}_i \right), B_r\left( \vY + \vY^{(k)} \right) \right),  \eta_{\ell, i} \right)
\] 
tends to 
\[
J_\ell\left(X_i : \left(B_\ell\left( \hat{\vX}_i \right), B_r \left( \vY\right)\right),  \eta_{\ell, i} \right)
\]
in $\tau$-norm.  A similar result holds for right bi-free conjugate variables.
\end{corollary}
\begin{proof}
The proof is identical to \cite{CS2020}*{Corollary 5.14} and thus is omitted.
\end{proof}

\begin{theorem}
\label{thm:fisher-info-with-semi-perturbation}
Let $(A,E,\varepsilon,\tau)$ be an analytical $B$-$B$-non-commutative probability space, let $\{\eta_{\ell, i}\}^n_{i=1}$ and $\{\eta_{r, j}\}^m_{j=1}$ be completely positive maps from $B$ to $B$, let $\vX \in A_\ell^n$ and $\vY \in A_r^m$ be tuples of self-adjoint operators, let $(\{S_i\}^n_{i=1}, \{D_j\}^m_{j=1})$ be a collection of $(\{\eta_{\ell, i}\}^n_{i=1}, \{\eta_{r, j}\}^m_{j=1})$ bi-semicircular operators in $A$ and let $(C_\ell, C_r)$ be pairs of $B$-algebras of $A$ such that
\[
\{(\alg(C_\ell, \vX), \alg(C_r, \vY)\rangle)\} \cup \{(\alg(B_\ell, S_i), B_r)\}^n_{i=1} \cup  \{(B_\ell, \alg(B_r, D_j))\}^m_{j=1}
\]
are bi-free.  Then, the map
\[
h : [0, \infty) \ni t \mapsto  \Phi^*\left(\vX + \sqrt{t} \vS \sqcup \vY + \sqrt{t} \vD : (C_\ell, C_r), (\{\eta_{\ell, i}\}^n_{i=1}, \{\eta_{r, j}\}^m_{j=1})   \right)
\]
is decreasing, right continuous, and
\[
\frac{K_2^2}{K_1 + K_2 t} \leq  h(t) \leq \frac{1}{t}K_3
\]
where
\begin{align*}
K_1 &= \tau\left(\sum^n_{i=1} X_i^2 + \sum^m_{j=1} Y_j^2 \right),\\
K_2 &= \sum^n_{i=1} \tau_B(\eta_{\ell, i}(1_B))  + \sum^m_{j=1} \tau_B(\eta_{r, j}(1_B))  \text{ and}\\
K_3 &= \sum^n_{i=1} \tau_B(\eta_{\ell, i}(1_B))^2  + \sum^m_{j=1} \tau_B(\eta_{r, j}(1_B))^2.
\end{align*}

Moreover if $(\vX, \vY)$ is the $(\{\eta_{\ell, i}\}^n_{i=1}, \{\eta_{r, j}\}^m_{j=1})$-bi-semicircular distribution and 
\[
\{(C_\ell, C_r)\} \cup \{(\alg(B_\ell, X_i), B_r)\}^n_{i=1}\cup \{(B_\ell, \alg(B_r, Y_j))\}^m_{j=1}
\]
are bi-free with amalgamation over $B$ with respect to $E$, then $h(t) = \frac{K_2^2}{K_1^2 + K_2 t}$ for all $t$.

Finally, if $C_ \ell = B_\ell$, $C_r = B_r$, and $h(t) = \frac{K_2^2}{K_1 + K_2 t}$ for all $t$, then $(\vX, \vY)$ is the $(\{\eta_{\ell, i}\}^n_{i=1}, \{\eta_{r, j}\}^m_{j=1})$-bi-semicircular distribution.
\end{theorem}
\begin{proof}
The proof becomes identical to \cite{CS2020}*{Theorem 5.15} using the above and the fact that
\[
\tau\left(\left(X_i + \sqrt{t}S_i\right)^2\right) = \tau\left(X_i^2\right) + t \tau\left(S_i^2\right) = \tau\left(X_i^2\right) + t \tau_B\left(\eta_{\ell, i}(1_B)\right),
\]
with an analogous computation on the right for use in the lower bound computation, in conjunction with the Bi-Free Cramer-Rao Inequality (Proposition \ref{prop:cramer-rao}).
\end{proof}

\section{Bi-Free Entropy with Respect to Completely Positive Maps}
\label{sec:Entropy}

With the construction and properties of the bi-free Fisher information with respect to completely positive maps complete, the construction and properties of bi-free entropy with respect to completely positive maps follows easily by extending results from \cite{S1998} and \cite{CS2020} with similar proofs.

\begin{definition}
\label{defn:bi-free-entropy}
Let $(A,E,\varepsilon,\tau)$ be an analytical $B$-$B$-non-commutative probability space, let $\{\eta_{\ell, i}\}^n_{i=1}$ and $\{\eta_{r, j}\}^m_{j=1}$ be completely positive maps from $B$ to $B$,  let $(C_\ell, C_r)$ be pairs of $B$-algebras of $A$, and let $\vX \in A_\ell^n$ and $\vY \in A_r^m$ be tuples of self-adjoint operators.  The \emph{relative bi-free entropy of $(\vX, \vY)$ with respect to $(\{\eta_{\ell, i}\}^n_{i=1}, \{\eta_{r, j}\}^m_{j=1})$ in the presence of $(B_\ell, B_r)$} is defined to be
\begin{align*}
\chi^* & (\vX \sqcup \vY : (C_\ell, C_r), (\{\eta_{\ell, i}\}^n_{i=1}, \{\eta_{r, j}\}^m_{j=1})) \\
&= \frac{K}{2} \ln(2\pi e) + \frac{1}{2} \int^\infty_0 \left(\frac{K}{1+t} - \Phi^*\left( \vX + \sqrt{t} \vS \sqcup \vY + \sqrt{t} \vD : (C_\ell, C_r), (\{\eta_{\ell, i}\}^n_{i=1}, \{\eta_{r, j}\}^m_{j=1}) \right) \right) \, dt,
\end{align*}
where 
\[
K= \sum^n_{i=1} \tau_B(\eta_{\ell, i}(1_B))  + \sum^m_{j=1} \tau_B(\eta_{r, j}(1_B))
\]
and $(\{S_i\}^n_{i=1}, \{D_j\}^m_{j=1})$ is a collection of $(\{\eta_{\ell, i}\}^n_{i=1}, \{\eta_{r, j}\}^m_{j=1})$ bi-semicircular operators such that
\[
\{(\alg(C_\ell, \vX), \alg(C_r, \vY)\rangle)\} \cup \{(B_\ell(S_i), B_r)\}^n_{i=1} \cup  \{(B_\ell, B_r(D_j))\}^m_{j=1}
\]
are bi-free (note such semicircular operators can be included in $A$ by Theorem \ref{thm:can-add-semis}).

In the case that $C_\ell = B_\ell$ and $C_r = B_r$, we use $\chi^*(\vX \sqcup \vY : (\{\eta_{\ell, i}\}^n_{i=1}, \{\eta_{r, j}\}^m_{j=1}))$ to denote the bi-free entropy.  If in addition $\eta_{ \ell, i} = \eta_{r, j} = \eta$ for all $i$ and $j$, we use $\chi^*(\vX \sqcup \vY : \eta)$ to denote the bi-free entropy.
\end{definition}

We note there is a slight change in the normalization used in Definition \ref{defn:bi-free-entropy} over that used in \cite{S1998}*{Definition 8.1}.  Generally this makes no real difference other than making some of the bounds in this section nice, such as the following one.

\begin{proposition}
In the context of Definition \ref{defn:bi-free-entropy}, if
\[
K_1 = \tau\left(\sum^n_{i=1} X_i^2 + \sum^m_{j=1} Y_j^2 \right),
\]
then
\[
\chi^* (\vX \sqcup \vY : (C_\ell, C_r), (\{\eta_{\ell, i}\}^n_{i=1}, \{\eta_{r, j}\}^m_{j=1})) \leq \frac{K}{2} \ln\left( \frac{2 \pi e }{K} K_1\right).
\]
Moreover, equality holds when $(\vX, \vY)$ are $ (\{\eta_{\ell, i}\}^n_{i=1}, \{\eta_{r, j}\}^m_{j=1})$-bi-semicircular operators such that $\{(C_\ell, C_r)\} \cup \{(\alg(B_\ell, X_i), B_r)\}^n_{i=1} \cup \{(B_\ell, \alg(B_r, Y_j))\}^m_{j=1}$ are bi-free and if $C_\ell = B_\ell$ and $C_r = B_r $, this is the only setting where equality holds.
\end{proposition}
\begin{proof}
The proof is identical to \cite{CS2020}*{Proposition 6.5} in conjunction with Theorem \ref{thm:fisher-info-with-semi-perturbation}.
\end{proof}
\begin{remark}
\begin{enumerate}[(i)]
\item In the case that $B = \bC$ and $\eta$ is unital, Definition \ref{defn:bi-free-entropy} produces the non-microstate bi-free entropy from \cite{CS2020}*{Definition 6.1}.
\item In the setting of Example \ref{exam:factors}, when $C_r = B_r$ and $\eta_{\ell, i} = \eta_{r, j} = \eta$ for all $i$ and $j$, Definition \ref{defn:bi-free-entropy} produces the free entropy with respect to a completely positive map from \cite{S1998}*{Definition 8.1} modulo an additive constant (which is 0 in the case $\eta$ is unital).
\end{enumerate}
\end{remark}

Of course, due to the fact that the bi-free Fisher information from Section \ref{sec:Fisher} behaves analogously to the Fisher informations considered in \cites{CS2020, S1998}, results for the behaviour of entropy automatically generalize.

\begin{proposition}
Using Remark \ref{rem:remarks-about-fisher-info}, the following hold:
\begin{enumerate}
\item[(v)]  In the context of Proposition \ref{prop:conj-variables-smaller-amalgamating-algebra} (i.e. reducing the $B$-$B$-non-commutative probability space to a $D$-$D$-non-commutative probability space),
\[
\chi^*(\vX \sqcup \vY : (C_\ell, C_r), (\{\eta_{\ell, i}\}^n_{i=1}, \{\eta_{r, j}\}^m_{j=1})) = \chi^*(\vX \sqcup \vY : (C_\ell, C_r), (\{\eta_{\ell, i} \circ F\}^n_{i=1}, \{\eta_{r, j}  \circ F\}^m_{j=1})).
\]
\item[(vi)] For all $\lambda \in \mathbb{R} \setminus \{0\}$
\begin{align*}
\chi^* & (\lambda \vX \sqcup \lambda  \vY : (C_\ell, C_r), (\{\eta_{\ell, i}\}^n_{i=1}, \{\eta_{r, j}\}^m_{j=1})) \\
& = \left(\sum^n_{i=1} \tau_B(\eta_{\ell, i}(1_B)) + \sum^m_{j=1} \tau_B(\eta_{r,j}(1_B))\right) \ln|\lambda| +  \chi^*(\vX \sqcup \vY : (C_\ell, C_r), (\{\eta_{\ell, i}\}^n_{i=1}, \{\eta_{r, j}\}^m_{j=1})).
\end{align*}
\item[(vii)]  In the context of Lemma \ref{lem:con-var-reducing-alg-in-presence} (i.e. $(D_\ell, D_r)$ a smaller pair of $B$-algebras than $(C_\ell, C_r)$),
\[
\chi^*( \vX \sqcup    \vY : (D_\ell, D_r), (\{\eta_{\ell, i}\}^n_{i=1}, \{\eta_{r, j}\}^m_{j=1})) \geq \chi^*(  \vX \sqcup    \vY : (C_\ell, C_r), (\{\eta_{\ell, i}\}^n_{i=1}, \{\eta_{r, j}\}^m_{j=1})).
\]
\item[(viii)] In the context of Proposition \ref{prop:con-var-do-not-change-adding-bi-free-parts}  (i.e. adding in a bi-free pair of $B$-algebras),
\[
\chi^*(  \vX \sqcup    \vY : (\alg(D_\ell, C_\ell), \alg(D_r, C_r)), (\{\eta_{\ell, i}\}^n_{i=1}, \{\eta_{r, j}\}^m_{j=1})) = \chi^*(  \vX \sqcup    \vY : (C_\ell, C_r), (\{\eta_{\ell, i}\}^n_{i=1}, \{\eta_{r, j}\}^m_{j=1})).
\]
\item[(ix)] If in addition to the assumptions of Definition \ref{defn:bi-free-entropy} we have $\vX' \in A_\ell^{n'}$, $\vY' \in A_r^{m'}$, $(\{\eta'_{\ell, i}\}^{n'}_{i=1}, \{\eta'_{r, j}\}^{m'}_{j=1})$ is a collection of completely positive maps on $B$, and $(D_\ell, D_r)$ is a pair of $B$-algebras, then
\begin{align*}
\chi^*( &  \vX, \vX' \sqcup    \vY, \vY' : (\alg(C_\ell, D_\ell) , \alg(C_r, D_r) ), (\{\eta_{\ell, i}\}^n_{i=1} \cup \{\eta'_{\ell, i}\}^{n'}_{i=1}, \{\eta_{r, j}\}^m_{j=1} \cup \{\eta'_{r, j}\}^{m'}_{j=1})) \\
&\leq \chi^*(  \vX \sqcup    \vY : (C_\ell, C_r), (\{\eta_{\ell, i}\}^n_{i=1}, \{\eta_{r, j}\}^m_{j=1})) + \chi^*(  \vX' \sqcup    \vY' : (D_\ell, D_r), (\{\eta'_{\ell, i}\}^{n'}_{i=1}, \{\eta'_{r, j}\}^{m'}_{j=1}))
\end{align*}
\item[(x)] In the context of (ix) with the additional assumption that 
\[
\left(\alg(C_\ell, \vX), \alg(C_r, \vY)    \right) \qqand \left(\alg(D_\ell, \vX'), \alg(D_r, \vY')    \right)
\]
are bi-free with amalgamation over $B$ with respect to $E$,  Proposition \ref{prop:con-var-do-not-change-adding-bi-free-parts} implies that 
\begin{align*}
\chi^*( &  \vX, \vX' \sqcup    \vY, \vY' : (\alg(C_\ell, D_\ell) , \alg(C_r, D_r) ), (\{\eta_{\ell, i}\}^n_{i=1} \cup \{\eta'_{\ell, i}\}^{n'}_{i=1}, \{\eta_{r, j}\}^m_{j=1} \cup \{\eta'_{r, j}\}^{m'}_{j=1})) \\
&= \chi^*(  \vX \sqcup    \vY : (C_\ell, C_r), (\{\eta_{\ell, i}\}^n_{i=1}, \{\eta_{r, j}\}^m_{j=1})) + \chi^*(  \vX' \sqcup    \vY' : (D_\ell, D_r), (\{\eta'_{\ell, i}\}^{n'}_{i=1}, \{\eta'_{r, j}\}^{m'}_{j=1}))
\end{align*}
\end{enumerate}
\end{proposition}

Using Proposition \ref{prop:fisher-limits}, Theorem \ref{thm:fisher-info-with-semi-perturbation} and the same arguments as \cite{CS2020}*{Proposition 6.7}, the following holds.

\begin{proposition}
Under the assumptions of Definition \ref{defn:bi-free-entropy}, if for each $k \in \bN$ there exists self-adjoint tuples $\vX^{(k)} \in A^n_\ell$ and $\vY^{(k)} \in A_r^m$ such that
\begin{align*}
& \limsup_{k \to \infty} \left\|X^{(k)}_i\right\| < \infty, \\
& \limsup_{k \to \infty} \left\|Y^{(k)}_j\right\| < \infty, \\
& s\text{-}\lim_{k \to \infty} X^{(k)}_i = X_i, \text{ and} \\
& s\text{-}\lim_{k \to \infty} Y^{(k)}_j = Y_j
\end{align*}
for all $1 \leq i \leq n$ and $1 \leq j \leq m$ (with the strong limit computed as bounded linear maps acting on $L_2(A, \tau)$), then
\begin{align*}
\limsup_{k \to \infty} \chi^*\left(\vX^{(k)} \sqcup \vY^{(k)} : (C_\ell, C_r), (\{\eta_{\ell, i}\}^n_{i=1}, \{\eta_{r,j}\}^m_{j=1})\right) \leq \chi^*(\vX \sqcup \vY : (C_\ell, C_r), (\{\eta_{\ell, i}\}^n_{i=1}, \{\eta_{r,j}\}^m_{j=1})).
\end{align*}
\end{proposition}

Using the previous proposition together with Theorem \ref{thm:fisher-info-with-semi-perturbation} and the same arguments as \cite{CS2020}*{Proposition 6.8}, the following holds.

\begin{proposition}
Under the assumptions of Definition \ref{defn:bi-free-entropy}, suppose $(\{S_i\}^n_{i=1}, \{D_j\}^m_{j=1})$ is a collection of $(\{\eta_{\ell, i}\}^n_{i=1}, \{\eta_{r, j}\}^m_{j=1})$ bi-semicircular operators such that
\[
(\alg(C_\ell, \vX), \alg(C_r, \vY)\rangle) \cup\{(\alg(B_\ell, S_i), B_r)\}^n_{i=1}\cup  \{(B_\ell, \alg(B_r, D_j))\}^m_{j=1}
\]
are bi-free.  For $t \in [0, \infty)$, let
\[
g(t) = \chi^*\left(\vX + \sqrt{t} \vS \sqcup \vY + \sqrt{t} \vD : (C_\ell, C_r), (\{\eta_{\ell, i}\}^n_{i=1}, \{\eta_{r, j}\}^m_{j=1}) \right).
\]
Then $g : [0, \infty) \to \bR \cup \{-\infty\}$ is a concave, continuous, increasing function such that $g(t) \geq \frac{K}{2} \ln(2 \pi e t)$ where
\[
K = \sum^n_{i=1} \tau_B(\eta_{\ell, i}(1_B))  + \sum^m_{j=1} \tau_B(\eta_{r, j}(1_B))
\]
and, when $g(t) \neq -\infty$, 
\[
\lim_{\epsilon \to 0+} \frac{1}{\epsilon} (g(t+\epsilon) - g(t)) = \frac{1}{2} \Phi^*\left( \vX + \sqrt{t} \vS \sqcup \vY + \sqrt{t} \vD : (C_\ell, C_r), (\{\eta_{\ell, i}\}^n_{i=1}, \{\eta_{r, j}\}^m_{j=1}) \right).
\]
\end{proposition}

Finally, using the Bi-Free Stam Inequality (Proposition \ref{prop:bi-free-stam-inequality}) together with the same proof as \cite{CS2020}*{Proposition 6.11} yields the following. 

\begin{proposition}
Under the assumptions of Definition \ref{defn:bi-free-entropy}, if
\[
\Phi^*\left(\vX \sqcup \vY : (C_\ell, C_r),(\{\eta_{\ell, i}\}^n_{i=1}, \{\eta_{r, j}\}^m_{j=1}) \right) < \infty,
\]
then
\begin{align*}
\chi^*(\vX \sqcup \vY : (B_\ell, B_r))  \geq \frac{K}{2} \ln\left(\frac{2\pi K e}{\Phi^*(\vX \sqcup \vY : (C_\ell, C_r),(\{\eta_{\ell, i}\}^n_{i=1}, \{\eta_{r, j}\}^m_{j=1})   )}\right)> - \infty,
\end{align*}
where
\[
K = \sum^n_{i=1} \tau_B(\eta_{\ell, i}(1_B))  + \sum^m_{j=1} \tau_B(\eta_{r, j}(1_B)).
\]
\end{proposition}

\section{Minimizing Bi-Free Fisher Information}
\label{sec:minimize-Fisher}

In this section, we will prove Theorem \ref{thm:minimizing-fisher-info} thereby describing the minimal value of the bi-free Fisher information of non-self-adjoint pairs of operators under certain distribution conditions.  Throughout the section, we will be working under the situation from Example \ref{canonicalexample} where $\A$ is a unital C$^*$-algebra, $\varphi : \A \to \bC$ is a state, $B = M_d(\bC)$ (the $d \times d$ matrices with complex entries) and $\tau_B = \tr_d$ (the normalized trace on $M_d(\bC)$).  Thus $A_d = \A \otimes M_d(\bC) \otimes M_d(\bC)^\op$, $E_d : A_d \to M_d(\bC)$ and $\tau_d : A_d \to \bC$ are defined such that
\[
E_d (Z \otimes b_1 \otimes b_2) = \varphi(Z) b_1b_2 \qqand \tau_d(Z \otimes b_1 \otimes b_2) = \varphi(Z) \tr_d(b_1b_2),
\]
for all $Z\in \A$ and $b_1, b_2 \in M_d(\bC)$.  We recall the following result that aids in computing moments in $(A_d, E_d, \varepsilon)$ where $\{E_{i,j}\}^n_{i,j=1} \subseteq M_d(\bC)$ are the canonical matrix units.

\begin{lemma}[\cite{S2016-1}*{Lemma 3.7}]
\label{lem:expanding-moments-of-matrices}
Let $(\A, \varphi)$ be a C$^*$-non-commutative probability space, let $\chi \in \{\ell, r\}^n$, and let
\[
Z_k = \begin{cases}
\sum^d_{i,j=1} z_{k; i,j} \otimes E_{i,j} \otimes I_d & \text{if } \chi(k) = \ell \\
\sum^d_{i,j=1} z_{k; i,j} \otimes I_d \otimes E_{i,j} & \text{if } \chi(k) = r
\end{cases}.
\]
Then
\[
E_d(Z_1 \cdots Z_n) = \sum^d_{\substack{i_1, \ldots, i_n = 1 \\ j_1, \ldots, j_n = 1}} \varphi(z_{1; i_1, j_1}  \cdots z_{n; i_n, j_n}) E_\chi((i_1, \ldots, i_n), (j_1, \ldots, j_n))
\]
where
\[
E_\chi((i_1, \ldots, i_n), (j_1, \ldots, j_n)) := E_{i_{s_\chi(1)}, j_{s_\chi(1)}} \cdots E_{i_{s_\chi(n)}, j_{s_\chi(n)}} \in M_d(\bC).
\]
\end{lemma}

To discuss conjugate variables, we need to consider $L_2(A_d, \tau_d)$.  It is not difficult to verify that $L_2(A_d, \tau_d)$ can be identified with the $d \times d$ matrices with entries in $L_2(\A, \varphi)$ where
\[
(Z \otimes I_d \otimes I_d) [\xi_{i,j}] = [Z \xi_{i,j}],
\]
$1_\A \otimes b \otimes I_d$ acts via left multiplication on $[\xi_{i,j}]$ and $1_\A \otimes I_d \otimes b$ acts via right multiplication on $[\xi_{i,j}]$ for all $[\xi_{i,j}] \in M_d(L_2(\A, \varphi))$ and $b \in M_d(\bC)$.

Next, we consider the subalgebra $D_d \subseteq M_d(\bC)$ of the diagonal matrices.  Clearly if $F_d : M_d(\bC) \to D_d$ is the canonical conditional expectation onto the diagonal, then $(A_d, F_d\circ E_d, \varepsilon|_{D_d \otimes D_d^{\op}}, \tau_d)$ is also an analytical $D_d$-$D_d$-non-commutative probability space (see Example \ref{exam:reducing-analytic-spaces-to-smaller-algebra}).

To begin stating our main result, let $x,y \in \A$.  For $d = 2$, let
\[
X = x \otimes E_{1,2} \otimes I_2 + x^* \otimes E_{2,1} \otimes I_2 \qqand Y = y \otimes I_2 \otimes  E_{1,2}  + y^* \otimes I_2 \otimes E_{2,1},
\]
which are then self-adjoint elements of $A_2$.  The pair $(X, Y)$ is intimately related to whether or not $(x,y)$ is bi-R-diagonal (\cite{K2019}*{Definition 2.12}).  Indeed, combining \cite{K2019}*{Proposition 2.21} and \cite{S2016-2}*{Theorem 4.9}, we  obtain the following proposition.

\begin{proposition}\label{prop:bi-R-diagonal-iff-bi-R-cyclic-iff-free-with-amal}
As described above, the following conditions are equivalent:
\begin{enumerate}[(i)]
    \item the pair $(x, y)$ is bi-R-diagonal,
    \item the pair $(X, Y)$ is bi-R-cyclic,
    \item the pair of algebras $(\alg(D_2, X), \alg(D_2, Y))$ is bi-free from $(M_2(\bC)_\ell, M_2(\bC)_r)$ with amalgamation over $D_2$ with respect to $F_2 \circ E_2$.
\end{enumerate}
\end{proposition}

In addition, the joint moments of $(X, Y)$ with respect to $\tau_d$ are not too difficult to describe.  Indeed, for any $n \in \bN$, $Z_1, Z_2, \ldots, Z_n \in \{X, Y\}$, $\chi \in \{\ell, r\}^n$, and $z_1, \ldots, z_n \in \{x,y\}$ such that
\[
\chi(k) = \begin{cases}
\ell & \text{if }Z_k = X \\
r & \text{if } Z_k = Y
\end{cases}\qqand z_k = \begin{cases}
x & \text{if }Z_k = X \\
y & \text{if } Z_k = Y
\end{cases},
\]
then
\[
E_2(Z_1 \cdots Z_n) = \sum^n_{k=1} \sum_{p_k \in \{1, *\}} \varphi(z_1^{p_1} \cdots z_n^{p_n}) (E_{1,2})^{p_{s_\chi(1)}}(E_{1,2})^{p_{s_\chi(2)}} \cdots (E_{1,2})^{p_{s_\chi(n)}}.
\]
Thus, if $n$ is odd, we see that $\tau_2(Z_1\cdots Z_n) = 0$ and  if $n$ is even, we see that
\[
\tau_2(Z_1 \cdots Z_n) = \frac{1}{2} \left(\varphi(z_1^{p_1} \cdots z_n^{p_n}) + \varphi(z_1^{q_1} \cdots z_n^{q_n})   \right),
\]
where
\[
p_{s_\chi(k)} = \begin{cases}
1 & \text{if }k \text{ is odd}\\
* & \text{if }k \text{ is even}
\end{cases} \qqand q_{s_\chi(k)} = \begin{cases}
* & \text{if }k \text{ is odd}\\
1 & \text{if }k \text{ is even}
\end{cases}
\]
(that is, the $1$'s and $*$'s alternate in the $\chi$-ordering).  Hence, the joint moments of $(X, Y)$ with respect to $\tau_d$ depend only on specific moments of $(x, y)$.  We let $\Delta_{X,Y}$ denote the set of all pairs $(x_0, y_0)$ in a C$^*$-non-commutative probability space $(\A_0, \varphi_0)$ such that if we apply the above procedure to $(x_0, y_0)$ resulting in $(X_0, Y_0)$, then $(X_0, Y_0)$ has the same joint distribution as $(X, Y)$ (so $\Delta_{X, Y} = \Delta_{X_0, Y_0}$).

One specific case worth mentioning is when $x$ and $y$ are  normal operators  with $[\alg(x,x^*), \alg(y, y^*)]=0$, thus defining  a probability measure $\mu$ on $\bC^2$. Then,  $X$ and $Y$ will be commuting self-adjoint operators and therefore their joint distribution gives rise to a compactly supported probability  measure $\mu_0$ on $\bR^2$ with moments
\[
\tau_2(X^n Y^m) = \begin{cases}
0 & \text{if } n+m\text{ is odd} \\
\varphi((x^*x)^i (y^*y)^j) & \text{if }n=2i \text{ and }m = 2j\\
\frac{1}{2}\varphi((x^*x)^i (xy^* + x^*y) (y^*y)^j)  & \text{if }n=2i+1 \text{ and }m = 2j+1
\end{cases}.
\]

One additional property is required in this section.  In particular, as we are attempting to generalize \cite{NSS1999}*{Theorem 1.1} which makes heavy use of traciality, we need a condition that lets us bypass the issue that $\tau_2$ is not tracial on $A_2$.  

\begin{definition}\label{defn:alternate-adjoint-flip}
Let $(\A, \varphi)$ be a C$^*$-non-commutative probability space and let $x,y \in \A$.  We say that $(x, y)$ is \emph{alternating adjoint flipping with respect to $\varphi$} if for any $n \in \bN$, $\chi \in \{\ell, r\}^{2n}$, and $z_1, \ldots, z_{2n} \in \{x,y\}$ such that
\[
z_k = \begin{cases}
x & \text{if } \chi(k) = \ell \\
y & \text{if } \chi(k) = r
\end{cases},
\]
we have that
\[
\varphi(z_1^{p_1} \cdots z_{2n}^{p_{2n}}) = \varphi(z_1^{q_1} \cdots z_{2n}^{q_{2n}}) 
\]
where
\[
p_{s_\chi(k)} = \begin{cases}
1 & \text{if }k \text{ is odd}\\
* & \text{if }k \text{ is even}
\end{cases} \qqand q_{s_\chi(k)} = \begin{cases}
* & \text{if }k \text{ is odd}\\
1 & \text{if }k \text{ is even}
\end{cases}.
\]
\end{definition}
\begin{remark}
If $(x, y)$ is alternating adjoint flipping, then the description of the joint moments of $(X, Y)$ above reduce to a nicer expression.  Furthermore we see that
\[
\varphi((x^*x)^m) = \varphi((xx^*)^m) \qqand \varphi((y^*y)^m) = \varphi((yy^*)^m)
\]
for all $m \in \bN$, so that $x^*x$ and $xx^*$ have the same distribution and $y^*y$ and $yy^*$ have the same distribution, which would be automatic if $\varphi$ was tracial when restricted to $\alg(x, x^*)$ and when restricted to $\alg(y,y^*)$ (a common assumption in bi-free probability).
\end{remark}

Of course, bi-Haar unitary pairs are trivially seen to be  alternating adjoint flipping, since any joint moment with an equal number of adjoint and non-adjoint terms is 1 and any joint moment with a differing number of adjoint and non-adjoint terms is 0.  Here is another example which is of use in this paper.

\begin{example}\label{exam:circular-pair}
Let $\H$ be any Hilbert space of dimension at least $4$, let $\F(\H)$ denote the Fock space generated by $\H$, let $\varphi_0$ be the vacuum vector state on $\mathcal{B}(\F(\H))$ and let $\{e_1, e_2, e_3, e_4\}$ be an orthonormal set.  For $i = 1,2$ let $s_i = l(e_i) + l^*(e_i)$ (i.e. left creation plus annihilation by $e_i$) and for $j=1,2$ let $d_j = r(e_{j+2}) + r^*(e_{j+2})$ (i.e. right creation and annihilation by $e_{j+2}$).  Thus $(\{s_1, s_2\}, \{d_1, d_2\})$ is a bi-free central limit distribution with variance 1 and covariance 0.

Let 
\[
c_\ell = \frac{1}{\sqrt{2}}(s_1 + i s_2) \qqand c_r = \frac{1}{\sqrt{2}} (d_1 + i d_2).
\]
We call the pair $(c_\ell, c_r)$ a \emph{bi-free circular pair} (with mean 0, variance 1, and covariance 0).

We claim that $(c_\ell, c_r)$ is an example of a bi-R-diagonal pair that is alternating adjoint flipping with respect to $\varphi_0$.  To see that $(c_\ell, c_r)$ is bi-R-diagonal, we note that any bi-free cumulant for $(\{s_1, s_2\}, \{d_1, d_2\})$ of order 1, of order greater than 3, or involving two different elements is 0.  As 
\begin{align*}
\kappa_{1_{(\ell, \ell)}}(c_\ell, c_\ell) &= \frac{1}{2}\kappa_{1_{(\ell, \ell)}}(s_1, s_1) + (i)^2 \frac{1}{2}\kappa_{1_{(\ell, \ell)}}(s_2, s_2)  =0, \\
\kappa_{1_{(\ell, \ell)}}(c^*_\ell, c^*_\ell) &= \frac{1}{2}\kappa_{1_{(\ell, \ell)}}(s_1, s_1) + (-i)^2 \frac{1}{2}\kappa_{1_{(\ell, \ell)}}(s_2, s_2)  =0,
\end{align*}
and similar computations hold on the right, we have that $(c_\ell, c_r)$ is bi-R-diagonal.

To see that $(c_\ell, c_r)$ is alternating adjoint flipping with respect to $\varphi_0$, first note that $\varphi_0$ is tracial when restricted to $\alg(s_1, s_2)$  as $s_1$ and $s_2$ are freely independent with respect to $\varphi_0$.  Hence for all $n \in \bN$
\[
\varphi_0((c_\ell^*c_\ell)^n) = \varphi_0((c_\ell c_\ell^*)^n).
\]
Moreover, as any monomial of odd length involving freely independent semicircular variables is 0, we obtain that for all $n \in \bN$ that
\[
\varphi_0(c_\ell (c_\ell^*c_\ell)^n) = 0 = \varphi_0(c_\ell^*(c_\ell c_\ell^*)^n).
\]
Similarly, for all $n \in \bN$ we have that
\[
\varphi_0((c_r^*c_r)^n) = \varphi_0((c_r c_r^*)^n) \qqand \varphi_0(c_r (c_r^*c_r)^n) = 0 = \varphi_0(c_r^*(c_r c_r^*)^n).
\]

To see the remaining moment conditions, first note that $\{c_\ell, c_\ell^*\}$ commutes with $\{c_r, c_r^*\}$.  Thus, as the $\chi$-ordering is not changed by commutation of left and right operators, it suffices to show that
\begin{align*}
\varphi_0((c_\ell^* c_\ell)^n (c_r c_r^*)^m) &= \varphi_0((c_\ell c_\ell^*)^n (c_r^* c_r)^m), \\
\varphi_0(c_\ell (c_\ell^* c_\ell)^n (c_r c_r^*)^m) &= \varphi_0(c_\ell^*(c_\ell c_\ell^*)^n (c_r^* c_r)^m), \\
\varphi_0((c_\ell^* c_\ell)^n c_r^* (c_r c_r^*)^m) &= \varphi_0(c_\ell^*(c_\ell c_\ell^*)^n c_r (c_r^* c_r)^m) \text{ and}\\
\varphi_0(c_\ell (c_\ell^* c_\ell)^n  c_r^* (c_r c_r^*)^m) &= \varphi_0(c_\ell^*(c_\ell c_\ell^*)^n c_r (c_r^* c_r)^m),
\end{align*}
for all $n,m \in \bN \cup \{0\}$.  However, as $\{c_\ell, c_\ell^*\}$ is classically independent from $\{c_r, c_r^*\}$ since the joint bi-free cumulants vanish, each of the 8 above moment expressions simplifies to the $\varphi_0$-moment of the $\{c_\ell, c_\ell^*\}$ term times the $\varphi_0$-moment of the $\{c_r, c_r^*\}$.  Thus, the desired moments are equal by the above knowledge of the $\varphi_0$-moments of the $\{c_\ell, c_\ell^*\}$ and the $\varphi_0$-moment of the $\{c_r, c_r^*\}$.
\end{example}

With the above definitions, notation and constructions out of the way, our main result is at hand.

\begin{theorem}\label{thm:minimizing-fisher-info}
Let $(\A, \varphi)$ be a C$^*$-non-commutative probability space and let $x,y \in \A$ be such that $x^*x$ and $xx^*$ have the same distribution with respect to $\varphi$ and $y^*y$ and $yy^*$ have the same distribution with respect to $\varphi$.  With $X$ and $Y$ as described above
\[
\min\left\{\left.\Phi^*(\{x_0, x_0^*\} \sqcup \{y_0, y_0^*\} : (\bC, \bC), \varphi )\, \right| (x_0, y_0) \in \Delta_{X, Y}   \right\} \geq 2 \Phi^*(X \sqcup Y)
\]
and equality holds and is achieved for any pair $(x_0, y_0)$ that is alternating adjoint flipping and bi-R-diagonal.
\end{theorem}

\begin{remark}
Note that Theorem \ref{thm:minimizing-fisher-info} is a generalization of \cite{NSS1999}*{Theorem 1.1} to the bi-free setting.  Prior to the acknowledgements of \cite{NSS1999} it is mentioned that the minimum in the free result can only be reached by an R-diagonal element via the result from \cite{V1999} that $\Phi^*(x_1, \ldots, x_n : B) = \Phi^*(x_1, \ldots, x_n) < \infty$ implies $\{x_1, \ldots, x_n\}$ is free from $B$.  As there is no such known analogous result in the bi-free case, we leave Theorem \ref{thm:minimizing-fisher-info} as stated.
\end{remark}

To begin the proof of Theorem \ref{thm:minimizing-fisher-info}, we note the following connecting the bi-free Fisher informations of $(\{x, x^*\}, \{y,y^*\})$ and $(X, Y)$ (and thereby demonstrating the necessity of considering bi-free Fisher information with respect to completely positive maps in this construction).  We note that the following is a generalization of \cite{NSS1999}*{Proposition 3.6} with a similar but more complicated proof due to the $\chi$-ordering and additional variables present.

\begin{proposition}\label{prop:fisher-from-collections-to-matrices}
Under the assumptions of Theorem \ref{thm:minimizing-fisher-info}, if $\eta : M_2(\bC) \to M_2(\bC)$ is defined by
\[
\eta\left(\begin{bmatrix}
a_{1,1} & a_{1,2} \\ a_{2,1} & a_{2,2}
\end{bmatrix}\right) = \begin{bmatrix}
a_{2,2} & 0 \\ 0 & a_{1,1}
\end{bmatrix},
\]
then
\[
\Phi^*(\{x, x^*\} \sqcup \{y, y^*\} : (\bC, \bC), \varphi) = 2 \Phi^*(X \sqcup Y : (M_2(\bC)_\ell, M_2(\bC)_r), \eta).
\]
\end{proposition}
\begin{proof}
First suppose the bi-free Fisher information $\Phi^*(\{x, x^*\} \sqcup \{y, y^*\})$ is finite.  Thus there exist
\[
\xi_1, \xi_2 \in \overline{\alg(x, x^*, y, y^*)}^{\left\|\,\cdot \,\right\|_\varphi}
\]
such that $\xi_1$ is the left bi-free conjugate variable for $x$ with respect to $\varphi$ in the presence of $(x, \{y,y^*\})$ and $\xi_2$ is the left bi-free conjugate variable for $x^*$ with respect to $\varphi$ in the presence of $(x^*, \{y,y^*\})$.  Let
\[
\Xi = \begin{bmatrix}
0 & \xi_2 \\
\xi_1 & 0
\end{bmatrix} \in M_2(L_2(A_2, \tau_2)).
\]
We claim that $\Xi = J_\ell\left(X : \left(M_2(\bC)_\ell, \alg\left(M_2(\bC)_r, Y\right)\right), \eta\right)$.  Since a similar result holds on the right and since
\[
\left\|\Xi\right\|_{\tau_2} = \frac{1}{2}\left(\left\|\xi_1\right\|_\varphi + \left\|\xi_2\right\|_\varphi\right),
\]
the result will follow in this case.

First we claim that $\Xi \in \overline{\alg\left(X, Y, M_2(\bC)_\ell,  M_2(\bC)_r\right)}^{\tau_2}$.  Indeed, it is not difficult to verify that
\[
z \otimes E_{i_1, j_1} \otimes E_{i_2, j_2} \in \alg\left(X, Y, M_2(\bC)_\ell,  M_2(\bC)_r\right)
\]
for all $z \in \{x,x^*, y,y^*\}$ and $i_1, i_2, j_1, j_2 \in \{1,2\}$.  Thus, since $\xi_1, \xi_2 \in \overline{\alg(x, x^*, y, y^*)}^{\left\|\,\cdot \,\right\|_\varphi}$, the claim follows.

To complete the claim that $\Xi$ is the appropriate left bi-free conjugate variable, we must show $\Xi$ satisfies the left bi-free conjugate variable relations; that is, for all $n \in \bN$, $b_0, b_1, \ldots, b_n \in M_2(\bC)$, $\chi \in \{\ell, r\}^n$ with $\chi(n) = \ell$, and $Z_1, \ldots, Z_{n-1} \in A_2$ and $C_1, \ldots C_{n-1} \in 1_\A \otimes M_2(\bC) \otimes M_2(\bC)^\op$ where
\[
Z_k = \begin{cases}
X   & \text{if } \chi(k) = \ell \\
Y & \text{if } \chi(k) = r
\end{cases} \qqand  C_k = \begin{cases}
L_{b_k} & \text{if } \chi(k) = \ell \\
R_{b_k} & \text{if } \chi(k) = r
\end{cases},
\]
we have that
\begin{align}
\tau_2\left(L_{b_0} R_{b_{n}} Z_1C_1  \cdots Z_{n-1}C_{n-1} \Xi\right) = \sum_{\substack{1 \leq k < n \\ \chi(k) = \ell}} \tau_2\left(L_{b_0} R_{b_{n}} \left( \prod_{p \in V_k^c\setminus \{k,n\}} Z_pC_p\right) L_{\eta\left(  E_2\left(C_k \prod_{p \in V_k} Z_pC_p\right)     \right)} \right), \label{eq:messy-conjugate-variable-for-min-free-fisher}
\end{align}
where $V_k = \{k < m < n \, \mid \, \chi(m) = \ell \}$.  By linearity, it suffices to consider $b_k = E_{i_k, j_k}$ for all $k$ where $i_k, j_k \in \{1,2\}$.   In that which follows, the proof is near identical to that of \cite{NSS1999}*{Proposition 3.6} taking into account the $\chi$-order.  For notational purposes, for $k \in \{1,2\}$ let $\overline{k} =3-k$.

Let $q = s_\chi^{-1}(n)$ (i.e. $\Xi$ appears $q^{\mathrm{th}}$ in the $\chi$-ordering).  We begin by computing the left-hand side of (\ref{eq:messy-conjugate-variable-for-min-free-fisher}).  Using Lemma \ref{lem:expanding-moments-of-matrices} (and recalling $\tau_2 = \tr_2 \circ E_2$), proceeding via $\chi$-order using commutation, we obtain that
\begin{itemize}
\item the only way the product produces a non-zero trace is if $i_0 = j_n$,
\item the term $L_{b_0} X L_{b_{s_\chi(1)}}$ can be made to appear in the product and is non-zero only if $j_0 =\overline{i_{s_\chi(1)}}$, 
\item the term $L_{b_{s_\chi(k-1)}} X L_{b_{s_\chi(k)}}$ can be made to occur for all $2 \leq k < q$ and is non-zero only if $j_{s_\chi(k-1)} = \overline{i_{s_\chi(k)}}$,
\item the term $L_{b_{s_\chi(q-1)}} R_{b_{s_\chi(q+1)}} \Xi$ can be made to occur and is non-zero only if $j_{s_\chi(q-1)} = \overline{i_{s_\chi(q+1)}}$, 
\item the term $R_{b_{s_\chi(k+1)}} Y R_{b_{s_\chi(k)}}$ can be made to occur for all $q < k < n$ and is non-zero only if $j_{s_\chi(k)} = \overline{i_{s_\chi(k+1)}}$ (recall the opposite multiplication), and 
\item the term $R_{b_{n}} Y R_{b_{s_\chi(n)}}$ can be made to occur and is non-zero only if $j_{s_\chi(n)} = \overline{i_{n}}$.
\end{itemize}
Note the discrepancy in notation around the $\Xi$ term due to the labelling of the left and right $B$-operators (i.e. $b_{s_\chi(q)} = b_n$ is in the wrong spot).   Thus with
\[
(X)_{1,2} = x, \quad (X)_{2,1} = x^*, \quad (Y)_{1,2} = y, \quad (Y)_{2,1} = y^*, \quad (\Xi)_{1,2} = \xi_2, \qand (\Xi)_{2,1} = \xi_1,
\]
and
\[
\overline{Z_k} = \begin{cases}
(X)_{j_0, \overline{j_0}} & \text{if } s_\chi(k) = 1 \\
(X)_{j_{s_\chi(k-1)}, \overline{j_{s_\chi(k-1)}}} & \text{if } 1 < s_\chi(k) < q \\
(\Xi)_{j_{s_\chi(q-1)}, \overline{j_{s_\chi(q-1)}}} & \text{if } s_\chi(k) = q\\
(Y)_{j_{s_\chi(k)}, \overline{j_{s_\chi(k)}}} & \text{if } q < s_\chi(k) < n \\
(Y)_{j_{s_\chi(n)}, \overline{j_{s_\chi(n)}}} & \text{if } s_\chi(k) = n \\
\end{cases}
\]
we see that the left-hand side of (\ref{eq:messy-conjugate-variable-for-min-free-fisher}) is
\begin{align}
\frac{1}{2} \delta_{j_n, i_0}& \delta_{j_0, \overline{i_{s_{\chi}(1)}}} \delta_{j_{{s_\chi}(1)}, \overline{i_{s_{\chi}(2)}}} \cdots \delta_{j_{{s_\chi}(q-2)}, \overline{i_{s_{\chi}(q-1)}}}\delta_{j_{{s_\chi}(1-2)}, \overline{i_{s_{\chi}(q+1)}}}\delta_{j_{{s_\chi}(q+1)}, \overline{i_{s_{\chi}(q+2)}}}\cdots \delta_{j_{{s_\chi}(n-1)}, \overline{i_{s_{\chi}(n)}}} \delta_{j_{{s_\chi}(n)}, \overline{i_{n}}} \nonumber \\
& \times \varphi\left(\overline{Z_1} \cdots \overline{Z_n}\right). \label{eq:LHS}
\end{align}
where $\delta_{j,i}$ is the Kronecker delta.  Moreover, using the conjugate variable relations for $\xi_1$ and $\xi_2$, we see that
\begin{align}
\varphi\left(\overline{Z_1} \cdots \overline{Z_n}\right)  = \sum_{\substack{1 \leq k < n \\ \chi(k) = \ell}}   \delta_{j_{s_\chi(k-1)}, \overline{j_{s_\chi(q-1)}}}\varphi\left(\prod_{p \in V_k^c \setminus \{k,n\}} \overline{Z_p} \right)\varphi\left(\prod_{p \in V_k} \overline{Z_p} \right) \label{eq:LHS-conj-variable-expand}
\end{align}
where the $\delta_{j_{s_\chi(k-1)}, \overline{j_{s_\chi(q-1)}}}$ should be $\delta_{j_{s_\chi(k-1)}, \overline{j_{s_\chi(q-1)}}}$ when $k = s_\chi^{-1}(1)$.  

To complete the proof that equation (\ref{eq:messy-conjugate-variable-for-min-free-fisher}) holds, we compute the right-hand side of equation  (\ref{eq:messy-conjugate-variable-for-min-free-fisher}) and show the $k^{\mathrm{th}}$ term in the sum equals the $k^{\mathrm{th}}$ term obtain in equation (\ref{eq:LHS}) using equation (\ref{eq:LHS-conj-variable-expand}).  Indeed, for a fixed $1 \leq k < n$ for which $\chi(k) = \ell$, we can compute 
\[
M_k = E_2\left(C_k \prod_{p \in V_k} Z_pC_p\right),
\] 
in a similar fashion to the above.  Thus, to obtain a non-zero value, the relations $j_{s_\chi(p-1)} = \overline{i_{s_\chi(p)}}$ for all $p \in V_k$ must hold.  Moreover, one immediately obtains when $M_k \neq 0$ that
\[
M_k = \varphi\left(\prod_{p \in V_k} \overline{Z_p} \right) T_k,
\]
for some $T_k \in M_2(\bC)$.

Next, notice $\eta\left(M_k\right)$ is equivalent to multiplying $M_k$ on the left by $U = E_{1,2} + E_{2,1}$ (for right conjugate variables, one would multiply on the right) and thus we consider $UM_k$ in place of $\eta\left(M_k\right)$.  At this point, notice by Lemma \ref{lem:expanding-moments-of-matrices} that $UT_k$ can be written as a product of $b_p$'s with $b_{s_\chi(q-1)}$ being the right-most term.  By commutation, $R_{b_{s_\chi(q+1)}}$ will act on the right of  $UT_k$ thereby multiplying by $b_{s_\chi(q+1)}$ on the right and forcing $j_{s_\chi(q-1)} = \overline{i_{s_\chi(q+1)}}$ for a non-zero value to be obtained.  One then proceeds as above to show that a non-zero value is obtained only if the above relations are satisfied and that the term that is produced agrees with the $k^{\mathrm{th}}$ term of (\ref{eq:LHS-conj-variable-expand}).  Hence the proof is complete in the case that $\Phi^*(\{x, x^*\} \sqcup \{y, y^*\}) < \infty$.

To prove the result in the case that $\Phi^*(\{x, x^*\} \sqcup \{y, y^*\}) = \infty$, it suffices to show that if $\Phi^*(X \sqcup Y : \eta) < \infty$ then $\Phi^*(\{x, x^*\} \sqcup \{y, y^*\} : ( \bC, \bC), \varphi)< \infty$.  Thus, suppose that $\Phi^*(X \sqcup Y : \eta) < \infty$.  Hence $\Xi = J_\ell\left(X : \left(M_2(\bC)_\ell, \alg\left(M_2(\bC)_r, Y\right)\right), \eta\right)$ exists and can be written as
\[
\Xi = \begin{bmatrix}
\xi_{1,1} & \xi_{1,2} \\
\xi_{2,1} & \xi_{2,2}
\end{bmatrix}.
\]
We claim that
\[
\xi_{2,1} = J_\ell(x : (x^*, \{y, y^*\}), \varphi) \qqand \xi_{1,2} = J_\ell(x^* : (x, \{y, y^*\}), \varphi).
\]
As a similar result will hold on the right, we will obtain that $\Phi^*(\{x, x^*\} \sqcup \{y, y^*\} : ( \bC, \bC), \varphi) < \infty$ as desired.

As $\Xi \in  \overline{\alg\left(X, Y, M_2(\bC)_\ell,  M_2(\bC)_r\right)}^{\tau_2}$, it is not difficult to see that $\xi_{i,j} \in \overline{\alg(x, x^*, y, y^*)}^{\left\|\,\cdot \,\right\|_\varphi}$.  To see that $\xi_{2,1}$ and $\xi_{1,2}$ satisfy the appropriate left bi-free conjugate variable relations, one need only use equation (\ref{eq:messy-conjugate-variable-for-min-free-fisher}), choose $b_k = E_{i_k, j_k}$ satisfying the above required  relations for a non-zero value and expand both sides of equation (\ref{eq:messy-conjugate-variable-for-min-free-fisher}) in an identical way to that above.  The resulting equations are exactly the left bi-free conjugate variable relations required.
\end{proof}

Using the results from this paper, there is some immediate knowledge about the bi-free Fisher information with respect to $\eta$ from Proposition \ref{prop:fisher-from-collections-to-matrices}.   We note that the following is a generalization of \cite{NSS1999}*{Proposition 3.7} with a similar but more complicated proof due to the $\chi$-ordering and additional variables present.

\begin{proposition}\label{prop:equivalent-fishers-for-minimization}
Under the assumptions and notation of Proposition \ref{prop:fisher-from-collections-to-matrices}, 
\[
\Phi^*(X \sqcup Y :  \eta) \geq \Phi^*(X \sqcup Y : \eta|_{D_2}) 
\]
and the inequality holds when $(\alg((D_2)_\ell,X), \alg((D_2)_r, Y))$ is bi-free from $(M_2(\bC)_\ell, M_2(\bC)_r)$ with amalgamation over $D_2$ with respect to $F_2$.  Moreover
\[
\Phi^*(X \sqcup Y : \eta|_{D_2}) \geq \Phi^*(X \sqcup Y)
\] 
and the inequality holds if $(x, y)$ is alternating adjoint flipping.
\end{proposition}
\begin{proof}
Since $\eta = \eta \circ F$, we have that
\[
\Phi^*(X \sqcup Y :  \eta) = \Phi^*(X \sqcup Y : ((M_2(\bC)_\ell, M_2(\bC)_r), \eta \circ F)
\]
by Remark \ref{rem:remarks-about-fisher-info} part (\ref{rem:remarks-about-fisher-info:part:smaller-amalgamation}).  Moreover
\[
\Phi^*(X \sqcup Y : ((M_2(\bC)_\ell, M_2(\bC)_r), \eta \circ F) \geq \Phi^*(X \sqcup Y : \eta|_{D_2})
\]
by Remark \ref{rem:remarks-about-fisher-info} part (\ref{rem:remarks-about-fisher-info:part:smaller-algebra}).  Furthermore, equality holds if $(\alg((D_2)_\ell,X), \alg((D_2)_r, Y))$ is bi-free from $(M_2(\bC)_\ell, M_2(\bC)_r)$ over $D_2$ with respect to $F_2$ by Remark \ref{rem:remarks-about-fisher-info} part (\ref{rem:remarks-about-fisher-info:part:adding-bi-free-algebra}).  

To see that $\Phi^*(X \sqcup Y : \eta|_{D_2}) \geq \Phi^*(X \sqcup Y)$  we assume that
\[
\Xi = J_\ell\left(X : \left((D_2)_\ell, \alg((D_2)_r,Y) \right), \eta|_{D_2}\right) \in \overline{\alg(X, Y, (D_2)_\ell, (D_2)_r) }^{\left\| \, \cdot \, \right\|_{\tau_2}}
\]
exists and show that $\Xi$ satisfies the left bi-free conjugate variable relations for $X$ in the presence of $Y$.  Thus if $P$ is the orthogonal projection of $L_2(A_2, \tau_2)$ onto  $\overline{\alg(X, Y)}^{\left\| \, \cdot \, \right\|_{\tau_2}}$ then $P(\Xi)$ will also satisfy the left bi-free conjugate variable relations for $X$ in the presence of $Y$.  As a similar result will hold on the right, the inequality $\Phi^*(X \sqcup Y : \eta|_{D_2}) \geq \Phi^*(X \sqcup Y)$ will be demonstrated.

By the defining property of $\Xi$, we know for all $n \in \bN$, $b_0, b_1, \ldots, b_n \in D_2$, $\chi \in \{\ell, r\}^n$ with $\chi(n) = \ell$, and $Z_1, \ldots, Z_{n-1} \in A_2$ and $C_1, \ldots C_{n-1} \in 1_\A \otimes M_2(\bC) \otimes M_2(\bC)^\op$ where
\[
Z_k = \begin{cases}
X   & \text{if } \chi(k) = \ell \\
Y & \text{if } \chi(k) = r
\end{cases} \qqand  C_k = \begin{cases}
L_{b_k} & \text{if } \chi(k) = \ell \\
R_{b_k} & \text{if } \chi(k) = r
\end{cases},
\]
that
\begin{align}
\tau_2\left(L_{b_0} R_{b_{n}} Z_1C_1  \cdots Z_{n-1}C_{n-1} \Xi\right) = \sum_{\substack{1 \leq k < n \\ \chi(k) = \ell}} \tau_2\left(L_{b_0} R_{b_{n}} \left( \prod_{p \in V_k^c\setminus \{k,n\}} Z_pC_p\right) L_{\eta\left( (F \circ  E_2)\left(C_k \prod_{p \in V_k} Z_pC_p\right)     \right)} \right), \label{eq:messy-conjugate-variable-for-diagonal-proof}
\end{align}
where $V_k = \{k < m < n \, \mid \, \chi(m) = \ell \}$.  We will use equation (\ref{eq:messy-conjugate-variable-for-diagonal-proof}) where $b_k = I_2$ for all $k$.  To begin, notice that \[
\eta\left( (F \circ  E_2)\left(C_k \prod_{p \in V_k} Z_pC_p\right) \right) = \tau_2\left(X^{|V_k|}\right)I_2,
\]
as odd moments of $X$ are zero and as $x^*x$ and $xx^*$ have the same distribution with respect to $\varphi$. Therefore
\[
\varphi((x^*x)^m) = \varphi((xx^*)^m) = \tau_2(X^{2m}),
\]
for all $m \in \bN$.  Hence, equation (\ref{eq:messy-conjugate-variable-for-diagonal-proof}) reduces to
\[
\tau_2\left(  Z_1   \cdots Z_{n-1}\Xi\right) = \sum_{\substack{1 \leq k < n \\ \chi(k) = \ell}} \tau_2\left((Z_1, \ldots, Z_{n-1})|_{V_k^c \setminus \{k, n\}} \right)\tau_2\left(X^{|V_k|}\right),
\]
which is exactly the desired formula.

To prove  $\Phi^*(X \sqcup Y : \eta|_{D_2}) \leq \Phi^*(X \sqcup Y)$ when $(x, y)$ is alternating adjoint flipping thereby completing the proof, we proceed in a similar (but more complicated) fashion.  Suppose 
\[
\Xi = J_\ell\left(X : \left(\bC, \alg(Y) \right)\right) \in  \overline{\alg(X, Y) }^{\left\| \, \cdot \, \right\|_{\tau_2}} \subseteq \overline{\alg(X, Y, (D_2)_\ell, (D_2)_r) }^{\left\| \, \cdot \, \right\|_{\tau_2}}
\]
exists.   We will demonstrate that $\Xi$ satisfies the left bi-free conjugate variable relations for $X$ with respect to $\eta$ in the presence of $((D_2)_\ell, \alg((D_2)_r, Y$)).  As an analogous result will hold on the right, this will complete the proof.

Write
\[
\Xi = \begin{bmatrix}
\xi_{1,1} & \xi_{1,2} \\
\xi_{2,1} & \xi_{2,2}
\end{bmatrix} \in M_2(L_2(\A, \varphi)) = L_2(A_2, \tau_2).
\]
First we will demonstrate that $\xi_{1,1} = \xi_{2,2} = 0$.  To begin, let
\begin{align*}
\H_e & = \overline{\mathrm{span}( Z_1 \cdots Z_{2n} \, \mid \, n \in \bN, Z_k \in \{X, Y\})  }^{\left\|\, \cdot \, \right\|_{\tau_2}} \text{ and} \\
\H_o &=\overline{\mathrm{span}( Z_1 \cdots Z_{2n-1} \, \mid \, n \in \bN, Z_k \in \{X, Y\})}^{\left\|\, \cdot \, \right\|_{\tau_2}}.
\end{align*}
By the defining property of $\Xi$, we know that for all $n \in \bN$, $\chi \in \{\ell, r\}^n$ with $\chi(n) = \ell$, and $Z_1, \ldots, Z_{n-1} \in A_2$ where
\[
Z_k = \begin{cases}
X   & \text{if } \chi(k) = \ell \\
Y & \text{if } \chi(k) = r
\end{cases} 
\]
that 
\begin{align}
\tau_2\left(  Z_1   \cdots Z_{n-1}\Xi\right) = \sum_{\substack{1 \leq k < n \\ \chi(k) = \ell}} \tau_2\left((Z_1, \ldots, Z_{n-1})|_{V_k^c \setminus \{k, n\}} \right)\tau_2\left(X^{|V_k|}\right), \label{eq:conj-variable-X-Y-scalar}
\end{align}
where $V_k = \{k < m < n \, \mid \, \chi(m) = \ell \}$.  Note as $\tau_2$ evaluates any odd product involving $X$ and $Y$ to 0 by Lemma \ref{lem:expanding-moments-of-matrices}, if $n-1$ is even, then $\tau_2\left(  Z_1   \cdots Z_{n-1}\Xi\right) = 0$.  Therefore, since $\Xi \in \H_e + \H_o$, we obtain that $\Xi \in \H_o$.

Note for $n \in \bN$, $\chi \in \{\ell, r\}^{2n-1}$, $Z_1, \ldots, Z_{2n-1} \in \{X, Y\}$ , and $z_1, \ldots, z_{2n-1} \in \{x,y\}$ where 
\[
Z_k = \begin{cases}
X   & \text{if } \chi(k) = \ell \\
Y & \text{if } \chi(k) = r
\end{cases} \qqand z_k = \begin{cases}
x   & \text{if } \chi(k) = \ell \\
y & \text{if } \chi(k) = r
\end{cases}
\]
that in $M_2(L_2(\A, \varphi))$ we have
\[
Z_1 \cdots Z_{2n-1} = \begin{bmatrix}
0 & z_1^{p_1} z_2^{p_2}  \cdots  z_{2n-1}^{p_{2n-1}} \\
z_1^{q_1} z_2^{q_2}  \cdots  z_{2n-1}^{q_{2n-1}} & 0
\end{bmatrix},
\]
where
\[
p_{s_\chi(k)} = \begin{cases}
1 & \text{if }k \text{ is odd}\\
* & \text{if }k \text{ is even}
\end{cases} \qqand q_{s_\chi(k)} = \begin{cases}
* & \text{if }k \text{ is odd}\\
1 & \text{if }k \text{ is even}
\end{cases}.
\]
Therefore, as $\H_o$ is the $\left\|\, \cdot \, \right\|_{\tau_2}$-limit of matrices of the above form and as $\Xi \in \H_o$, we obtain that $\xi_{1,1} = \xi_{2,2} = 0$ as desired.  

Let
\begin{align*}
H_1 & = \mathrm{span}\left( z_1 \cdots z_{2n-1} \, \left| \, \substack{ n \in \bN, z_k \in \{x, x^*, y, y^*\}\\ \text{the powers of the }z_k\text{'s alternate between 1 and * in the $\chi$-ordering} \\ \text{and the first and last element in the $\chi$-ordering have power 1 } }  \right.\right) \quad \text{and}\\  
H_* & = \mathrm{span}\left( z_1 \cdots z_{2n-1} \, \left| \, \substack{ n \in \bN, z_k \in \{x, x^*, y, y^*\}\\ \text{the powers of the }z_k\text{'s alternate between 1 and * in the $\chi$-ordering} \\ \text{and the first and last element in the $\chi$-ordering have power *} }  \right.\right).
\end{align*}
Using the above and the notation $\xi = \xi_{1,2}$ and $\xi^* = \xi_{2,1}$ (note we do not claim that there is an involution operation on $L_2(\A, \varphi)$ as we do not know $\varphi$ is tracial), we see that $\xi \in \overline{H_1}^{\left\|\, \cdot \, \right\|_\varphi}$, $\xi^* \in \overline{H_*}^{\left\|\, \cdot \, \right\|_\varphi}$, and if we have a $\left\|\, \cdot \, \right\|_{\varphi}$-limiting sequence using $\{x,x^*, y,y^*\}$ producing $\xi$ we can obtain an a $\left\|\, \cdot \, \right\|_{\varphi}$-limiting sequence using $\{x,x^*, y,y^*\}$ producing $\xi^*$ by exchanging $x \leftrightarrow x^*$ and $y \leftrightarrow y^*$.  This, in conjunction with the alternating adjoint flipping condition lets us show if $n \in \bN$, $\chi \in \{\ell, r\}^{2n}$, $Z_1, \ldots, Z_{2n-1} \in \{X, Y\}$ and $z_1, \ldots, z_{2n-1} \in \{x,y\}$ where 
\[
Z_k = \begin{cases}
X   & \text{if } \chi(k) = \ell \\
Y & \text{if } \chi(k) = r
\end{cases}  \qqand 
z_k = \begin{cases}
x   & \text{if } \chi(k) = \ell \\
y & \text{if } \chi(k) = r
\end{cases}
\]
that
\begin{align}
\varphi(z_1^{p_1} z_2^{p_2}  \cdots  z_{2n-1}^{p_{2n-1}} \xi^{p_{2n}})  = \varphi(z_1^{q_1} z_2^{q_2}  \cdots  z_{2n-1}^{q_{2n-1}} \xi^{q_{2n}}), \label{eq:xi-can-do-it}
\end{align}
where 
\[
p_{s_\chi(k)} = \begin{cases}
1 & \text{if }k \text{ is odd}\\
* & \text{if }k \text{ is even}
\end{cases} \qqand q_{s_\chi(k)} = \begin{cases}
* & \text{if }k \text{ is odd}\\
1 & \text{if }k \text{ is even}
\end{cases}.
\]
Indeed, consider $z_1^{p_1} \cdots  z_{2n-1}^{p_{2n-1}} \xi^{p_{2n}}$ with $p_{2n} = 1$ (the case $p_{2n}=*$ is analogous).  As  the terms preceding $\xi$ in the $\chi$-ordering both must have $*$'s on them and as $\xi$ is a $\left\| \, \cdot \, \right\|_\varphi$-limit of elements of $H_1$, we see that $z_1^{p_1} \cdots  z_{2n-1}^{p_{2n-1}} \xi^{p_{2n}}$ is a $\left\| \, \cdot \, \right\|_\varphi$-limit of a linear combination of monomials in $\{x,x^*, y,y^*\}$ that alternate between  $*$ and non-$*$-terms in the $\chi$-ordering.  As $x \leftrightarrow x^*$ and $y \leftrightarrow y^*$ produces the same $\varphi$-moment by the alternating adjoint flipping condition (as $z_1^{p_1} \cdots  z_{2n-1}^{p_{2n-1}}$ and every element of $H_1$ is of odd length) and produces a sequence that converges to $z_1^{q_1} z_2^{q_2}  \cdots  z_{2n-1}^{q_{2n-1}} \xi^{q_{2n}}$ with respect to $\left\| \, \cdot \, \right\|_\varphi$, the claim is complete.

Returning to showing $\Xi$ satisfies the left bi-free conjugate variable relations for $X$ with respect to $\eta$ in the presence of $((D_2)_\ell, \alg((D_2)_r, Y$)), it suffices to demonstrate that equation (\ref{eq:messy-conjugate-variable-for-diagonal-proof}) holds for this $\Xi$.  Furthermore, it suffices to verify that equation (\ref{eq:messy-conjugate-variable-for-diagonal-proof}) holds when $b_k = E_{i_k, i_k}$ for all $k$.  By the same computations as done in the proof of Proposition \ref{prop:fisher-from-collections-to-matrices} with $j_k = i_k$ for all $k$, we see with $q = s_\chi^{-1}(n)$ that 
\begin{align*}
&\tau_2\left(L_{b_0} R_{b_{n}} Z_1C_1  \cdots Z_{n-1}C_{n-1} \Xi\right)\\
&= \begin{cases}
\frac{1}{2}\varphi(z_1^{p_1} z_2^{p_2}  \cdots  z_{n-1}^{p_{n-1}} \xi^{p_{n}})  & \text{if } n \text{ is even and }\left(i_{0}, i_{s_\chi(1)}, \ldots, i_{s_\chi(q-1)}, i_{s_\chi(q+1)}, \ldots, i_{s_\chi(n)}\right) = (1,2,\ldots, 1,2)  \\
\frac{1}{2}\varphi(z_1^{q_1} z_2^{q_2}  \cdots  z_{n-1}^{q_{n-1}} \xi^{q_{n}})  & \text{if } n \text{ is even and }\left(i_{0}, i_{s_\chi(1)}, \ldots, i_{s_\chi(q-1)}, i_{s_\chi(q+1)}, \ldots, i_{s_\chi(n)}\right) = (2,1,\ldots, 2,1) \\
0 & \text{otherwise}
\end{cases}
\end{align*}
and
\begin{align*}
&\sum_{\substack{1 \leq k < n \\ \chi(k) = \ell}} \tau_2\left(L_{b_0} R_{b_{n}} \left( \prod_{p \in V_k^c\setminus \{k,n\}} Z_pC_p\right) L_{\eta\left( (F \circ  E_2)\left(C_k \prod_{p \in V_k} Z_pC_p\right)     \right)} \right)\\
&= \begin{cases}
 \frac{1}{2}\sum_{\substack{1 \leq k < n \\ \chi(k) = \ell \\ |V_k| \text{ even}}} \varphi\left( (z_1^{p_1},   \ldots,   z_{n-1}^{p_{n-1}} )|_{V_k^c \setminus \{k,n\}}\right)  \varphi\left( (z_1^{p_1},    \ldots,   z_{n-1}^{p_{n-1}} )|_{V_k}\right)  \\ \qquad \qquad \text{if } n \text{ is even and } \left(i_{0}, i_{s_\chi(1)}, \ldots, i_{s_\chi(q-1)}, i_{s_\chi(q+1)}, \ldots, i_{s_\chi(n)}\right) = (1,2,\ldots, 1,2)  \\
\frac{1}{2}\sum_{\substack{1 \leq k < n \\ \chi(k) = \ell \\ |V_k| \text{ even}}} \varphi\left( (z_1^{q_1},   \ldots,   z_{n-1}^{q_{n-1}} )|_{V_k^c \setminus \{k,n\}}\right)  \varphi\left( (z_1^{q_1},    \ldots,   z_{n-1}^{p_{qn-1}} )|_{V_k}\right)  \\ \qquad \qquad \text{if } n \text{ is even and } \left(i_{0}, i_{s_\chi(1)}, \ldots, i_{s_\chi(q-1)}, i_{s_\chi(q+1)}, \ldots, i_{s_\chi(n)}\right) = (2,1,\ldots, 2,1)  \\
0, \qquad \quad \text{otherwise}
\end{cases}
\end{align*}
where $V_k = \{k < m < n \, \mid \, \chi(m) = \ell \}$ (note only the terms where $|V_k|$ is even survive from the $\eta \circ F \circ E_2$ expression due to the form of $X$) and $z_k$, $p_k$, and $q_k$ are defined as usual in this proof.  Hence, it suffices to show when $n$ is even that
\begin{align}
\varphi(z_1^{p_1} z_2^{p_2}  \cdots  z_{n-1}^{p_{n-1}} \xi^{p_{n}}) &= \sum_{\substack{1 \leq k < n \\ \chi(k) = \ell \\ |V_k| \text{ even}}} \varphi\left( (z_1^{p_1},   \ldots,   z_{n-1}^{p_{n-1}} )|_{V_k^c \setminus \{k,n\}}\right)  \varphi\left( (z_1^{p_1},    \ldots,   z_{n-1}^{p_{n-1}} )|_{V_k}\right), \label{eq:annoying-equation-1}\\ 
\varphi(z_1^{q_1} z_2^{q_2}  \cdots  z_{n-1}^{q_{n-1}} \xi^{q_{n}}) &= \sum_{\substack{1 \leq k < n \\ \chi(k) = \ell \\ |V_k| \text{ even}}} \varphi\left( (z_1^{q_1},   \ldots,   z_{n-1}^{q_{n-1}} )|_{V_k^c \setminus \{k,n\}}\right)  \varphi\left( (z_1^{q_1},    \ldots,   z_{n-1}^{p_{qn-1}} )|_{V_k}\right). \label{eq:annoying-equation-2}
\end{align}

Note equations (\ref{eq:annoying-equation-1}) and (\ref{eq:annoying-equation-2}) are the same equation by the alternating adjoint flipping condition and equation (\ref{eq:xi-can-do-it}).   Moreover, due to the defining property of $\Xi$, we know with $n$ even that
\begin{align*}
\tau_2(Z_1 \cdots Z_{n-1} \Xi) &= \sum_{\substack{1 \leq k < n \\ \chi(k) = \ell}} \tau_2\left( \left( \prod_{p \in V_k^c\setminus \{k,n\}} Z_p \right) L_{\eta\left( (F \circ  E_2)\left(\prod_{p \in V_k} Z_p\right)     \right)} \right)\\
&= \sum_{\substack{1 \leq k < n \\ \chi(k) = \ell}} \tau_2\left( \left( \prod_{p \in V_k^c\setminus \{k,n\}} Z_p \right)\right)  \tau_2\left(X^{|V_k|}\right).
\end{align*}
Due to the form of $X$ and the alternating adjoint flipping condition, we immediately see that
\[
\tau_2\left(X^{|V_k|}\right) = \begin{cases}
0 & \text{if } |V_k| \text{ is odd} \\
\varphi\left( (z_1^{p_1},    \ldots,   z_{n-1}^{p_{n-1}} )|_{V_k}\right) & \text{if } |V_k| \text{ is even}
\end{cases}
\]
and, for $n$ even and $k$ such that $|V_k|$ is even, we have
\begin{align*}
\tau_2\left( \left( \prod_{p \in V_k^c\setminus \{k,n\}} Z_p \right)\right)   &= \frac{1}{2}\left(\varphi\left( (z_1^{p_1},   \ldots,   z_{n-1}^{p_{n-1}} )|_{V_k^c \setminus \{k,n\}}\right) + \varphi\left( (z_1^{q_1},   \ldots,   z_{n-1}^{q_{n-1}} )|_{V_k^c \setminus \{k,n\}}\right)     \right) \\
&= \varphi\left( (z_1^{p_1},   \ldots,   z_{n-1}^{p_{n-1}} )|_{V_k^c \setminus \{k,n\}}\right), 
\end{align*}
thereby completing the proof.
\end{proof}

\begin{proof}[Proof of Theorem \ref{thm:minimizing-fisher-info}]
The proof follows immeditaley by combining Propositions \ref{prop:bi-R-diagonal-iff-bi-R-cyclic-iff-free-with-amal}, \ref{prop:fisher-from-collections-to-matrices}, and \ref{prop:equivalent-fishers-for-minimization}.
\end{proof}

\section{Maximizing Bi-Free Entropy}
\label{sec:Max-Entropy}

In this section, we will prove Theorem \ref{thm:maximizing-bi-free-entropy} obtaining an upper bound for the bi-free entropy of a pair of operators and their adjoints based on the entropy of a pair of matrices and demonstrate when equality is obtained.  In particular, this generalizes an essential result from \cite{NSS1999}*{Section 5}.

To begin, we must establish a formula for the bi-free entropy of non-self-adjoint operators.

\begin{definition}\label{defn:entropy-for-non-self-adjoints}
Let $(\A, \varphi)$ be a C$^*$-non-commutative probability space and let $\{X_i, X_i^*\}^n_{i=1} \cup \{X'_i\}^{n'}_{i=1} \cup \{Y_j, Y_j^*\}^m_{j=1} \cup \{Y'_j\}^{m'}_{j=1} \subseteq \A$ where $X'_i$ and $Y'_j$ are self-adjoint for all $i$ and $j$.  The \emph{bi-free entropy of $(\{\vX, \vX^*, \vX'\}, \{\vY, \vY^*, \vY'\})$} is defined to be
\begin{align*}
\chi^* (\vX, \vX^*, \vX' \sqcup \vY, \vY^*, \vY') = \frac{2n + 2m + n' + m'}{2} \ln(2\pi e) + \frac{1}{2} \int^\infty_0 \left(\frac{2n + 2m + n' + m'}{1+t} - g(t) \right) \, dt,
\end{align*}
where 
\[
g(t) = \Phi^*\left( \vX + \sqrt{t} \vC_\ell, \vX^* + \sqrt{t} \vC^*_\ell, \vX' + \sqrt{t} \vS \sqcup \vY + \sqrt{t} \vC_r, \vY^* + \sqrt{t} \vC^*_r, \vY' + \sqrt{t} \vD\right),
\]
where $\vS$ and $\vD$ consist of semicircular variables of mean 0, variance 1,  covariance 0 and $\vC_\ell$ and $\vC_r$ consist of circular variables of mean 0, variance 1 and covariance 0  such that
\[
(\{\vX, \vX^*, \vX'\}, \{\vY, \vY^*, \vY'\}) \cup  \{ (S_i, 1)\}^{n'}_{i=1} \cup \{(1, D_j)\}^{m'}_{j=1} \cup \{(\{C_{\ell,i}, C_{\ell,i}^*\}, 1)   \}^n_{i=1} \cup \{(1, \{C_{r,j}, C_{r,j}^*\})\}^m_{j=1}
\]
are bi-free.
\end{definition}
\begin{remark}\label{rem:entropy-for-non-self-adjoints}
Given any C$^*$-non-commutative probability space $(\A, \varphi)$, it is always possible to find a larger C$^*$-non-commutative probability space that contains the necessary bi-free elements from Definition \ref{defn:entropy-for-non-self-adjoints}.  Indeed, one need only consider the scalar reduced free product of the appropriate spaces and use Definition \ref{bifreedef} to obtain bi-freeness.  The fact that the state is positive follows as it will be a vector state.  

In Definition \ref{defn:entropy-for-non-self-adjoints}, the bi-free Fisher information for tuples without self-adjoint elements is that in this paper with $\eta = \varphi$ and $B = \bC$.  

In the simplest case, one may ask why we do not simply define
\[
\chi^* (\{x,x^*\} \sqcup \{y,y^*\}) = \chi^*\left(\left\{\Re(x), \Im(x)\right\} \sqcup \left\{\Re(y), \Im(y)\right\}\right)
\]
to trivially reduce to the self-adjoint case in a similar fashion to Remark \ref{rem:remarks-about-fisher-info} part (\ref{rem:remarks-about-fisher-info:part:self-adjoint}) and why the integrand in Definition \ref{defn:entropy-for-non-self-adjoints} is well-defined.  Both of these questions are answered via Remark \ref{rem:remarks-about-fisher-info} part (\ref{rem:remarks-about-fisher-info:part:self-adjoint}) as
\begin{align*}
& \Phi^*\left( \left\{x + \sqrt{t} c_\ell, x^* + \sqrt{t} c_\ell^*\right\} \sqcup \left\{y + \sqrt{t} c_r, y^* + \sqrt{t} c_r^*\right\}\right)  \\
 &=  \frac{1}{2}\Phi^*\left( \left\{ \Re(x) + \sqrt{t} \Re(c_\ell), \Im(x) + \sqrt{t} \Im(c_\ell)\right\} \sqcup \left\{ \Re(y) + \sqrt{t} \Re(c_r), \Im(y) + \sqrt{t} \Im(c_r)\right\}\right) \\
 &=  \frac{1}{2}\Phi^*\left( \left\{ \Re(x) + \frac{\sqrt{t}}{\sqrt{2}} s_1, \Im(x) + \frac{\sqrt{t}}{\sqrt{2}} s_2\right\} \sqcup \left\{ \Re(y) + \frac{\sqrt{t}}{\sqrt{2}}d_1, \Im(y) + \frac{\sqrt{t}}{\sqrt{2}}d_2\right\}\right) \\
&=\Phi^*\left( \left\{ \sqrt{2}\Re(x) +  \sqrt{t} s_1, \sqrt{2}\Im(x) + \sqrt{t} s_2\right\} \sqcup \left\{ \sqrt{2}\Re(y) + \sqrt{t}d_1, \sqrt{2}\Im(y) + \sqrt{t} d_2\right\}\right)
\end{align*}
where $s_1, s_2, d_1$, and $d_2$ are as in Example \ref{exam:circular-pair}.  Hence the integrand in Definition \ref{defn:entropy-for-non-self-adjoints} is well-defined with
\begin{align*}
\chi^* (\{x,x^*\} \sqcup \{y,y^*\}) &= \chi^*\left(\left\{\sqrt{2}\Re(x), \sqrt{2}\Im(x)\right\} \sqcup \left\{\sqrt{2}\Re(y), \sqrt{2}\Im(y)\right\}\right) \\
&= \chi^*\left(\left\{\Re(x), \Im(x)\right\} \sqcup \left\{\Re(y), \Im(y)\right\}\right) + 4 \ln(\sqrt{2}).
\end{align*}
\end{remark}

We normalize Definition \ref{defn:entropy-for-non-self-adjoints} so that the following holds and generalizes \cite{NSS1999}*{Theorem 1.4} in the case $d=1$.

\begin{theorem}\label{thm:maximizing-bi-free-entropy}
Let $(\A, \varphi)$ be a C$^*$-non-commutative probability space and let $x,y \in \A$ be such that $x^*x$ and $xx^*$ have the same distribution with respect to $\varphi$ and $y^*y$ and $yy^*$ have the same distribution with respect to $\varphi$.  With $X$ and $Y$ as in Section \ref{sec:minimize-Fisher}, 
\[
\chi^*(\{x,x^*\} \sqcup \{y,y^*\}) \leq 2 \chi^*(X \sqcup Y)
\]
and equality holds whenever the pair $(x, y)$ is bi-R-diagonal and alternating adjoint flipping.
\end{theorem}
To prove Theorem \ref{thm:maximizing-bi-free-entropy}, we need two technical lemmata.  For the first, note the following does not immediately follow from Remark \ref{rem:inflating-preservers-bi-freeness} as being bi-free over $M_2(\bC)$ with respect to $E_2$ does not imply being bi-free with respect to $\tau_2$.

\begin{lemma}\label{lem:inflating-circular-yields-bi-free-semis}
Let $(\A, \varphi)$ be a C$^*$-non-commutative probability space, let $x,y \in \A$ be such that $x^*x$ and $xx^*$ have the same distribution with respect to $\varphi$ and $y^*y$ and $yy^*$ have the same distribution with respect to $\varphi$, and let $(c_\ell, c_r)$ be a bi-free circular pair in $\A$ with mean 0, variance 1 and covariance 0 such that
\[
(\{x,x^*\}, \{y,y^*\}) \cup \{(\{c_\ell, c^*_\ell\}, 1)\}\cup \{(1,\{c_r, c_r^*\})\}
\]
are bi-free with respect to $\varphi$.  Using the notation of Section \ref{sec:Max-Entropy}, if
\[
S_\ell = c_\ell \otimes E_{1,2} \otimes I_2 + c_\ell^* \otimes E_{2,1} \otimes I_2 \in A_2 \qand S_r = c_r\otimes I_2 \otimes E_{1,2} + c_r^* \otimes I_2 \otimes E_{2,1} \in A_2,
\]
then $S_\ell$ and $S_r$ have semicircular distributions with respect to $\tau_2$ of mean 0 and variance 1 and
\[
\{(X, Y)\} \cup \{(S_\ell, 1_{A_2})\} \cup \{(1_{A_2}, S_r)\}
\]
are bi-free with respect to $\tau_2$.
\end{lemma}
\begin{proof}
As  $\{(\{c_\ell, c^*_\ell\}, 1)\}\cup \{(1,\{c_r, c_r^*\})\}$ are bi-free with respect to $\varphi$ by Example \ref{exam:circular-pair}, Remark \ref{rem:inflating-preservers-bi-freeness} implies that $(S_\ell, 1)$ and $(1, S_r)$ are bi-free with respect to $E_2$.  Moreover, as $c_\ell$ and $c_r$ commute, $s_\ell$ and $s_r$ commute.  Hence we see for all $n,m \in \bN$ that
\begin{align*}
\tau_2(s_\ell^n s_r^m) = \tr_2(E_2(s_\ell^n s_r^m)) = \tr_2(E_2(s_\ell^n)E_2(s_r^m))= \begin{cases}
0 & \text{if }n \text{ or }m \text{ is odd}\\
\varphi\left((c_\ell^*c_\ell)^{\frac{n}{2}}\right)\varphi\left((c_r^*c_r)^{\frac{m}{2}}\right) & \text{if $n$ and $m$ are even}
\end{cases},
\end{align*}
by Example \ref{exam:circular-pair} and the alternating adjoint flipping condition.  Therefore, as $c_\ell^*c_\ell$ and $c_r^*c_r$ are known to have the same distributions as the square of a semicircular element of mean 0 and variance 1 (see \cite{VDN1992}*{Section 5.1}), we obtain that $(s_\ell, s_r)$ is the bi-free central limit distribution with mean 0, variance 1 and covariance 0 with respect to $\tau_2$.  Hence $\{(S_\ell, 1_{A_2})\} \cup \{(1_{A_2}, S_r)\}$ are bi-free with respect to $\tau_2$.

To complete the proof, it suffices to show that $\{(X,Y)\} \cup \{(S_\ell, S_r)\}$ are bi-free with respect to $\tau_2$.  Therefore, by \cite{CNS2015-2}, it suffices to show for all $n \in \bN$, $\chi \in \{\ell, r\}^n$, non-constant $\gamma  \in \{1,2\}^n$, and $Z_k \in \{X, Y, S_\ell, S_r\}$ where
\[
Z_k = \begin{cases}
X & \text{if }\chi(k) = \ell \text{ and } \gamma(k) = 1 \\
Y & \text{if }\chi(k) = r \text{ and } \gamma(k) = 1 \\
S_\ell & \text{if }\chi(k) = \ell \text{ and } \gamma(k) = 2 \\
S_r & \text{if }\chi(k) = r \text{ and } \gamma(k) = 2
\end{cases}
\]
that
\begin{align}
\tau_2(Z_1 \cdots Z_n) = \sum_{\substack{ \pi \in \BNC(\chi) \\ \pi \leq \gamma}} \kappa^{\tau_2}_\pi(Z_1, \ldots, Z_n) \label{eq:cumulant-eq-to-show-bi-free}
\end{align}
where $\gamma$ is representing the partition $\{\{k \, \mid \, \gamma(k) =1\}, \{k \mid \, \gamma(k) = 2\}\}$.  Note if $z_1, \ldots, z_n \in \{x,y,c_\ell, c_r\}$ are such that
\[
z_k = \begin{cases}
x & \text{if } Z_k = X \\
y & \text{if } Z_k = Y \\
c_\ell & \text{if } Z_k = S_\ell \\
c_r & \text{if } Z_k = S_r 
\end{cases}
\]
then by Lemma \ref{lem:expanding-moments-of-matrices} and the fact that $(\{x,x^*\}, \{y,y^*\}) \cup \{(c_\ell, 1)\}\cup \{(1,c_r)\}$ are bi-free with respect to $\varphi$, we have that
\begin{align*}
\tau_2(Z_1 \cdots Z_n) &= \begin{cases}
0 & \text{if } n \text{ is odd} \\
\frac{1}{2}\left(\varphi\left(z_1^{p_1} \cdots z_n^{p_n}\right) + \varphi\left(z_1^{q_1} \cdots z_n^{q_n}\right)   \right) & \text{if }n\text{ is even}
\end{cases} \\
&= \begin{cases}
0 & \text{if } n \text{ is odd} \\
\frac{1}{2} \sum_{\substack{ \pi \in \BNC(\chi) \\ \pi \leq \gamma}} \kappa^{\varphi}_\pi\left(z_1^{p_1}, \ldots,  z_n^{p_n}\right) + \kappa^{\varphi}_\pi\left(z_1^{q_1}, \ldots, z_n^{q_n}\right)  & \text{if }n\text{ is even}
\end{cases},
\end{align*}
where
\[
p_{s_\chi(k)} = \begin{cases}
1 & \text{if }k \text{ is odd}\\
* & \text{if }k \text{ is even}
\end{cases} \qqand q_{s_\chi(k)} = \begin{cases}
* & \text{if }k \text{ is odd}\\
1 & \text{if }k \text{ is even}
\end{cases}.
\]
To show this agrees with the right-hand side of equation (\ref{eq:cumulant-eq-to-show-bi-free}), we divide the discussion into several cases.  To this end, let
\[
I_{X,Y} = \{k \, \mid \, \gamma(k) = 1\} \qqand I_{S} = \{k \, \mid \, \gamma(k) = 2\}.
\]

First suppose $n$ is odd.  If $|I_S|$ is odd, then the right-hand side of equation (\ref{eq:cumulant-eq-to-show-bi-free}) is zero as there must be a cumulant involving an odd number of $S_\ell$ and $S_r$ and $\{(S_\ell, S_r)\}$ is a bi-free central limit distribution with 0 mean.    Otherwise, $|I_{X,Y}|$ is odd.  In this case, we may rearrange the sum on the right-hand side of equation (\ref{eq:cumulant-eq-to-show-bi-free}) to add over all $\pi \in \BNC(\chi)$ with $\pi \leq \gamma$ that form the same partition when restricted to $I_S$.   Since summing over such partitions yields a product of moment terms in the $X$'s and $Y$'s where the sum of the lengths of the moments is $|I_{X,Y}|$ and since all odd moment terms involving only $X$'s and $Y$'s is zero by Lemma \ref{lem:expanding-moments-of-matrices}, this portion of the sum yields zero.  Hence, equation (\ref{eq:cumulant-eq-to-show-bi-free}) holds when $n$ is odd.

In the case $n$ is even, note if $|I_S|$ is odd, then the right-hand side of equation (\ref{eq:cumulant-eq-to-show-bi-free}) is still zero.  However, 
\[
\frac{1}{2} \sum_{\substack{ \pi \in \BNC(\chi) \\ \pi \leq \gamma}} \kappa^{\varphi}_\pi\left(z_1^{p_1}, \ldots,  z_n^{p_n}\right) + \kappa^{\varphi}_\pi\left(z_1^{q_1}, \ldots, z_n^{q_n}\right) = 0
\]
as there must be a cumulant involving an odd number of $(\{c_\ell, c_\ell^*\}, \{c_r,c_r^*\})$ and $(c_\ell, c_r)$ is a bi-free circular pair.  Thus, we may assume that $n$, $|I_S|$, and $|I_{X,Y}|$ are even.

Under these assumptions, we claim that
\begin{align*}
\frac{1}{2} \sum_{\substack{ \pi \in \BNC(\chi) \\ \pi \leq \gamma}} \kappa^{\varphi}_\pi\left(z_1^{p_1}, \ldots,  z_n^{p_n}\right) + \kappa^{\varphi}_\pi\left(z_1^{q_1}, \ldots, z_n^{q_n}\right) = \sum_{\substack{ \pi \in \BNC(\chi) \\ \pi \leq \gamma}} \kappa^{\tau_2}_\pi(Z_1, \ldots, Z_n). 
\end{align*}
To see this, again we need only consider $\pi \in \BNC(\chi)$ that form pair partitions when restricted to $I_S$ and no block of $\pi$ contains both an element of $\{k \, \mid \, \chi(k) = \ell\}$ and of $\{k \, \mid \, \chi(k) = r \}$, since $\{(S_\ell, 1)\}\cup \{(1,S_r)\}$ are bi-free with respect to $\tau_2$ and $\{(\{c_\ell, c^*_\ell\}, 1)\}\cup \{(1,\{c_r, c_r^*\})\}$ are bi-free with respect to $\varphi$.  For such a partition $\pi$, if we let $\widehat{\pi}$ be the largest partition on $I_{X,Y}$ such that $\widehat{\pi} \cup \pi|_{I_S}$ is an element of $\BNC(\chi)$, then by adding over all $\sigma \in \BNC(\chi)$ with $\sigma \leq \gamma$ and $\sigma|_{I_S} = \pi|_{I_S}$, it suffices to show that
\begin{align}
\frac{1}{2} &\left( \varphi_{\widehat{\pi}}\left((z_1^{p_1}, \ldots,  z_n^{p_n})|_{I_{X,Y}}\right)\kappa^{\varphi}_{\pi|_{I_S}}\left(\left(z_1^{p_1}, \ldots,  z_n^{p_n}\right)|_{I_S} \right)   +  \varphi_{\widehat{\pi}}\left((z_1^{q_1}, \ldots,  z_n^{q_n})|_{I_{X,Y}}\right)\kappa^{\varphi}_{\pi|_{I_S}}\left(\left(z_1^{q_1}, \ldots,  z_n^{q_n}\right)|_{I_S} \right)    \right) \nonumber \\
&= (\tau_2)_{\widehat{\pi}}\left( \left(Z_1, \ldots, Z_n\right)|_{I_{X,Y}}\right). \label{eq:final-moment-cumulant-eq-to-show}
\end{align}
Note that
\[
\kappa^{\varphi}_{\pi|_{I_S}}\left(\left(z_1^{p_1}, \ldots,  z_n^{p_n}\right)|_{I_S} \right)  = 0 \quad \text{ or } \quad \kappa^{\varphi}_{\pi|_{I_S}}\left(\left(z_1^{q_1}, \ldots,  z_n^{q_n}\right)|_{I_S} \right)  = 0
\]
if and only if $\pi$ has a block with two $*$-terms or two non-$*$-terms, as $\pi|_{I_S}$ is a pair partition and $(c_\ell, c_r)$ is a bi-circular pair.  In this case we would have that $\widehat{\pi}$ has a block of  odd length and thus the right-hand side of equation (\ref{eq:final-moment-cumulant-eq-to-show}) is also zero, as any odd $\tau_2$-moment involving $X$ and $Y$ is zero.  Otherwise, both $\varphi$-cumulants are 1 and this forces every block of $\widehat{\pi}$ to be of even length and alternate between $1$ and $*$ in the $\chi$-ordering.  Since
\begin{align*}
\varphi((x^*x)^m) = \varphi((xx^*)^m) &= \tau_2(X^{2m}) \qand \varphi((y^*y)^m) = \varphi((yy^*)^m) = \tau_2(Y^{2m})
\end{align*}
and since (by the assumption that $\pi$ does not contain a block containing elements of $\{k \, \mid \, \chi(k) = \ell\}$ and of $\{k \, \mid \, \chi(k) = r \}$) there is a single block of $\widehat{\pi}$ containing elements of $\{k \, \mid \, \chi(k) = \ell\}$ and  $\{k \, \mid \, \chi(k) = r \}$, adding the two $\varphi$-terms together produces exactly the $\tau_2$ term in equation (\ref{eq:final-moment-cumulant-eq-to-show}). 
\end{proof}

\begin{lemma}\label{lem:perturbing-preserves-AAF}
Let $(\A, \varphi)$ be a C$^*$-non-commutative probability space, let $x,y \in \A$, and let $(c_\ell, c_r)$ be a bi-free circular pair in $\A$ with mean 0, variance 1 and covariance 0 such that
\[
(\{x,x^*\}, \{y,y^*\}) \cup \{(\{c_\ell, c^*_\ell\}, 1)\}\cup \{(1,\{c_r, c_r^*\})\}
\]
are bi-free with respect to $\varphi$.  Then:
\begin{enumerate}[(i)]
\item If $(x, y)$ is bi-R-diagonal, then $\left(x + \sqrt{t} c_\ell, y + \sqrt{t} c_r \right)$ is bi-R-diagonal for all $t \in (0, \infty)$.
\item If $(x, y)$ is alternating adjoint flipping, then $\left(x + \sqrt{t} c_\ell, y + \sqrt{t} c_r \right)$ is alternating adjoint flipping for all $t \in (0, \infty)$.
\item If $x^*x$ and $xx^*$ (respectively $y^*y$ and $yy^*$) have the same distribution with respect to $\varphi$, then $(x+\sqrt{t} c_\ell)^*(x+\sqrt{t} c_\ell)$ and $(x+\sqrt{t} c_\ell)(x+\sqrt{t} c_\ell)^*$ (respectively $(y+\sqrt{t} c_r)^*(y+\sqrt{t} c_r)$ and $(y+\sqrt{t} c_r)(y+\sqrt{t} c_r)^*)$ have the same distribution with respect to $\varphi$.
\end{enumerate}
\end{lemma}
\begin{proof}
As $(c_\ell, c_r)$ is bi-R-diagonal by Example \ref{exam:circular-pair} and as sums and scalar multiples of bi-R-diagonal pairs are bi-R-diagonal by \cite{K2019}*{Proposition 3.1}, (i) follows.

To see that (ii) holds, first we claim for all $n \in \bN$, $\chi \in \{\ell, r\}^{2n}$, and $z_1, \ldots, z_n \in \{x, y, c_\ell, c_r\}$ such that
\[
z_k \in \begin{cases}
\{x, c_\ell\} & \text{if } \chi(k) = \ell \\
\{y, c_r\} & \text{if } \chi(k) = r
\end{cases},
\]
we have that
\[
\varphi(z_1^{p_1} \cdots z_{2n}^{p_{2n}}) = \varphi(z_1^{q_1} \cdots z_{2n}^{q_{2n}}),
\]
where
\[
p_{s_\chi(k)} = \begin{cases}
1 & \text{if }k \text{ is odd}\\
* & \text{if }k \text{ is even}
\end{cases} \qqand q_{s_\chi(k)} = \begin{cases}
* & \text{if }k \text{ is odd}\\
1 & \text{if }k \text{ is even}
\end{cases}.
\]
Recall that
\[
\varphi(z_1^{p_1} \cdots z_{2n}^{p_{2n}}) = \sum_{\pi \in \BNC(\chi)} \kappa_\pi(z_1^{p_1},  \ldots,  z_{2n}^{p_{2n}})
\]
and the bi-free cumulant is zero if any block of $\pi$ contains both an element of $\{x, x^*, y, y^*\}$ and an element of $\{c_\ell, c^*_\ell, c_r, c_r^*\}$
As the only cumulants involving $c_\ell, c^*_\ell, c_r, c^*_r$ with non-zero values are
\[
\kappa_{1_{(\ell, \ell)}}(c_\ell, c^*_\ell) = 1 = \kappa_{1_{(\ell, \ell)}}(c^*_\ell, c_\ell) \qqand \kappa_{1_{(r, r)}}(c_r, c^*_r) = 1 = \kappa_{1_{(r, r)}}(c^*_r, c_r),
\]
for any fixed $\pi \in \BNC(\chi)$ for which the blocks containing $\{c_\ell, c^*_\ell, c_r, c_r^*\}$ do not cause the bi-free cumulant to be zero, we may add over all elements of $\BNC(\chi)$ with the same blocks as $\pi$ for those indices corresponding to elements of $\{c_\ell, c^*_\ell, c_r, c_r^*\}$ to obtain a product of moments involving $\{x, x^*, y, y^*\}$, each of which is of even length and alternates between $1$ and $*$ in the $\chi$-ordering.  We may then use the alternating adjoint flipping condition on $(x, y)$ to exchange the powers and reverse this cumulant reduction process to obtain $\varphi(z_1^{q_1} \cdots z_{2n}^{q_{2n}})$, thereby completing the claim.  Thus (ii) then follows by linearity.

To see that (iii) holds, we desire to show that
\[
\varphi\left( \left((x+\sqrt{t} c_\ell)^*(x+\sqrt{t} c_\ell)\right)^n   \right) = \varphi\left( \left((x+\sqrt{t} c_\ell)(x+\sqrt{t} c_\ell)^*\right)^n   \right)
\]
for all $n \in \bN$.  To see how the left-hand side can be changed into the right-hand side, arguments similar to the proof of Lemma \ref{lem:inflating-circular-yields-bi-free-semis} are used.  First, we expand out the product and expand the moment using linearity.  Then, for each moment term, we expand via the free cumulants and use the fact that mixed free cumulants vanish.  Cumulants involving an odd number of $c_\ell$ and $c_\ell^*$ vanish and thus we can consider only pair partitions when restricted to entries involving $c_\ell$ and $c_\ell^*$.  Any cumulant involving just $c_\ell$ or just $c_\ell^*$ vanishes and can be ignored.  By adding over all partitions with the same blocks on $c_\ell$ and $c_\ell^*$ that do not vanish yields a product of moment terms of the form $\varphi((x^*x)^m)$ and $\varphi((xx^*)^m)$.  For any such terms, viewing the $(2n)^{\mathrm{th}}$ term as the first term doesn't change the value, as the distributions of $x^*x$ and $xx^*$ are the same, thereby effectively moving the $x$ or $c_\ell$ term at the end to the beginning.  One then reverses the above process and obtains the right-hand side as desired.
\end{proof}

\begin{proof}[Proof of Theorem \ref{thm:maximizing-bi-free-entropy}.]
As per Remark \ref{rem:entropy-for-non-self-adjoints}, we may assume without loss of generality that there exists a bi-free circular pair $(c_\ell, c_r)$ (with mean 0, variance 1 and covariance 0) in $\A$ such that \[
\{(\{x,x^*\}, \{y,y^*\})\} \cup \{(\{c_\ell, c_\ell^*\}, 1)\}\cup \{(1,\{c_r,c_r^*\})\}
\]
are bi-free.   Therefore, as $\{(X, Y)\} \cup \{(S_\ell, S_r)\}$ are bi-free with respect to $\tau_2$ by Lemma \ref{lem:inflating-circular-yields-bi-free-semis}, we obtain that
\[
\chi^*(X \sqcup Y) = \ln(2\pi e) + \frac{1}{2} \int^\infty_{0} \left( \frac{2}{1+t} - \Phi^*\left( X + \sqrt{t} S_\ell \sqcup Y + \sqrt{t} S_r\right) \right) \, dt.
\]
However, as
\begin{align*}
X + \sqrt{t} S_\ell &= (x + \sqrt{t} c_\ell) \otimes E_{1,2} \otimes I_2 + (x + \sqrt{t} c_\ell)^* \otimes E_{2,1} \otimes I_2, \\
Y + \sqrt{t} S_r &= (y + \sqrt{t} c_r) \otimes I_2 \otimes  E_{1,2}  + (y + \sqrt{t} c_r)^* \otimes I_2\otimes E_{2,1}
\end{align*}
and as Lemma \ref{lem:perturbing-preserves-AAF} part (iii) shows that the assumptions of Theorem \ref{thm:minimizing-fisher-info} as satisfied, we obtain that
\begin{align}
\Phi^*\left(  \{x + \sqrt{t} c_\ell, (x + \sqrt{t} c_\ell)^*\} \sqcup \{ y + \sqrt{t} c_r,  (y + \sqrt{t} c_r)^*\}\right) \geq 2 \Phi^*\left( X + \sqrt{t} S_\ell \sqcup Y + \sqrt{t} S_r\right) \label{eq:final-in-max-entropy}
\end{align}
for all $t \in (0, \infty)$.  Hence the inequality 
\[
\chi^*(\{x,x^*\} \sqcup \{y,y^*\}) \leq 2 \chi^*(X \sqcup Y)
\]
follows by comparing the above bi-free entropy formula with that from Definition \ref{defn:entropy-for-non-self-adjoints}.

In the case that $(x, y)$ is bi-R-diagonal and alternating adjoint flipping, Lemma \ref{lem:perturbing-preserves-AAF} implies $(x + \sqrt{t} c_\ell, y + \sqrt{t} c_r)$ is bi-R-diagonal and alternating adjoint flipping for all $t \in (0, \infty)$, thus equality holds in equation (\ref{eq:final-in-max-entropy}) by Theorem \ref{thm:minimizing-fisher-info}.
\end{proof}

\section{Other Results}
\label{sec:Other}

In this section, we will examine other results from \cite{NSS1999} that generalize to the bi-free setting.  As these results are less connected to bi-free entropy with respect to a completely positive map and have proofs that can be adapted from \cite{NSS1999} using the same  modifications from Sections \ref{sec:minimize-Fisher} and \ref{sec:Max-Entropy} to deal with the $\chi$-ordering, we simply state these results.

\begin{theorem}[Generalization of \cite{NSS1999}*{Theorem 1.2}]
Let $(\A, \varphi)$ be a C$^*$-non-commutative probability space, let $d \in \bN$, and let $(A_d, E_d, \varepsilon, \tau_d)$ be as in Section \ref{sec:minimize-Fisher}.  Then
\begin{enumerate}[(i)]
\item For all $\{x_{i,j}\}^d_{i,j=1}, \{y_{i,j}\}^d_{i,j=1} \subseteq \A$, if
\[
X = \sum^d_{i,j=1} x_{i,j} \otimes E_{i,j} \otimes I_d \qqand Y = \sum^d_{i,j=1} y_{i,j} \otimes I_d \otimes E_{i,j},
\]
then
\[
\Phi^*\left( \{x_{i,j}, x^*_{i,j}\}^d_{i,j=1} \sqcup \{y_{i,j}, y_{i,j}^*\}^d_{i,j=1}\right) \geq d^3 \Phi^*(X, X^* \sqcup Y, Y^*).
\]
Moreover, equality holds if $(\{X, X^*\}, \{Y, Y^*\})$ is bi-free from $(M_d(\bC)_\ell, M_d(\bC)_r)$ with respect to $\tau_d$.
\item If in (i) $X$ and $Y$ are self-adjoint, then 
\[
\Phi^*\left( \{x_{i,j}\}^d_{i,j=1} \sqcup \{y_{i,j}\}^d_{i,j=1}\right) \geq d^3 \Phi^*(X\sqcup Y)
\]
with equality holding if $(\{X\}, \{Y\})$ is bi-free from $(M_d(\bC)_\ell, M_d(\bC)_r)$ with respect to $\tau_d$.
\end{enumerate}
\end{theorem}

\begin{theorem}[Generalization of \cite{NSS1999}*{Theorem 1.5}]\label{thm:NSS-Thm-1.5}
Let $(\A, \varphi)$ be a C$^*$-non-commutative probability space, let $d \in \bN$, and let $(A_d, E_d, \varepsilon, \tau_d)$ be as in Section \ref{sec:minimize-Fisher}.  Then
\begin{enumerate}[(i)]
\item For all $\{x_{i,j}\}^d_{i,j=1}, \{y_{i,j}\}^d_{i,j=1} \subseteq \A$, if
\[
X = \sum^d_{i,j=1} x_{i,j} \otimes E_{i,j} \otimes I_d \qqand Y = \sum^d_{i,j=1} y_{i,j} \otimes I_d \otimes E_{i,j},
\]
then
\[
\chi^*\left( \{x_{i,j}, x^*_{i,j}\}^d_{i,j=1} \sqcup \{y_{i,j}, y_{i,j}^*\}^d_{i,j=1}\right) \geq d^2\left( \chi^*(X, X^* \sqcup Y, Y^*) - 2\ln(d)\right).
\]
Moreover, equality holds if $(\{X, X^*\}, \{Y, Y^*\})$ is bi-free from $(M_d(\bC)_\ell, M_d(\bC)_r)$ with respect to $\tau_d$.
\item If in (i) $X$ and $Y$ are self-adjoint, then 
\[
\chi^*\left( \{x_{i,j}\}^d_{i,j=1} \sqcup \{y_{i,j}\}^d_{i,j=1}\right) \leq d^2\left(  \chi^*(X\sqcup Y) - \ln(d)\right)
\]
with equality holding if $(\{X\}, \{Y\})$ is bi-free from $(M_d(\bC)_\ell, M_d(\bC)_r)$ with respect to $\tau_d$.
\end{enumerate}
\end{theorem}

To prove Theorem \ref{thm:NSS-Thm-1.5}, we note it is essential to prove the following.

\begin{lemma}[Generalization of \cite{NSS1999}*{Proposition 5.3}]\label{lem:inflate-semi-and-circular-bonus}
Let $(\A, \varphi)$ be a C$^*$-non-commutative probability space, let $D_\ell, D_r$ be $*$-subalgebras of $A$, let $d \in \bN$, and let $(A_d, E_d, \varepsilon, \tau_d)$ be as in Section \ref{sec:minimize-Fisher}.   
\begin{enumerate}[(i)]
\item If $\{c_{\ell, i,j}\}^d_{i,j=1}$ and $\{c_{r, i,j}\}^d_{i,j=1}$ are circular elements of $\A$ with mean 0 and variance 1 such that
\[
\{(D_\ell, D_r)\} \cup  \{ (\{c_{\ell, i,j}, c^*_{\ell, i,j}\}, 1 )\}^d_{i,j=1}\cup  \{ (1, \{c_{r, i,j}, c^*_{r, i,j})\}^d_{i,j=1}
\]
are bi-free, then
\[
C_\ell = \sum^d_{i,j=1} c_{\ell, i,j} \otimes E_{i,j} \otimes I_d \qand C_r = \sum^d_{i,j=1} c_{r,i,j} \otimes I_d \otimes E_{i,j}
\]
are circular elements of mean 0 and variance $d$ such that $(\{C_\ell, C_\ell^*\}, 1)$, $(1, \{C_r, C^*_r\})$, and $(D_\ell \otimes M_d(\bC) \otimes I_D, D_r \otimes I_d \otimes M_d(\bC))$ are bi-free with respect to $\tau_d$.
\item If $\{s_{\ell, i,j}\}^d_{i,j=1}$ and $\{s_{r, i,j}\}^d_{i,j=1}$ are elements of $\A$ with mean 0 and variance 1 such that $s_{\ell, i,i}$ and $s_{r,i,i}$ are semicircular elements for all $i$, $s_{\ell, i,j}$ and $s_{r, i,j}$ are circular elements for all $i,j$, $s_{\ell, i,j}^* = s_{\ell, j,i}$ and $s_{r, i,j}^* = s_{r, j,i}$ for all $i,j$, and 
\[
\{(D_\ell, D_r)\} \cup  \{(s_{\ell, i,i}, 1)\}^d_{i=1} \cup \{(1, s_{r,i,i})\}^d_{i=1} \cup  \{ (\{s_{\ell, i,j}, s^*_{\ell, i,j}\}, 1 )\}_{1 \leq i < j \leq d}\cup  \{ (1, \{s_{r, i,j}, ss^*_{r, i,j})\}_{1 \leq i < j \leq d}
\]
are bi-free, then
\[
S_\ell = \sum^d_{i,j=1} s_{\ell, i,j} \otimes E_{i,j} \otimes I_d \qand S_r = \sum^d_{i,j=1} s_{r,i,j} \otimes I_d \otimes E_{i,j}
\]
are semicircular elements of mean 0 and variance $d$ such that $(S_\ell, 1)$, $(1, S_r)$, and $(D_\ell \otimes M_d(\bC) \otimes I_D, D_r \otimes I_d \otimes M_d(\bC))$ are bi-free with respect to $\tau_d$.
\end{enumerate}
\end{lemma}

Note the proof of Lemma \ref{lem:inflate-semi-and-circular-bonus} is obtained by first generalizing \cite{NSS1999}*{Lemma 5.4}, which clearly holds due to the bi-free cumulant characterization of the conjugate variables.  The proof then proceeds via constructing the appropriate conjugate variables for either $(C_\ell, C_r)$ or $(S_\ell, S_r)$ using the techniques from Section \ref{sec:minimize-Fisher}.

\end{document}